\documentclass[a4paper]{article}

\usepackage[utf8]{inputenc}
\usepackage[width=15cm,height=22cm]{geometry}
\usepackage{amsmath,amsfonts,amssymb,amsthm}
\usepackage[hidelinks,pdfusetitle]{hyperref}
\usepackage[capitalise]{cleveref}
\usepackage[dvipsnames]{pstricks}
\usepackage{bbm}

\crefname{equation}{}{}

\newtheorem{proposition}{Proposition}[section]
\newtheorem{theorem}[proposition]{Theorem}

\newtheorem{definition}[proposition]{Definition}
\newtheorem{corollary}[proposition]{Corollary}
\newtheorem{lemma}[proposition]{Lemma}
\newtheorem{remark}[proposition]{Remark}
\numberwithin{equation}{section}
\newtheorem{example}[proposition]{Example}

\newcommand*\samethanks[1][\value{footnote}]{\footnotemark[#1]}

\usepackage{etoolbox}
\newcounter{proofcounter}
\newcounter{stepcounter}[proofcounter]
\newcommand{\step}[2][]{
  \ifnum\value{stepcounter}>0 \medskip\fi
  \refstepcounter{stepcounter}
  \noindent\emph{\textbf{Step \thestepcounter.} #2}
  \ifstrempty{#1}{}{\,\label{#1}}
}
\AtBeginEnvironment{proof}{\stepcounter{proofcounter} \setcounter{stepcounter}{0}}
\crefname{stepcounter}{Step}{Steps}
\Crefname{stepcounter}{Step}{Steps}

\title{Small-time local control of a Schrödinger equation: \\ a negative and a positive quadratic result}

\author{Karine Beauchard\texorpdfstring{\thanks{Univ Rennes, ENS Rennes, INRIA, CNRS, IRMAR - UMR 6625, F-35000 Rennes, France}}{},
Fr\'ed\'eric Marbach\texorpdfstring{\thanks{DMA, École normale supérieure, Université PSL, CNRS, 75005 Paris, France}}{},
Thomas Perrin\texorpdfstring{\samethanks[1]}{}}

\newcommand{\N}{\mathbb{N}}
\newcommand{\C}{\mathbb{C}}
\newcommand{\R}{\mathbb{R}}
\newcommand{\Z}{\mathbb{Z}}
\newcommand{\dd}{\,\mathrm{d}}
\newcommand{\eps}{\varepsilon}
\newcommand{\ad}{\operatorname{ad}}
\newcommand{\pv}{\operatorname{pv}}

\newcommand{\sign}{\operatorname{sgn}}
\newcommand{\supp}{\operatorname{supp}}
\newcommand{\vect}{\operatorname{span}}
\newcommand{\sphere}{\mathbb{S}}
\newcommand{\tgt}{\mathsf{T}_{\varphi_0} \mathbb{S}}

\newcommand{\Hall}{\Theta}
\newcommand{\Hpv}{\Theta_\mathrm{pv}}

\newcommand{\Hreg}{\Theta_\mathrm{reg}}

\newcommand{\sobC}[1]{\widetilde{H}^{#1}((0,T);\C)}
\newcommand{\nsob}[2]{\left\lVert #1 \right\rVert_{\widetilde{H}^{#2}}}

\newcommand{\PH}{P_{\mathcal{H}}}

\begin{document}

\maketitle

\begin{abstract}
    We study the small-time local controllability (STLC) of a bilinear Schrödinger equation with Neumann boundary conditions near its ground state.
    We focus on the degenerate case where the linearized system is not controllable, necessitating a second-order analysis.
    We prove two complementary results.
    The negative result provides a new PDE instance of Sussmann's classical quadratic obstruction, corresponding to a non-vanishing Lie bracket.
    The positive result appears to be the first to establish STLC at the quadratic order for a physical PDE with a single scalar control.

    Both proofs rely on a Fourier-based approach, which is crucial because the integral kernel of the second-order term lacks the regularity required by standard integration-by-parts arguments.
    Along the way, we develop tools valid in a more general setting to analyze such quadratic forms.
    In particular, we prove results that allow for the multiplication of a kernel by a modulation function.
\end{abstract}

\setcounter{tocdepth}{1}
\tableofcontents

\bigskip
\bigskip
\bigskip

\noindent
\textbf{MSC classes:} 
93B05, 93C20, 35Q41


\newpage
\section{Introduction}

In this paper, we consider the following Schrödinger equation:
\begin{equation}
    \label{eq:psi}
    \begin{cases}
        i \partial_t \psi(t,x) = - \partial_x^2 \psi(t,x) - u(t) \mu(x) \psi(t,x)
        & \quad t \in (0,T), \quad x \in (0,1),
        \\
        \partial_x \psi(t,0) = \partial_x \psi(t,1) = 0
        & \quad t \in (0,T).
    \end{cases}
\end{equation}
In this system, $\mu \in H^2((0,1);\R)$ is a fixed function depending only on $x$, $u(t)$ is a real-valued control and the state $\psi(t,x)$ is a complex-valued wave function.

From the modeling perspective, the most common case is $\mu(x) = x - \frac 12$, which corresponds to a non-relativistic charged quantum particle evolving in a space-uniform electric field of time-varying amplitude $u(t)$.
The choice $\mu(x) = x - \frac 12$ also arises after a change of variables for the case of a moving potential well (or box), where $u(t)$ is the acceleration of the box (see \cite{Rouchon2003}).
In some contexts, allowing for a more general space-dependent $\mu(x)$ makes sense (see e.g.\ \cite[eq.\ (5)]{DionKellerAtabekBandrauk1999} concerning the interaction between a polarized laser and a hydrogen cyanide molecule, where~$\mu$ represents the permanent dipole moment of the molecule and $u(t)$ the intensity of the laser's electric field).

In such contexts, the Schrödinger equation is typically endowed with Dirichlet boundary conditions (see e.g.\ \cite{BeauchardLaurent2010} for a sufficient condition relying on the linear theory and \cite{Coron2006} for a first quadratic obstruction).
Here, we focus on Neumann boundary conditions, which, while perhaps less intuitive from a modeling perspective, offer an elegant toy problem for developing a general methodology for analyzing quadratic obstructions to controllability and for establishing positive quadratic results.

\subsection{Notations and well-posedness}

We consider $L^2((0,1);\C)$ equipped with its usual Hermitian inner product $\langle \phi, \psi \rangle := \int_0^1 \phi \overline{\psi}$.
We define the self-adjoint unbounded operator $A := - \Delta$ on
\begin{equation}
    \label{eq:H2N}
    \operatorname{Dom}(A) = H^2_N((0,1);\C) := \left\{ \varphi \in H^2((0,1);\C) ; \enskip \varphi'(0) = \varphi'(1) = 0 \right\}.
\end{equation}
Its eigenvalues $(\lambda_j)_{j \in \N}$ and eigenvectors $(\varphi_j)_{j \in \N}$ are respectively given by
\begin{equation}
    \label{eq:lambda_j}
    \lambda_j := (j\pi)^2
    \qquad \text{and} \qquad
    \varphi_j(x) :=
    \begin{cases}
        1 & \text{ if } j = 0, \\
        \sqrt{2} \cos (j \pi x) & \text{ if } j \geq 1.
    \end{cases}
\end{equation}
The family $(\varphi_j)_{j\in \N}$ is an orthonormal basis of $L^2((0,1);\C)$, $\varphi_0$ is the \emph{ground state}.
Let
\begin{equation}
    \sphere := \left\{ \varphi \in L^2((0,1);\C) ; \enskip \| \varphi \|_{L^2} = 1 \right\}.
\end{equation}
With these notations, system \eqref{eq:psi} is well-posed in the following sense.

\begin{proposition}
    \label{p:WP}
    Let $\mu \in H^2((0,1);\R)$.
    For any $T > 0$ and $u \in L^2((0,T);\R)$ there exists a unique mild solution $\psi \in C^0([0,T];H^2_N \cap \sphere) \cap H^1((0,T);L^2)$ to \eqref{eq:psi} with initial data $\psi(0) = \varphi_0$.
    Moreover, $\| \psi \|_{L^\infty H^2_N} \leq C_\mu(T,\|u\|_{L^2})$ where $C_\mu$ is a non-decreasing function of $T$ and $\|u\|_{L^2}$.
\end{proposition}

\begin{proof}
    This can be proved as in the Dirichlet case (see \cite[Proposition 2]{BeauchardLaurent2010}).
    The proof is based on a fixed-point argument, relying on a smoothing result (see \cref{lem:smoothing}).
\end{proof}

\subsection{Main result}

We will use the following definition, concerning the controllability of \eqref{eq:psi} near the ground state.

\begin{definition}
    \label{Def:STLC}
    Let $\mu \in H^2((0,1);\R)$.
    We say that system \eqref{eq:psi} is \emph{small-time locally controllable from $\varphi_0$} (STLC) when, for any $T > 0$ and $\eps > 0$, there exists $\delta > 0$ such that, for any $\psi^* \in H^2_N \cap \sphere$ such that $\| \psi^* - \varphi_0 \|_{H^2} \leq \delta$, there exists $u \in L^2((0,T);\R)$ with $\|u\|_{L^2} \leq \eps$ such that the associated solution to \eqref{eq:psi} with initial data $\psi(0) = \varphi_0$ satisfies $\psi(T) = \psi^*$.
\end{definition}

Adapting the strategy of \cite{BeauchardLaurent2010}, one may prove the following result.

\begin{proposition}
    \label{p:oui-STLC}
    Let $\mu \in H^2((0,1);\R)$ such that
    \begin{equation}
        \label{eq:hyp-c0}
        \exists c > 0, \quad \forall j \in \N,
        \quad |\langle \mu, \varphi_j \rangle| \geq \frac{c}{\langle j \rangle^2}.
    \end{equation}
    Then system \eqref{eq:psi} is STLC.
\end{proposition}

\begin{proof}
    Under assumption \eqref{eq:hyp-c0}, one can prove that, for any $T > 0$, the linearized perturbative version of system~\eqref{eq:psi} at $\varphi_0$ (see \eqref{eq:psi1}) is controllable with initial and final states in $H^2_N((0,1);\C) \cap \tgt$ where $\tgt := \{ \varphi \in L^2((0,1);\C) \mid \Re \langle \varphi, \varphi_0 \rangle = 0 \}$ and controls in $L^2((0,T);\R)$.
    Then, one concludes for the nonlinear system thanks to an inverse mapping theorem.
    Both steps are well understood in the Dirichlet setting (see \cite[Sections~2.3 and~2.4]{BeauchardLaurent2010}) and easily translated to our Neumann context.
\end{proof}

\begin{remark}
    Assumption \eqref{eq:hyp-c0} holds generically with respect to $\mu \in H^2((0,1);\R)$.
    Indeed, integrating by parts twice proves that, for all $j \in \N^*$,
    \begin{equation}
        \label{eq:mu-phi_j}
        \langle \mu , \varphi_j \rangle
        = \sqrt{2}\frac{(-1)^j\mu'(1)-\mu'(0)}{(j\pi)^2} - \frac{\langle \mu'',\varphi_j \rangle}{(j\pi)^2}.
    \end{equation}
    Thus, if $\mu \in H^2((0,1);\R)$ is such that $\mu'(1) \pm \mu'(0) \neq 0$ and, for all $j \in \N$, $\langle \mu,\varphi_j\rangle \neq 0$, then~\eqref{eq:hyp-c0} holds using that $\langle \mu'', \varphi_j \rangle = o(1)$ since $\mu \in H^2$.
\end{remark}

We now investigate situations when \eqref{eq:hyp-c0} fails, so that the study of the linearized system is not sufficient to determine the controllability of the nonlinear one.
Let $k \in \N$ and $\mu \in H^2((0,1);\R)$ such that $\langle \mu, \varphi_k \rangle = 0$.
We define the following family of coefficients $(c_j)_{j \in \N} \in \R^\N$ and $a_k \in \R$ by
\begin{equation}
    \label{def:cj_a}
    \forall j \in \N, \quad
    c_j := \langle \mu , \varphi_j \rangle \langle \varphi_j , \mu \varphi_k \rangle
    \qquad \text{and} \qquad
    a_k :=\sum_{j\in\N}\left( \lambda_j-\frac{\lambda_k}{2} \right) c_j.
\end{equation}
We refer to \cref{rk:ak-Lie} and \cref{sec:modulation} for the interpretation of the coefficient $a_k$.
Note that the series defining $a_k$ converges absolutely because \eqref{eq:mu-phi_j} (applied with $\mu$ and $\mu \leftarrow \mu \varphi_k$) implies
\begin{equation}
    \label{cj_asympt}
    c_j= \underset{j \to + \infty}{O}(j^{-4}).
\end{equation}
An accessible version of our main results is the following one (which provides a dichotomy within the class of functions $\mu$ that satisfy \eqref{eq:hyp-c0-k}).

\begin{theorem}
    \label{thm:main}
    Let $k\in\N$ and $\mu \in H^2((0,1);\R)$ such that $\langle \mu, \varphi_k \rangle = 0$.
    \begin{enumerate}
        \item \label{it:thm-1}
        Assume that $a_k \neq 0$. Then system \eqref{eq:psi} is not STLC.
        \item \label{it:thm-2}
        Assume that $a_k = 0$ and
        \begin{equation}
            \label{eq:hyp-c0-k}
            \exists c > 0, \quad \forall j \in \N \setminus \{k\},
            \quad |\langle \mu, \varphi_j \rangle| \geq \frac{c}{\langle j \rangle^2}.
        \end{equation}
        Then system \eqref{eq:psi} is STLC.
    \end{enumerate}
\end{theorem}

\begin{remark}
    \label{rk:k=0}
    When $k = 0$, if $\mu \in H^2((0,1);\R)$ and $\langle \mu, \varphi_0 \rangle = 0$, then system \eqref{eq:psi} is not STLC.
    Indeed, either $\mu = 0$, or $a_0 =\sum_{j\in\N} \lambda_j | \langle \mu , \varphi_j \rangle|^2 >0$ so \cref{it:thm-1} of \cref{thm:main} applies.
    In particular, system \eqref{eq:psi} with $\mu(x) = x - \frac 12$ is not STLC.
\end{remark}

Both statements of \cref{thm:main} are of quadratic nature.
To highlight this, let us state the following sharper versions which we will prove in this work.
In these statements, we denote by $u_1(t) := \int_0^t u$ the primitive of the control.

\begin{theorem}
    \label{thm:negative}
    Let $k\in\N$ and $\mu \in H^2((0,1);\R)$ such that $\langle \mu, \varphi_k \rangle = 0$.
    Assume that $a_k \neq 0$.
    There exist $T^*, C, \nu > 0$ such that, for all $T \in (0,T^*)$, and all $u \in L^2((0,T);\R) $ with $\|u\|_{L^2} \leq 1$, the associated solution to \eqref{eq:psi} with initial data $\psi(0) = \varphi_0$ satisfies
    \begin{equation}
        \label{drift}
        \left| \langle \psi(T) - \varphi_0, \varphi_k \rangle + i a_k \|u_1\|_{L^2}^2 \right| \leq C T^{\nu} \|u_1\|_{L^2}^2 +
        C \|\psi(T)-\varphi_0\|_{L^2}^2.
    \end{equation}
\end{theorem}

\begin{remark}
    \label{rq:imp_motion}
    In particular, for small enough times (e.g.\ such that $C T^\nu \leq |a_k|$), estimate \eqref{drift} prevents $\psi(T)$ from reaching targets of the form
    \begin{equation}
        \psi^*_\varepsilon := \sqrt{1-\varepsilon^2} \varphi_0 + i \sign(a_k)\eps \varphi_k
    \end{equation}
    with $0<\eps\ll1$ because this would entail $\eps \leq 2 C \eps^2$.
\end{remark}

\begin{theorem}
    \label{thm:positive}
    Let $k\in\N^*$ and $\mu \in H^2((0,1);\R)$ such that $\langle \mu, \varphi_k \rangle = 0$.
    Assume that $a_k=0$ and that \eqref{eq:hyp-c0-k} holds.
    There exists $T^* > 0$ such that, for every $T \in (0,T^*)$, there exists $\delta >0$ such that, for every $\psi^* \in H^2_N((0,1);\C) \cap \sphere$ with $\|\psi^*-\varphi_0\|_{H^2} \leq \delta$, there exists $u \in L^2((0,T);\R)$ such that the associated solution to \eqref{eq:psi} with initial data $\psi(0) = \varphi_0$ satisfies $\psi(T) = \psi^*$.
    
    Moreover, we have the estimates
    \begin{align}
        \label{eq:cost-L2}
        \|u\|_{L^2} & \leq C_T |\langle \psi^*,\varphi_k \rangle|^{\frac 12} + C_T \| \psi^*-\varphi_0\|_{H^2_N}, \\
        \label{eq:cost-H-1}
        \|u_1\|_{L^2} & \leq C
        \left(\frac{|\Re \langle \psi^*,\varphi_k \rangle|}{T^2} \right)^{\frac 12} + C
        \left( \frac{|\Im \langle \psi^*,\varphi_k \rangle |}{T} \right)^{\frac 12} + C_T \| \psi^* - \varphi_0 \|_{L^2},
    \end{align}
    where $C$ denotes a time-independent constant and $C_T$ a time-dependent one.
\end{theorem}

\begin{remark}
    The cost estimate \eqref{eq:cost-L2} involving $|\langle \psi^*,\varphi_k \rangle|^{\frac 12}$ is typical of results where the controllability stems from the quadratic order (see e.g.\ \cite[(1.34) in Section 1.6]{BeauchardMarbach2018} for a parabolic system in small time, or \cite[Theorem 1.3]{Nguyen2025} for a Korteweg--de Vries system in large time).

    When $a_k \neq 0$, \eqref{drift} states that the amplitude of the unwanted motion in the direction $- i a_k \varphi_k$ is given by $\|u_1\|_{L^2}^2$.
    The sharp cost estimate \eqref{eq:cost-H-1} proves that, when $a_k = 0$, one can move in the directions $\pm i \varphi_k$ with an amplitude scaling like $T \| u_1 \|_{L^2}^2$, and in the directions $\pm \varphi_k$ with an amplitude $T^2 \|u_1\|_{L^2}^2$ (motions in the other directions are governed by the linear theory).
\end{remark}

\subsection{Heuristic and comparison with previous results}
\label{sec:heuristic}

\subsubsection{An ODE prototype, drifts and Lie brackets}
\label{sec:prototype}

The root cause of \cref{thm:negative} (\cref{it:thm-1} of \cref{thm:main}) is embedded in the following caricatural finite-dimensional control-system example:
\begin{equation}
    \label{eq:x-W1}
    \begin{cases}
        \dot{x}_1 = u \\
        \dot{x}_2 = x_1 \\
        \dot{x}_3 = x_1^2 - x_2^2 - x_1^3 - 2 u x_1.
    \end{cases}
\end{equation}
Let $u_1(t) := \int_0^t u$ be the primitive of the control.
In the sequel, this function will implicitly be considered as a function defined on $(0,T)$, and $\|u_1\|_{L^2}$ will thus denote $\|u_1\|_{L^2(0,T)}$.
Explicit integration from $x(0)=0$ yields
\begin{equation}
    \label{eq:x-W1-drift}
    \begin{split}
        x_3(T) & = \int_0^T u_1^2 - \int_0^T \left(\int_0^t u_1\right)^2 \dd t - \int_0^T u_1^3 - (u_1(T))^2 \\
        & \geq \left(1-T^2-\sqrt{T}\|u\|_{L^2}\right) \|u_1\|_{L^2}^2 - (u_1(T))^2 \\
        & \geq c_0 \| u_1 \|_{L^2}^2 - (u_1(T))^2
    \end{split}
\end{equation}
for any $c_0 \in (0,1)$, for small enough times and controls.
Since $x_1(T) = u_1(T)$, all final states in the ``half-space'' $\{ x \in \R^3 ; \enskip x_3 + x_1^2 < 0 \}$ are not reachable, so that system \eqref{eq:x-W1} is not STLC at~$0$.
In particular, one cannot reach any target of the form $(0,0,-\delta)$ for $\delta > 0$.

This example is a particular case of the following more general result, proved by Sussmann in \cite[Proposition 6.3]{Sussmann1983} (see also \cite[Theorem 1.10]{BeauchardMarbach2022} for a modern proof with sharper hypotheses).

\begin{proposition}
    \label{p:sussmann}
    Let $f_0, f_1$ be real-analytic vector fields on $\R^d$ with $f_0(0) = 0$.
    Consider the control system $\dot{x} = f_0 + u f_1$.
    Let $S_1 \subset \R^d$ denote the controllable vector subspace for the linearized system at $0$.
    If $[f_1,[f_1,f_0]](0) \notin S_1$, then the nonlinear system is not STLC at $0$.
\end{proposition}

We refer to the situation of \eqref{eq:x-W1-drift} (or more generally of \cref{p:sussmann}) as a ``quadratic obstruction to controllability'', involving a ``drift'' quantified by the (squared) $H^{-1}$ Sobolev norm of the control.
For ODEs, the quadratic obstructions to controllability have been classified in \cite{BeauchardMarbach2018,BeauchardMarbach2022} and are linked with the $H^{-k}$ norms of the control for $1 \leq k \leq \dim S_1$, and with iterated Lie brackets of length $2k+1$.

These quadratic obstructions of integer order are proved on the equation \eqref{eq:psi} in \cref{sec:drift_integer}, for functions $\mu$ that appropriately vanish at the boundary (in particular $\mu'(0)=\mu'(1)=0$, which is incompatible with \eqref{eq:hyp-c0-k}, see \eqref{eq:mu-phi_j}). They have also been observed for other PDEs (see \cite{BeauchardMorancey2014,Bournissou2023_Quad,Coron2006} for Schrödinger with Dirichlet boundary conditions, \cite{BeauchardMarbach2020} for a parabolic model, \cite{CoronKoenigNguyen2024} for a fluid mechanics model, and \cite{NiuXiang2025} for a Korteweg--de Vries system).

The proof of \cref{thm:negative} involves the same main terms and arguments as in \eqref{eq:x-W1-drift}:
\begin{itemize}
    \item obtaining a coercivity estimate $\geq \| u_1 \|_{L^2}^2$ for the quadratic terms when $T \ll 1$,
    \item obtaining a bound for the cubic terms of the form $o(\| u_1 \|_{L^2}^2)$ when $\|u\|_{L^2} \leq 1$,
    \item a geometric description of unreachable targets to handle terms involving $u_1(T)$.
\end{itemize}

\begin{remark}
    \label{rk:ak-Lie}
    Under appropriate regularity assumptions on $\mu$, the assumptions $\langle \mu, \varphi_k \rangle = 0$ and $a_k \neq 0$ of \cref{thm:negative} precisely correspond to the Lie bracket condition of \cref{p:sussmann}.
    
    For instance, if $\mu \in H^2((0,1);\R)$ and $\mu'(0)=\mu'(1)=0$ (i.e.\ $\mu \in H^2_N$), then
    \begin{equation}
        \begin{split}
            2 a_k & =
            \sum_{j\in\N} \big( \lambda_j + (\lambda_j-\lambda_k) \big) \langle \mu ,\varphi_j\rangle \langle \varphi_j,\mu\varphi_k\rangle
            \\ & =
            \sum_{j\in\N} \big( \langle \mu ,A \varphi_j\rangle \langle \varphi_j,\mu\varphi_k\rangle +
            \langle \mu ,\varphi_j\rangle \langle (A-\lambda_k) \varphi_j,\mu\varphi_k\rangle \big)
            \\ & =
            \langle A \mu ,\mu\varphi_k\rangle +
            \langle \mu ,(A-\lambda_k)(\mu\varphi_k)\rangle
            = - \langle [\mu,[\mu,A]]\varphi_0 , \varphi_k \rangle,
        \end{split}
    \end{equation}
    where the third line results from integrations by parts using $\mu'(0)=\mu'(1)=0$.
    
    The assumption $\langle \mu , \varphi_k \rangle = 0$ implies that the controllable space~$S_1$ of the linearized system~\eqref{eq:psi1} is contained in $\vect \{ \varphi_j \mid j \neq k \}$.
    The assumption $a_k \neq 0$ corresponds to $[\mu,[\mu,A]]\varphi_0 \notin S_1$.
\end{remark}

\subsubsection{Power-series expansion for the Schrödinger equation}
\label{sec:power-series}

As is usual for such results, we perform a power series expansion of the state with respect to the control, writing $\psi = \varphi_0 + \psi_1 + \psi_2 + \dotsc$ where $\varphi_0$ is the equilibrium around which we are studying controllability, $\psi_1$ is the solution to the linearized system
\begin{equation}
    \label{eq:psi1}
    \begin{cases}
        i \partial_t \psi_1(t,x) = - \partial_x^2 \psi_1(t,x) - u(t) \mu(x) \varphi_0(x)
        & \quad t \in (0,T), \quad x \in (0,1),
        \\
        \partial_x \psi_1(t,0) = \partial_x \psi_1(t,1) = 0
        & \quad t \in (0,T),
    \end{cases}
\end{equation}
and $\psi_2$ is the quadratic order expansion, solution to
\begin{equation}
    \label{eq:psi2}
    \begin{cases}
        i \partial_t \psi_2(t,x) = - \partial_x^2 \psi_2(t,x) - u(t) \mu(x) \psi_1(t,x)
        & \quad t \in (0,T), \quad x \in (0,1),
        \\
        \partial_x \psi_2(t,0) = \partial_x \psi_2(t,1) = 0
        & \quad t \in (0,T),
    \end{cases}
\end{equation}
with initial data $\psi_1(0) = \psi_2(0) = 0$ when $\psi(0) = \varphi_0$.

Applying the Duhamel formula to \eqref{eq:psi1} and \eqref{eq:psi2} yields
\begin{align}
    \label{linearise_explicit}
    \psi_1(t) & = i \int_0^t u(s) e^{-i A(t-s)} \mu \varphi_0 \dd s = i \sum_{j \in \N} \langle \mu , \varphi_j \rangle \left( \int_0^t u(s) e^{-i \lambda_j (t-s)} \dd s \right) \varphi_j,
    \\ \label{eq:psi2-phik}
    \langle \psi_2(T), \varphi_k \rangle
    & = i \int_0^T u(t) e^{-i\lambda_k(T-t)} \langle \mu \psi_1(t), \varphi_k \rangle \dd t.
\end{align}
Using the coefficients $c_j$ of \eqref{def:cj_a}, and $\lambda_0 = 0$, one has when $k = 0$ (see \cref{sec:modulation} for $k > 0$),
\begin{equation}
    \label{eq:psi2-eq-half-Q0}
    \begin{split}
        \langle \psi_2(T), \varphi_0 \rangle
        & = - \sum_{j \in \N} c_j \int_0^T u(t) \int_0^t u(s) e^{-i \lambda_j(t-s)} \dd s \dd t \\
        & = - \iint_{0 < s < t < T} \Big( \sum_{j \in \N} c_j e^{-i \lambda_j |t-s|} \Big) u(s) u(t) \dd s \dd t.
    \end{split}
\end{equation}
Using symmetry to extend the integral to the square $(s,t) \in (0,T)^2$, and using that the control is real-valued, we obtain $\langle \psi_2(T), \varphi_0 \rangle = - \frac 1 2 Q(u,u)$ where we define, for $u, v \in L^2((0,T);\C)$,
\begin{equation}
    \label{eq:Q_0}
    Q(u,v)
    := \int_0^T \int_0^T K(t-s) u(s) \bar{v}(t) \dd s \dd t
    = \langle K \star u, v \rangle_{L^2}
\end{equation}
where
\begin{equation}
    \label{eq:K_0}
    K(\sigma) := \sum_{j \in \N} c_j e^{-i\lambda_j|\sigma|}.
\end{equation}
By \eqref{cj_asympt} this series converges absolutely and $K \in L^{\infty}(\R;\C)$.
One must then study the complex valued bilinear form \eqref{eq:Q_0} to attempt to exhibit a drift generalizing the form \eqref{eq:x-W1-drift}.

\subsubsection{Usual approach using integrations by parts}
\label{sec:no-ipp}

When faced with a quadratic form such as \eqref{eq:Q_0}, the usual approach to isolate the $H^{-1}$ drift term is to perform two integrations by parts to force the presence of $u_1(t) = \int_0^t u$ instead of $u$.
For instance, assuming that $K \in C^0(\R;\C)$ enjoys sufficient regularity\footnote{The equality and estimate given in \eqref{eq:Q-IPP} for example hold when $K \in C^0(\R;\C)$ is such that $K \in H^2(-T,0) \cap H^2(0,T)$, but not necessarily $H^2(-T,T)$ so that one can have $K'(0^+) \neq K'(0^-)$.
See \cref{lem:IPP-Armand} for precise assumptions and detailed integrations by parts, also valid for non-convolution kernels.} on $\R \setminus \{0\}$, one gets
\begin{equation}
    \label{eq:Q-IPP}
    \begin{split}
        Q(u,u) & = (K'(0^-)-K'(0^+)) \int_0^T u_1^2 + K(0) (u_1(T))^2 - \int_0^T \int_0^T K''(t-s) u_1(t) u_1(s) \dd s \dd t \\
        & \qquad \qquad \qquad + u_1(T) \int_0^T u_1(t) (K'(T-t) - K'(t-T)) \dd t \\
        & = (K'(0^-)-K'(0^+)) \|u_1\|_{L^2}^2 + O \left( |u_1(T)|^2 + T \|u_1\|_{L^2}^2 \right).
    \end{split}
\end{equation}
In particular, if $(K'(0^-)-K'(0^+)) \neq 0$, when $T \ll 1$, one obtains a drift quantified by $\|u_1\|_{L^2}^2$ (up to terms involving $u_1(T)$, which can be handled as in example \eqref{eq:x-W1}).
This key step is ubiquitous in the proof of such results linked with integer-order quadratic obstructions: see \cref{subsec:coer_app} for the system \eqref{eq:psi}, \cite[Lemma 3.1]{BeauchardMorancey2014} and \cite[Proposition 5.1]{Bournissou2023_Quad} for the Schrödinger equation with Dirichlet boundary conditions, \cite[Proposition 3.3]{BeauchardMarbach2020} for a parabolic system, \cite[Lemma 3.8]{CoronKoenigNguyen2024} for a water-tank system.

If $\mu$ has an infinite number of non-vanishing Fourier coefficients $\langle \mu,\varphi_j\rangle$, one cannot differentiate twice the kernel $K$ defined in \eqref{eq:K_0}.
Indeed, taking into account \eqref{eq:lambda_j} and \eqref{cj_asympt}, the derivative $K'$ of this kernel
is heuristically similar to the Riemann function $\sigma \mapsto \sum_{j=1}^{\infty} j^{-2} \sin (j^2 \sigma)$, which is differentiable almost nowhere (see \cite[Theorem 4.31]{Hardy1916} and \cite{Gerver1970,Gerver1971}).

Thus, one must look for a method avoiding two integrations by parts to highlight the drift.
The method we therefore develop could also be used to weaken the assumptions of existing results (for example, \cite[Theorem 1.3, case $n=1$]{Bournissou2023_Quad} assumes a condition of the form $\sum j^{4} |c_j| < \infty$, while our strategy could probably be adapted to conclude when one only has $\sum j^2 |c_j| < \infty$).

\subsubsection{Fourier-based approach for quadratic obstructions}

To circumvent this difficulty, we will avoid integrations by parts and perform our estimates in the Fourier frequency domain.
Indeed, benefiting from the convolution structure of \eqref{eq:Q_0}, heuristically, one expects to be able to obtain from Plancherel's theorem that
\begin{equation}
    \label{eq:Q-F(K)-intro}
    Q(u,u) = \frac{1}{2 \pi} \int_\R \widehat{K}(\omega) |\widehat{u}(\omega)|^2 \dd \omega.
\end{equation}
Then, one can derive properties on the coercivity of $Q$ by studying the asymptotic behavior of $\widehat{K}$.
Formally, one would for example expect that, if $\widehat{K}(\omega) \sim \omega^{-2}$ as $|\omega| \to \infty$, one would obtain an $H^{-1}$ Sobolev norm.
This is for example what happens for some critical lengths of the Korteweg--de Vries system (see \cite[Lemma 4.2]{NiuXiang2025}).

Such Fourier-based approaches to quadratic obstructions, or variations thereof, have been used in \cite{BeauchardMarbach2020} to obtain $H^{-s}$ drifts for parabolic systems, in \cite{CoronKoenigNguyen2022} and \cite{Nguyen2025} to obtain $H^{-2/3}$ and $H^{-1/6}$ drifts for Korteweg--de Vries systems, and in \cite{NguyenTran2025} to obtain an $H^{-5/4}$ drift in finite time for a Burgers system.
For this last system, the $H^{-5/4}$ drift had already been established in \cite{Marbach2018} using an analysis in the time domain relying on weakly singular integral operators.

In this paper, we encounter difficulties linked with the fact that the most natural expression of $\widehat{K}$ is only a distribution, which involves Dirac masses and principal values (see \cref{sec:F(K)}) so that the asymptotic $\widehat{K}(\omega) \sim \omega^{-2}$ is not valid near these singularities.
Nevertheless, we show that these local perturbations can be neglected as $T \to 0$.

\subsubsection{An intuition of the positive result}
\label{sec:positive-heuristic}

When $a_k = 0$, the key difficulty lies in proving that it is possible to move approximately in the directions $\pm i \varphi_k$.
Hence the core of the proof of \cref{thm:positive} consists in constructing, for any $T \in (0,1]$, controls $u^{\pm i} \in L^2((0,T);\R)$ such that $Q(u^{\pm i}, u^{\pm i}) \approx \pm i T \|u^{\pm i}_1\|_{L^2}^2$.
From \eqref{eq:Q-F(K)-intro}, for this to be possible, we need $\Im \widehat{K}$ to change sign.
Moreover, for $T \ll 1$, since the control is supported in $[0,T]$, its Fourier transform will involve high frequencies and have a wide support.
So $\Im \widehat{K}$ not only needs to change sign, but it needs to do so up to $\pm \infty$, and achieve both signs on arbitrarily large intervals.
Heuristically, to achieve a motion of amplitude $T \|u_1\|_{L^2}^2$, we will need $\pm \Im \widehat{K}(\omega) \geq T \omega^{-2}$ on intervals much larger than $1/T$.
One can then choose a control oscillating at the median frequency of such an interval, say $\omega_0$, to achieve the desired sign.

Such considerations were exactly what motivated the construction in \cite[Section 5]{BeauchardMarbach2018} where the first two authors designed an abstract nonlinearity such that $\widehat{K}(\omega) \approx (\sin (\ln (1+|\omega|))) / \omega^2$ which precisely takes positive and negative values on increasingly larger intervals as $|\omega| \to \infty$.
Here, we prove in \cref{sec:approximate} that the same situation occurs -- although this time, the nonlinearity is more physically reasonable, as it is just given by the bilinearity in the Schrödinger model.

The controls $u^{\pm i}$ must also leave the linearized system invariant, i.e.\ they must guarantee that $\psi_1(T) = 0$ (recall \eqref{linearise_explicit}).
Indeed, otherwise, for small controls (which are required by \cref{Def:STLC} of STLC), the linear movement would dominate the quadratic one, so we would not be achieving (approximate) motions in the directions $\pm i \varphi_k$.
More formally, the controls must be chosen within the kernel of the map $u \mapsto \psi_1(T)$.
This corresponds to the requirement that $\widehat{u}(\lambda_j) = 0$ for all $j \in \N$.
Our construction ensures this by choosing $\omega_0$ far from the $\lambda_j$.

Once motions in the directions $\pm i \varphi_k$ are established, those in the directions $\pm \varphi_k$ follow from a standard tangent vector argument (see \cref{p:mvt-Im-to-Re}).
For control-affine systems in finite dimension $\dot{x} = f_0 + u f_1$ with $f_0(0) = 0$, it is well-known that if $\pm \xi$ are tangent vectors (i.e.\ one can move infinitesimally in these directions for small times), then the same holds for $\pm Df_0(0) \xi$ (see \cite[Theorem 6]{HermesKawski1987} or \cite[Property (P2)]{BianchiniStefani1986}).
This argument has been adapted to several PDEs (see \cite[Section 6.4]{Bournissou2024} or \cite[Proposition 4.8]{Gherdaoui2025} for Schrödinger with Dirichlet boundary conditions, or \cite[Section 4.3]{ChowdhuryErvedoza2019} for Navier--Stokes).
Here $- i A$ plays the role of $f_0$ and $- i A (\pm i \varphi_k) = \mp \lambda_k \varphi_k$.

\subsubsection{Contributions}

The contributions of the present work to the control theory literature are the following:
\begin{itemize}
    \item Exhibit in \cref{thm:negative} yet another instance of Sussmann's historical quadratic ODE obstruction of \cref{p:sussmann} in the context of PDEs (after \cite{Coron2006} for Schrödinger with Dirichlet boundary conditions, \cite{BeauchardMarbach2020} for a parabolic model, \cite{CoronKoenigNguyen2024} for a fluid mechanics model, and \cite{NiuXiang2025} for a Korteweg--de Vries system).
    \item Establish in \cref{thm:positive} the first positive scalar-control small-time quadratic controllability result for a reasonable physical system (the previous examples of \cite[Theorems 5 and 6]{BeauchardMarbach2018} relied on cooked-up academic nonlinearities).
    \item Give a general method to transfer estimates obtained on a quadratic form given by a kernel $K$ to one given by a kernel $\chi K$ (see \cref{sec:multiply}), which we expect to be re-usable to simplify the study of other systems.
    \item Derive energy estimates for the Schrödinger equation with inhomogeneous Neumann boundary conditions in \cref{sec:remainder}.
\end{itemize}

\subsection{Plan of the paper}

The cornerstone of this work is \cref{sec:quad}, where we express the reference quadratic form $Q$ in the Fourier domain, and analyze its properties.
We use this to prove \cref{thm:negative} in \cref{sec:negative} and \cref{thm:positive} in \cref{sec:positive}.
Both of these sections rely on a general methodology derived in \cref{sec:multiply} which allows to estimate the difference between the reference quadratic form $Q$ (corresponding to $k=0$) and the modulated one $Q_k$ (see \cref{sec:modulation}).
They also rely on an estimate of the cubic remainder in the state's expansion, derived in \cref{sec:remainder}, using an auxiliary system.
Finally, appendices gather the proof of higher-order obstructions (\cref{sec:drift_integer}) and useful well-known or folklore material (\cref{sec:appendix}).

\subsection{Notations}
\label{sec:notations}

Throughout this text, the symbol $\lesssim$ will denote inequalities which hold up to a constant that may change from line to line but which is independent of $T$ and $u$.

For $\nu \in \R$, we endow the Sobolev space $H^\nu(\R;\C)$ with its usual Hilbert norm
\begin{equation}
    \| u \|_{H^\nu(\R)} := \left( \frac{1}{2 \pi} \int_{\R} \langle \omega \rangle^{2\nu} |\widehat{u}(\omega)|^2 \dd\omega\right)^{\frac{1}{2}},
\end{equation}
where we use the Japanese bracket notation $\langle \omega \rangle:=(1+\omega^2)^{\frac 12}$ for $\omega \in \R$ and
\begin{equation}
    \label{eq:hat-u}
    \widehat{u}(\omega) := \int_\R u(t) e^{-i\omega t} \dd t.
\end{equation}
In the sequel, we implicitly identify functions defined on $(0,T)$ with their extension by $0$ to $\R$.
In particular, Fourier transforms always denote the one defined by \eqref{eq:hat-u} after extension by zero, and similarly for convolutions.
For controls defined on $(0,T)$, we will use the following Sobolev space.

\begin{definition}
    \label{def:Htilde-nu}
    Let $T > 0$ and $\nu \in \R$.
    Let $\sobC{\nu}$ denote the closure (for the norm $\| \cdot \|_{H^\nu(\R)}$) of $C^\infty_c((0,T);\C)$ in $H^\nu(\R;\C)$.
    For $u \in L^2((0,T);\C) \cap \sobC{\nu}$, there holds
    \begin{equation}
        \label{eq:Hnu}
        \nsob{u}{\nu} = \| \widetilde{u} \|_{H^\nu(\R)}
    \end{equation}
    where $\widetilde{u}$ denotes the extension of $u$ by $0$ on $\R \setminus (0,T)$.
\end{definition}

This space (which differs from $H^\nu_0((0,T);\C)$ for some values of $\nu$) is the natural Sobolev space for our computations involving Fourier transforms, and is for example defined and studied in \cite[Section 3]{McLean2000}, \cite[Section 1.3]{Grisvard1985}, \cite[Section 4.3]{ChandlerWildeHewettMoiola2015}.

In particular, for $\nu > 0$ such that $\nu \notin \frac{1}{2} + \N$, $\sobC{\nu} = H^\nu_0((0,T);\C)$ (see \cite[Corollary 1.4.4.5]{Grisvard1985}) and for all $\nu > 0$, $\sobC{-\nu}$ is the dual of $H^\nu((0,T);\C)$ (see \cite[Theorem 3.30]{McLean2000}).

\subsection{Frequency modulation of the reference quadratic form}
\label{sec:modulation}

In \cref{sec:power-series}, to avoid burdening the heuristic, we only presented the computation of $\langle \psi_2(T), \varphi_k \rangle$ in the case $k = 0$, leading to the reference quadratic form $Q$ of \eqref{eq:Q_0} defined by its reference kernel~$K$ of \eqref{eq:K_0}.
This reference case, which is sufficient to understand the core issues, is the subject of \cref{sec:quad}.
The general case will then stem from it.

Let us explain what changes for $k \neq 0$.
Starting again from \eqref{linearise_explicit} and \eqref{eq:psi2-phik}, we obtain, using the coefficients $c_j$ defined by \eqref{def:cj_a} and the function $\rho_k$ defined by $\rho_k(t):=e^{i \lambda_k t / 2}$,
\begin{equation}
    \label{eq:psi2=Qk}
    \begin{split}
        \langle \psi_2(T), \varphi_k e^{-i\lambda_k T} \rangle
        & = - \sum_{j \in \N} c_j \int_0^T u(t) e^{i\lambda_k t} \int_0^t u(s) e^{-i \lambda_j(t-s)} \dd s \dd t
        \\ & = - \iint_{0 < s < t < T} \Big( \sum_{j \in \N} c_j e^{i \frac{\lambda_k}{2}|t-s|} e^{-i \lambda_j |t-s|} \Big) (u\rho_k)(t) (u\rho_k)(s) \dd s \dd t
        \\ & = - \frac{1}{2} Q_k(u \rho_k, u \overline{\rho_k}),
    \end{split}
\end{equation}
where, as after \eqref{eq:psi2-eq-half-Q0}, we use that $u(t) \in \R$, symmetry, and we define, for $u, v \in L^2((0,T);\C)$,
\begin{align}
    \label{eq:Q_k}
    Q_k(u,v) & := \int_0^T \int_0^T K_k(t-s) u(s) \bar{v}(t) \dd s \dd t, \\
    \label{eq:K_k}
    K_k(\sigma) & := \exp \left(i \frac{\lambda_k}{2} |\sigma|\right) K(\sigma).
\end{align}
The frequency modulation in the arguments of $Q_k$ (i.e.\ the fact that $Q_k$ is evaluated at $(u \rho_k, u \overline{\rho_k})$ instead of $(u,u)$ in \eqref{eq:psi2=Qk}) is immaterial (see \cref{lem:Q_k-conj-best}).

The frequency modulation in the kernel $K_k$ with respect to $K$ is more interesting.
We dedicate \cref{sec:multiply} to the general abstract topic of transferring estimates obtained on a quadratic form given by a reference kernel $K$ to estimates on a modulated quadratic form associated with $\chi K$ for a given multiplier~$\chi$.
We use this abstract formalism to estimate $Q_k(u,v) - Q(u,v)$ in \cref{p:Qk-Q}.

In particular, this modulation explains the formula for the coefficient $a_k$ of \eqref{def:cj_a}.
Indeed, we learned from the heuristic \eqref{eq:Q-IPP} that it is the jump of the derivative of the kernel at $0$ which governs the $H^{-1}$ drift of a system.
We check that
\begin{equation}
    K_k'(0^-)-K_k'(0^+)
    = (K'(0^-)-K'(0^+)) - i \lambda_k K(0)
    = i \sum_{j \in \N} (2 \lambda_j - \lambda_k) c_j
    = 2 i a_k.
\end{equation}
In particular, we can write
\begin{equation}
    \label{eq:a-ak}
    a_k = a - \frac{\lambda_k}{2} K(0) \quad \text{where} \quad
    a := \sum_{j \in \N} \lambda_j c_j
\end{equation}
and $a$ is the coefficient driving the $H^{-1}$ coercivity of $Q$, while $a_k$ is the one for $Q_k$.

Recall that $a_k$ also has an interpretation in terms of Lie brackets given in \cref{rk:ak-Lie}.

\section{Study of the quadratic form}
\label{sec:quad}

In this section, we study the reference quadratic form $Q(u,v)$ defined in \eqref{eq:Q_0} from the convolution kernel $K$ of \eqref{eq:K_0}.
We start by computing the Fourier transform of $K$, to express $Q$ in the frequency domain (see \cref{Prop:FQenFourier}).
We then decompose this Fourier transform into several parts, and analyze their properties in preparation for the subsequent sections.
Unless otherwise stated, all functions and distributions in this section are complex-valued.

\subsection{Expression of the kernel in the frequency domain}
\label{sec:F(K)}

The function $K$ (see \eqref{eq:K_0}, \eqref{def:cj_a}, \eqref{eq:lambda_j}) does not belong to $L^1(\R)$, but defines a tempered distribution because $K \in L^{\infty}(\R) \subset \mathcal{S}'(\R)$.
Thus its Fourier transform $\mathcal{F}(K)$ is well defined in $\mathcal{S}'(\R)$.
The goal of this section is to prove the following explicit expression of $\mathcal{F}(K)$:
\begin{equation}
    \mathcal{F}(K) =
    \pi \sum_{j\in\N} c_j (\delta_{\lambda_j}+\delta_{-\lambda_j})-2i \sum_{j\in\N}\frac{ \lambda_j c_j }{\lambda_j^2-\omega^2},
\end{equation}
where the right-hand side contains a sum of Dirac masses and principal value distributions (see \cref{Prop:Kchapeau} for a precise statement).

\begin{remark}
    The fact that $\mathcal{F}(K)$ is merely a distribution and not an integrable function comes from the lack of decay of $K$ at $\pm \infty$.
    However, the definition of $Q$ in \eqref{eq:Q_0} only involves the values of $K(\sigma)$ for $\sigma \in [-T,T]$.
    Hence, one would keep the same quadratic form by replacing $K$ with $K_T := K \cdot \mathbbm{1}_{[-T,T]} \in L^2(\R)$, which thus has an $L^2(\R)$ Fourier transform.
    However, up to the necessary precautions implied by working with distributions, we find that working with $K$ (and not $K_T$) allows a better comprehension of the different contributions composing the kernel, and more explicit computations.
    These are particularly useful to obtain the positive result in \cref{sec:positive}.
\end{remark}

We need the following definitions and preliminary results.

\begin{definition}[Dilation, translation, Fourier transform]
    \label{Def:Distrib}
    Let $\alpha \in \R$ and $\lambda > 0$.
    For a function $f \in L^1(\R)$, we denote by $D_{\lambda}f$, $\tau_{\alpha} f$ and $\widehat{f}$ the functions defined by
    \begin{equation}
        D_{\lambda} f (t) := \sqrt{\lambda}f(\lambda t), \qquad
        \tau_{\alpha} f (t):=f(t-\alpha), \qquad
        \widehat{f}(\omega):=\int_{\R} f(t) e^{-i\omega t} \dd t.
    \end{equation}
    For a tempered distribution $S \in \mathcal{S}'(\R)$, we denote by $D_{\lambda} S$, $\tau_{\alpha} S$ and $\mathcal{F}(S)$ the tempered distributions defined by their action on Schwartz functions as
    \begin{equation}
        \label{eq:distributions-1}
        \langle D_{\lambda} S , \phi \rangle_{\mathcal{S'},\mathcal{S}}
        :=
        \langle S , D_{\lambda^{-1}} \phi \rangle_{\mathcal{S'},\mathcal{S}},
        \enskip
        \langle \tau_{\alpha} S , \phi \rangle_{\mathcal{S'},\mathcal{S}}
        :=
        \langle S , \tau_{-\alpha} \phi \rangle_{\mathcal{S'},\mathcal{S}},
        \enskip
        \langle \mathcal{F}(S) , \phi \rangle_{\mathcal{S'},\mathcal{S}}
        :=
        \langle S , \widehat{\phi} \rangle_{\mathcal{S'},\mathcal{S}}
    \end{equation}
    and then the following equalities hold in $\mathcal{S}'(\R)$:
    \begin{equation}
        \label{eq:distributions-2}
        \mathcal{F}(D_{\lambda} S)=D_{\lambda^{-1}} \mathcal{F}(S),
        \qquad
        \mathcal{F}(\tau_{\alpha} S)=e^{-i \alpha \cdot} \mathcal{F}(S),
        \qquad
        \mathcal{F}(e^{i \alpha \cdot} S)=\tau_{\alpha} \mathcal{F}(S).
    \end{equation}
\end{definition}

\begin{definition}[Principal value]
    The distribution $\pv(\frac{1}{\omega}) \in \mathcal{S}'(\R)$ is defined by
    \begin{equation}
        \label{def:vp}
        \langle \pv\left(\frac{1}{\omega}\right) , \phi \rangle_{\mathcal{S'},\mathcal{S}} :=
        \lim_{\eps \to 0} \int_{|\omega|>\eps} \frac{\phi(\omega)}{\omega} \dd \omega = \int_0^{\infty} \frac{\phi(\omega)-\phi(-\omega)}{\omega} \dd \omega
    \end{equation}
    and satisfies (see e.g.\ \cite[eq.\ (VII.7.19)]{Schwartz1966})
    \begin{equation}
        \label{Fourier(vp)}
        \mathcal{F} \left(\pv\left(\frac{1}{\omega}\right) \right)=-i\pi \sign(\cdot).
    \end{equation}
\end{definition}

\begin{lemma}
    \label{Lem:Fourier1}
    For $\lambda > 0$,
    \begin{equation}
        \label{F(exp)}
        \mathcal{F}\left( e^{-i\lambda| \cdot |} \right)
        =\pi (\delta_{\lambda}+\delta_{-\lambda})
        - \frac{2i\lambda}{\lambda^2-\omega^2}
    \end{equation}
    in the following sense: for every $\phi \in \mathcal{S}(\R)$,
    \begin{equation}
        \label{eq:F(exp)-limit}
        \langle \mathcal{F}\left( e^{-i\lambda| \cdot |} \right) , \phi \rangle_{\mathcal{S'},\mathcal{S}}
        = \pi (\phi(\lambda)+\phi(-\lambda))
        - 2 i \lambda
        \lim_{\eps \to 0}
        \int_{|\omega \pm \lambda|>\eps} \frac{\phi(\omega)}{\lambda^2-\omega^2} \dd \omega.
    \end{equation}
\end{lemma}

\begin{proof}
    \step[st:Fourier1:1]{We first prove that $\mathcal{F}(e^{-i|\cdot|})=\pi(\delta_{1}+\delta_{-1}) - i (\tau_{-1}-\tau_{+1}) \pv\left(\frac{1}{\omega}\right)$ in $\mathcal{S}'(\R)$.}
    We have $e^{-i|\cdot|}=\cos(\cdot) - i \sin|\cdot|$ and $\mathcal{F}(\cos)=\pi(\delta_{1}+\delta_{-1})$ thus it suffices to prove that
    $\mathcal{F} ( \sin |\cdot| ) = (\tau_{-1}-\tau_{+1}) \pv\left(\frac{1}{\omega}\right)$ in $\mathcal{S}'(\R)$.
    By \eqref{eq:distributions-2} and \eqref{Fourier(vp)}, the Fourier transform of the right-hand side is $2 \pi \sin |\cdot|$.
    The Fourier inversion formula on $\mathcal{S}'(\R)$ gives the result since $\sin |\cdot|$ is even.

    \step{Scaling argument.}
    By \eqref{eq:distributions-1}, \eqref{eq:distributions-2} and \cref{st:Fourier1:1},
    \begin{equation}
        \langle \mathcal{F}\left( e^{-i\lambda| \cdot |} \right) , \phi \rangle_{\mathcal{S'},\mathcal{S}} =
        \langle \mathcal{F}\left( e^{-i| \cdot |} \right) , \phi (\lambda\cdot) \rangle_{\mathcal{S'},\mathcal{S}}
        = \pi (\phi(\lambda)+\phi(-\lambda)) - i L_\lambda \phi
    \end{equation}
    where $L_\lambda \phi := \langle (\tau_{-1} - \tau_{+1}) \pv \left(\frac 1 \omega\right) , \phi(\lambda \cdot) \rangle_{\mathcal{S}',\mathcal{S}}$ so that, recalling \cref{def:vp},
    \begin{equation}
        \begin{split}
            L_\lambda \phi & =
            \lim_{\eps \to 0} \left(
            \int_{|\omega|>\eps} \frac{1}{\omega} \phi(\lambda(\omega+1)) \dd \omega
            - \int_{|\omega|>\eps} \frac{1}{\omega} \phi(\lambda(\omega-1)) \dd \omega \right)
            \\ & = \lim_{\eps \to 0} \left(
            \int_{|y-\lambda|>\eps} \frac{\phi(y)}{y-\lambda}
            \dd y
            - \int_{|y+\lambda|>\eps} \frac{\phi(y)}{y+\lambda}
            \dd y \right)
            \\&= \lim_{\eps \to 0}
            \int_{|y\pm \lambda|>\eps} \left(\frac{1}{y-\lambda}
            -
            \frac{1}{y+\lambda}\right) \phi(y) \dd y.
        \end{split}
    \end{equation}
    The second equality results from the change of variables and the third one from the dominated convergence theorem.
\end{proof}

\begin{remark}
    As in \eqref{def:vp}, the limit in the right-hand side of \eqref{eq:F(exp)-limit} has an explicit expression.
    For $\eps > 0$ small enough, one can write
    \begin{equation}
        \label{j:lim_delta}
        \lim_{\delta \to 0} \int_{|\omega \pm \lambda|>\delta} \frac{\phi(\omega)}{\lambda^2-\omega^2} \dd \omega
        =
        \int_{|\omega \pm \lambda|>\eps} \frac{\phi(\omega)}{\lambda^2-\omega^2} \dd \omega
        + \int_0^{\eps}
        \frac{\zeta(\omega)-\zeta(-\omega)}{\omega}
        \dd \omega
    \end{equation}
    where
    \begin{equation}
        \label{def:zetaj}
        \zeta(\omega):=-\frac{\phi(\lambda+\omega)}{2\lambda+\omega}
        + \frac{\phi(-\lambda+\omega)}{2\lambda-\omega}.
    \end{equation}
    This equality holds for every $\eps \in (0,\lambda)$, so that the denominators in $\zeta$ do not vanish when $\omega \in (0,\eps)$.
\end{remark}

\begin{proposition}
    \label{Prop:Kchapeau}
    For every $\phi \in \mathcal{S}(\R)$,
    \begin{equation}
        \label{eq:Kchapeau}
        \langle \mathcal{F}(K) , \phi \rangle_{\mathcal{S}',\mathcal{S}}=
        \pi \sum_{j\in\N} c_j (\phi(\lambda_j)+\phi(-\lambda_j))-2i
        \int_{\R} \phi(\omega) \Hall(\omega) \dd \omega
    \end{equation}
    where
    \begin{equation}
        \label{eq:Theta}
        \Hall(\omega) := \sum_{j\in\N}\frac{ \lambda_j c_j }{\lambda_j^2-\omega^2},
    \end{equation}
    so that, by \eqref{cj_asympt}, $\Hall \in L^1_{\mathrm{loc}}(\R \setminus \{ \pm \lambda_j \mid j \in \N \};\R)$, and the integral in the right-hand side of \eqref{eq:Kchapeau} must be understood in the sense of a convergent integral (not a Lebesgue one), which is the limit as $\eps \to 0$ of the Lebesgue integral over $F_\varepsilon$ where
    \begin{equation}
        \label{eq:F-eps}
        F_{\eps} := \big\{ \omega \in\R \mid \forall j \in \N, |\omega \pm \lambda_j| \geq \eps\big\}.
    \end{equation}
\end{proposition}

\begin{proof}[Proof of Proposition \ref{Prop:Kchapeau}]
    By \cref{eq:K_0} and \cref{Lem:Fourier1}, it suffices to prove that
    \begin{equation}
        \label{2lim=}
        \lim_{\eps \to 0} \int_{F_{\eps}} \phi(\omega) \sum_{j\in\N}\frac{ \lambda_j c_j }{\lambda_j^2-\omega^2} \dd \omega
        =
        \sum_{j\in\N} \lambda_j c_j \lim_{\delta \to 0}
        \int_{|\omega \pm \lambda_j|>\delta} \frac{\phi(\omega)}{\lambda_j^2-\omega^2} \dd \omega.
    \end{equation}
    Note that the series in the right-hand side of \eqref{2lim=} necessarily converges:
    the series defining $K$ converges in $L^{\infty}(\R)$ thus also in $\mathcal{S}'(\R)$ and so does the series defining $\mathcal{F}(K)$. Therefore, we just need to prove the convergence as $\eps \to 0$.

    For every $\eps>0$, Fubini's theorem proves that
    \begin{equation}
        \label{int_F_eps}
        \int_{F_{\eps}} \phi(\omega) \sum_{j\in\N}\frac{ \lambda_j c_j }{\lambda_j^2-\omega^2} \dd \omega =
        \sum_{j\in\N} \lambda_j c_j \int_{F_{\eps}} \frac{\phi(\omega)}{\lambda_j^2-\omega^2} \dd \omega
    \end{equation}
    because, using \eqref{cj_asympt},
    \begin{equation}
        \int_{F_{\eps}} |\phi(\omega)| \sum_{j\in\N}\frac{ |\lambda_j c_j| }{|\lambda_j^2-\omega^2|} \dd \omega \leq \frac{1}{\eps^2} \| \phi \|_{L^1} \sum_{j\in\N} |\lambda_j c_j| < \infty.
    \end{equation}
    Let $\eps \in (0,1)$ and $j \in \N$.
    The following unions are disjoint:
    \begin{equation}
        \{ \omega\in\R;|\omega\pm\lambda_j|\geq\eps \}=F_{\eps} \sqcup J_j^{\eps}
        \qquad \text{ where } \qquad
        J_j^{\eps}:=
        \underset{k \neq j}{\sqcup}
        \{\omega \in\R;|\omega\pm\lambda_k|< \eps\}.
    \end{equation}
    Using \eqref{int_F_eps} and \eqref{j:lim_delta}, we obtain
    \begin{equation}
        \label{ineq}
        \begin{split}
            \Bigg| \int_{F_{\eps}} & \phi(\omega) \sum_{j\in\N}\frac{ \lambda_j c_j }{\lambda_j^2-\omega^2} \dd \omega -
            \sum_{j\in\N} \lambda_j c_j \lim_{\delta \to 0}
            \int_{|\omega \pm \lambda_j|>\delta} \frac{\phi(\omega)}{\lambda_j^2-\omega^2} \dd \omega \Bigg|
            \\ & \leq
            \sum_{j\in\N} |\lambda_j c_j|\Bigg|
            \int_{F_{\eps}} \frac{ \phi(\omega) }{\lambda_j^2-\omega^2} \dd \omega
            - \int_{|\omega \pm \lambda_j|\geq\eps} \frac{\phi(\omega)}{\lambda_j^2-\omega^2} \dd \omega
            - \int_0^{\eps}
            \frac{\zeta_j(y)-\zeta_j(-y)}{y}
            \dd y
            \Bigg|
            \\ & \leq
            \sum_{j\in\N} |\lambda_j c_j|\Bigg(
            \int_{J_j^{\eps}}
            \frac{|\phi(\omega)|}{|\lambda_j^2-\omega^2|} \dd \omega
            +
            \int_0^{\eps}
            \frac{|\zeta_j(y)-\zeta_j(-y)|}{y}
            \dd y
            \Bigg)
            \\ & \leq
            \sum_{j\in\N} |\lambda_j c_j| \Bigg( \| \phi \|_{L^1(\R \setminus F_\eps)} + \eps \|\phi\|_{W^{1,\infty}(\R)}
            \Bigg)
        \end{split}
    \end{equation}
    where we used $|\lambda_j^2 - \omega^2| \geq 1$ for $\omega \in J_j^\varepsilon$ and the estimate (see \eqref{def:zetaj}), for $y \in (0,\eps)$ and $j \in \N^*$,
    \begin{equation}
        \label{borne_zeta'}
        \frac{|\zeta_j(y)-\zeta_j(-y)|}{y} \leq 2 \|\zeta'\|_{L^{\infty}(0,\eps)} \leq 4 \left( \frac{\|\phi'\|_{L^2}}{|\lambda_j|} + \frac{\|\phi\|_{L^{\infty}}}{|\lambda_j|^2} \right) \leq
        \frac{8}{\pi^2} \|\phi\|_{W^{1,\infty}}.
    \end{equation}
    The dominated convergence theorem proves that the last right-hand side of \eqref{ineq} tends to $0$ as $\eps \to 0$, which proves \eqref{2lim=} and concludes the proof.
\end{proof}

\subsection{Expression of the quadratic form in the frequency domain}

We can now deduce the following expression for the complex-valued bilinear form~$Q$ of \eqref{eq:Q_0}.

\begin{proposition}
    \label{Prop:FQenFourier}
    Let $T>0$.
    For $u,v \in C^{\infty}_c((0,T);\C)$,
    \begin{equation}
        \label{FQ_Fourier}
        2 \pi Q(u,v) =
        \pi \sum_{j\in\N} c_j
        \left( (\widehat{u} \overline{\widehat{v}}) ( \lambda_j)
        + (\widehat{u} \overline{\widehat{v}}) (- \lambda_j)
        \right)
        - 2 i
        \int_{\R} \Hall(\omega)
        \widehat{u} (\omega) \overline{\widehat{v}} (\omega) \dd\omega
    \end{equation}
    where
    the integral of the right-hand side must be understood in the sense of a convergent integral (not a Lebesgue one), which is the limit as $\eps \to 0$ of the Lebesgue integral over $F_{\eps}$ (see \eqref{eq:F-eps}).
\end{proposition}

\begin{proof}
    The function $K \star u$ is bounded because $K \in L^{\infty}(\R)$ and $u \in L^1(\R)$, thus $K \star u \in \mathcal{S}'(\R)$.
    Since $v \in \mathcal{S}(\R)$, we have
    \begin{equation}
        2 \pi Q(u,v)
        = 2 \pi \langle K \star u, \overline{v} \rangle_{\mathcal{S}',\mathcal{S}}
        = \langle \mathcal{F}(K \star u), \overline{\widehat{v}} \rangle_{\mathcal{S}',\mathcal{S}}.
    \end{equation}
    Since $K \in \mathcal{S}'(\R)$ and $u \in \mathcal{E}'(\R)$ (i.e.\ $u$ is a distribution with compact support), we have $\mathcal{F}(K\star u)=\mathcal{F}(K) \widehat{u}$ in $\mathcal{S}'(\R)$.
    Thus $2 \pi Q(u,v) = \langle \mathcal{F}(K), \widehat{u} \overline{\widehat{v}} \rangle_{\mathcal{S}',\mathcal{S}}$ and
    \cref{Prop:Kchapeau} gives the conclusion.
\end{proof}

\subsection{Decomposition of the Fourier kernel}

Formula \eqref{FQ_Fourier} already separates two different contributions to the quadratic form $Q$: one involving Dirac masses at the $\pm \lambda_j$, and one defined by integration of the Fourier transforms against the Fourier kernel $\Theta$ defined in \eqref{eq:Theta}.
To analyze $Q$, we must further decompose $\Theta$.
The following example yields a good intuition.

\begin{example}
    Assume that $\mu \in H^2((0,1);\R)$ is such that $c_0 = 0$ and for all $j \in \N^*$, $c_j = \pi^2 / j^4$ (which is consistent with \eqref{cj_asympt}).
    Then, for $\omega \in \R \setminus \{ \pm \lambda_j \mid j \in \N \}$,
    \begin{equation}
        \label{eq:Theta-example}
        \Hall(\omega)
        = \sum_{j \in \N^*} \frac{1}{j^2} \frac{1}{j^4-(\omega/\pi^2)^2}
        = - \frac{\pi^6}{6 \omega^2} - \frac{\pi^6 \cot( \sqrt{\omega})}{4 \omega^2 \sqrt{\omega}} + \frac{\pi^6 \coth(\sqrt{\omega})}{4 \omega^2 \sqrt{\omega}}.
    \end{equation}
    This equality can be proved by decomposing the rational function with respect to $j$, then recognizing the Mittag-Leffler series for $\cot$ and $\coth$.
\end{example}

To mimic the three components in \eqref{eq:Theta-example}, we write for $\omega \in \R \setminus \{ \pm \lambda_j \mid j \in \N \}$,
\begin{equation}
    \label{eq:Theta=3}
    \Hall(\omega) = - \frac{a}{\omega^2} + \frac{\Hpv(\omega)}{\omega^2} + \frac{\Hreg(\omega)}{\omega^2}
\end{equation}
where $a := \sum_{j \in \N^*} \lambda_j c_j$ was already defined in \eqref{eq:a-ak} and $\Hpv$ isolates, within each frequency region, the dominant singularity.
More precisely,
\begin{equation}
    \label{eq:H_pv-j_om}
    \Hpv(\omega) := \frac{1}{2} \frac{\lambda_{j_\omega}^2 c_{j_\omega}}{\lambda_{j_\omega} - |\omega|}
    \quad \text{where} \quad
    j_\omega := \left\lfloor \sqrt{\frac{1}{4} + \frac{|\omega|}{\pi^2}} + \frac 12 \right\rfloor.
\end{equation}
This corresponds to setting $j_\omega = j$ for $|\omega| \in [\pi^2 (j^2-j), \pi^2 (j^2+j))$ for all $j \in \N$, so that $\vert \omega \vert \approx \lambda_{j_\omega}$.
This corresponds roughly to keeping only the term for $j = j_\omega$ within the sum \eqref{eq:Theta}, up to a lower-order term.
Eventually, $\Hreg$ is the remainder.

The first term in \eqref{eq:Theta=3}, scaling like $\omega^{-2}$ will yield the $H^{-1}$ drift when its coefficient $a \neq 0$.
The second term contains localized singularities at $\pm \lambda_j$.
The associated contribution will be negligible with respect to the first term for small enough times when $a \neq 0$.
However, when $a = 0$, we will exploit their change of sign in \cref{sec:approximate} to construct controls achieving $Q(u,u) = \pm 1$.
The third term will be estimated in a weaker norm (heuristically the $H^{-5/4}$ norm in the case \eqref{eq:Theta-example}), and will always be negligible.

In the following subsections, we derive estimates for each part.

\subsection{Subspace of controls}

Recall that, for a given control $u \in L^2((0,T);\C)$ we denote by $u_1(t) := \int_0^t u$ its primitive vanishing at $0$.
To avoid technicalities, we will establish all the results of the next subsections only for controls in the following subspace (i.e.\ of zero average on $(0,T)$)
\begin{equation}
    \label{eq:cH}
    \mathcal{H}:=\{ u \in L^2((0,T);\C) \mid u_1(T)=0 \}.
\end{equation}
This will be sufficient for our purpose thanks to \cref{lem:Q_k-Proj} and thanks to closed-loop estimates, which allow to estimate $u_1(T)$ from the state (see \cref{sec:closed-loop}).

For $u \in \mathcal{H}$ and $\omega \in \R$, an integration by parts yields
\begin{equation}
    \label{eq:hat-u=hat-u1}
    \widehat{u}(\omega) = i \omega \widehat{u_1}(\omega),
\end{equation}
since the boundary term vanishes thanks to $u_1(T) = 0$.

We will use the notation $\PH$ to denote the orthogonal projection on $\mathcal{H}$, i.e.\ for $u \in L^2((0,T);\C)$, $\PH u := u - u_1(T)/T$.
We have $(\PH u)_1(t) = u_1(t) - u_1(T) \frac t T$.
Thus, using the reverse triangular inequality, for $\nu \geq 0$ and $T \in (0,1]$,
\begin{equation}
    \label{eq:puu1-u1}
    |\nsob{(\PH u)_1}{-\nu} - \nsob{u_1}{-\nu}|
    \leq |u_1(T)| \nsob{\frac{t}{T}}{-\nu}
    \leq |u_1(T)| \left\Vert \frac{t}{T} \right\Vert_{L^2}
    \leq T^{\frac12} |u_1(T)|.
\end{equation}

\subsection{Estimates for the singular part}

Here, we focus on the part of the kernel's Fourier transform involving principal values.
We start by estimating precisely the contribution of the principal value integral for frequencies at distance at most $\delta/T$ of the eigenvalues (see \cref{p:Hpv-bound}) then prove a continuity result for the full contribution associated with $\Hpv$ (see \cref{p:Hpv-easy-bound}).

We use the notation $\lesssim$ of \cref{sec:notations} for constants independent of time and controls.

\begin{proposition}
    \label{p:Hpv-bound}
    There exists $C > 0$ such that, for all $T, \delta > 0$ and $u, v \in L^2((0,T);\C)$,
    \begin{equation}
        \label{eq:Hpv-bound}
        \Bigg| \int_{\R \setminus F_{\delta /T}} \Hpv(\omega) \widehat{u}(\omega) \overline{\widehat{v}}(\omega) \dd \omega \Bigg|
        \leq C T \delta \|u\|_{L^2}\|v\|_{L^2} + C \nsob{u}{-\frac 14} \nsob{v}{-\frac 14}
    \end{equation}
    where, as in \eqref{eq:F-eps}, $\R \setminus F_{\delta / T} = \{ \omega \in \R; \enskip ||\omega| - \lambda_{j_\omega}| \leq \delta / T \}$.
\end{proposition}

\begin{proof}
    To lighten the computations, we focus on the region $\omega \geq 0$ and introduce $\phi(\omega) := \widehat{u}(\omega) \overline{\widehat{v}}(\omega)$.
    Recalling the definition of $j_\omega$ in \eqref{eq:H_pv-j_om}, letting $\delta_j := \min \{ \pi^2 j, \delta / T \}$, the intervals $(\lambda_j - \delta_j, \lambda_j + \delta_j)$ are disjoint and we have
    \begin{equation}
        \label{eq:Q0-sing-1}
        \int_{\R_+ \setminus F_{\delta/T}} \frac{\lambda_{j_\omega}^2 c_{j_\omega}}{\lambda_{j_\omega} - |\omega|} \phi(\omega) \dd \omega
        = \sum_{j=1}^\infty \lambda_j^2 c_j \int_{\lambda_j-\delta_j}^{\lambda_j+\delta_j} \frac{\phi(\omega)}{\lambda_j - \omega} \dd \omega,
    \end{equation}
    where, here and in the sequel of this proof, all such integrals are to be understood as convergent principal value integrals as in \cref{Prop:Kchapeau}.
    Using a Taylor expansion (see \cref{lem:pv_bound}), one has
    \begin{equation}
        \Big| \int_{\lambda_j-\delta_j}^{\lambda_j+\delta_j} \frac{\phi(\omega)}{\lambda_j - \omega} \dd \omega \Big| \leq \int_{\lambda_j-\delta_j}^{\lambda_j+\delta_j} |\phi'| + 2\delta_j \int_{\lambda_j-\delta_j}^{\lambda_j+\delta_j} |\phi''|.
    \end{equation}
    Let $J := \lceil 1 / T \rceil$.
    By \cref{cj_asympt}, for $j \in \N$, $|\lambda_j^2 c_j | \lesssim 1$.
    Thus
    \begin{equation}
        \label{eq:Q0-sing-2}
        \sum_{j=1}^{J-1} \Bigg| \lambda_j^2 c_j \int_{\lambda_j-\delta_j}^{\lambda_j+\delta_j} \frac{\phi(\omega)}{\lambda_j - \omega} \dd \omega \Bigg|
        \lesssim \| \phi' \|_{L^1(0,\pi^2 J^2)} + 2 \frac{\delta}{T} \| \phi'' \|_{L^1}.
    \end{equation}
    For $j \geq J$, using a Taylor expansion (see \cref{lem:pv_bound}), we write
    \begin{equation}
        \Bigg| \int_{\lambda_j-\delta_j}^{\lambda_j+\delta_j} \frac{\phi(\omega)}{\lambda_j - \omega} \dd \omega \Bigg|
        \leq \frac{2\delta}{T} |\phi'(\lambda_j)| + \frac{\delta}{T} \int_{\lambda_j-\delta_j}^{\lambda_j+\delta_j} |\phi''|
    \end{equation}
    By Taylor expansion, for every $\omega \in \R$ and $M > 0$
    \begin{equation}
        |\phi'(\omega)| \leq \frac{1}{2M} \int_{\omega-M}^{\omega+M} |\phi'| + \int_{\omega-M}^{\omega+M}|\phi''|.
    \end{equation}
    In particular, with $\omega = \lambda_j$ and $M = 1/T$,
    \begin{equation}
        |\phi'(\lambda_j)| \leq \frac{T}{2} \int_{\lambda_j-1/T}^{\lambda_j+1/T} |\phi'| + \int_{\lambda_j-1/T}^{\lambda_j+1/T} |\phi''|.
    \end{equation}
    Since $j \geq J \geq 1/T$, the intervals $[\lambda_j-1/T,\lambda_j+1/T]$ are disjoint.
    Thus,
    \begin{equation}
        \label{eq:Q0-sing-3}
        \sum_{j \geq J} \Bigg| \lambda_j^2 c_j \int_{\lambda_j-\delta_j}^{\lambda_j+\delta_j} \frac{\phi(\omega)}{\lambda_j - \omega} \dd \omega \Bigg|
        \lesssim \delta \| \phi' \|_{L^1} + \frac{\delta}{T} \| \phi'' \|_{L^1}.
    \end{equation}
    It is straightforward to check that $\| \phi' \|_{L^1} \leq 4 \pi T \|u \|_{L^2} \| v \|_{L^2}$ and $\| \phi'' \|_{L^1} \leq 8 \pi T^2 \| u \|_{L^2} \| v \|_{L^2}$ using for example that $\partial_\omega \widehat{u}(\omega) = - i \widehat{t u(t)}(\omega)$ and similar formulas for $v$ and second-order derivatives.
    Hence, gathering \eqref{eq:Q0-sing-1}, \eqref{eq:Q0-sing-2} and \eqref{eq:Q0-sing-3},
    \begin{equation}
        \label{eq:Q0-sing-4}
        \begin{split}
            \Bigg| \int_{\R_+ \setminus F_{\delta /T}} \Hpv(\omega) \widehat{u}(\omega) \overline{\widehat{v}}(\omega) \dd \omega \Bigg|
            & \lesssim \| \phi' \|_{L^1(0,\pi^2 J^2)} + \delta T \| u \|_{L^2} \| v \|_{L^2} \\
            & \lesssim (1+\delta) T \| u \|_{L^2} \| v \|_{L^2}.
        \end{split}
    \end{equation}
    Moreover, using that, for $\omega \in [0,\pi^2 J^2]$, $1 \lesssim T^{-1} \langle \omega \rangle^{-\frac 12}$,
    \begin{equation}
        \begin{split}
            \| \phi' \|_{L^1(0,\pi^2 J^2)}
            \leq
            \frac{1}{T} \int_\R \frac{|\phi'(\omega)|}{\langle \omega \rangle^{\frac 12}} \dd \omega
            & \lesssim \frac1T \nsob{u}{-\frac14}\nsob{tv(t)}{-\frac14}
            + \frac1T \nsob{tu(t)}{-\frac14}\nsob{v}{-\frac14} \\
            & \lesssim \nsob{u}{-\frac14}\nsob{v}{-\frac14}
        \end{split}
    \end{equation}
    thanks to \cref{lem:mul-t}.
    Eventually, combined with \eqref{eq:Q0-sing-4}, this yields \eqref{eq:Hpv-bound}.
\end{proof}

The proof of \cref{p:Hpv-bound} actually entails the following continuity result.

\begin{corollary}
    \label{p:Hpv-CT}
    There exists $C > 0$ such that, for all $T > 0$ and $u, v \in L^2((0,T);\C)$,
    \begin{equation}
        \label{eq:Hpv-CT}
        \Bigg| \int_{\R} \Hpv \widehat{u} \overline{\widehat{v}}\Bigg| \leq C T \|u\|_{L^2}\|v\|_{L^2}.
    \end{equation}
\end{corollary}

\begin{proof}
    For $u, v \in L^2((0,T);\C)$, we write, for $T > 0$:
    \begin{equation}
        \int_{\R} \Hpv \widehat{u} \overline{\widehat{v}} =
        \int_{F_{1/T}} \Hpv \widehat{u} \overline{\widehat{v}} + \int_{\R \setminus F_{1/T}} \Hpv \widehat{u} \overline{\widehat{v}},
    \end{equation}
    where $F_{1/T}$ is defined in \eqref{eq:F-eps}.
    By \cref{cj_asympt}, for $j \in \N$, $|\lambda_j^2 c_j | \lesssim 1$.
    Hence, we deduce from the definition~\eqref{eq:H_pv-j_om} of $\Hpv$ that, for all $\omega \in F_{1/T}$, $| \Hpv(\omega) | \lesssim T$.
    Thus,
    \begin{equation}
        \int_{F_{1/T}} | \Hpv \widehat{u} \overline{\widehat{v}} |
        \lesssim T \int_\R | \widehat{u} \overline{\widehat{v}} | \lesssim T \|u\|_{L^2}\|v\|_{L^2}.
    \end{equation}
    From the proof of \cref{p:Hpv-bound} with $\delta = 1$, we have (recall \eqref{eq:Q0-sing-4}):
    \begin{equation}
        \left| \int_{\R \setminus F_{1/T}} \Hpv \widehat{u} \overline{\widehat{v}} \right| \lesssim T \|u\|_{L^2}\|v\|_{L^2}.
    \end{equation}
    Combining these estimates yields the conclusion.
\end{proof}

Thus, $\Hpv(\omega)/\omega^2$ yields a bilinear operator continuous for the $H^{-1}$ topology.

\begin{corollary}
    \label{p:Hpv-easy-bound}
    There exists $C > 0$ such that, for all $T > 0$ and $u, v \in \mathcal{H}$,
    \begin{equation}
        \label{eq:Hpv-easy-bound}
        \Bigg| \int_{\R} \frac{\Hpv(\omega)}{\omega^2} \widehat{u}(\omega) \overline{\widehat{v}}(\omega) \dd \omega \Bigg|
        \leq C T \|u_1\|_{L^2}\|v_1\|_{L^2}.
    \end{equation}
\end{corollary}

\begin{proof}
    For $u, v \in \mathcal{H}$ and $\omega \in \R$, one has $\widehat{u}(\omega)\overline{\widehat{v}}(\omega) = \omega^2 \widehat{u_1}(\omega)\overline{\widehat{v_1}}(\omega)$ by \eqref{eq:hat-u=hat-u1}, so \cref{p:Hpv-CT} immediately implies estimate~\eqref{eq:Hpv-easy-bound}.
\end{proof}

\subsection{Estimates for the Dirac part}

We prove a continuity estimate for the part of $Q$ involving Dirac masses at frequencies $\pm \lambda_j$.

\begin{proposition}
    \label{p:Dirac-bound}
    There exists $C > 0$ such that, for all $T > 0$ and $u, v \in \mathcal{H}$,
    \begin{equation}
        \bigg| \sum_{j\in\N} c_j
        \left( (\widehat{u} \overline{\widehat{v}}) ( \lambda_j)
        + (\widehat{u} \overline{\widehat{v}}) (- \lambda_j)
        \right) \bigg|
        \leq C T \| u_1 \|_{L^2} \| v_1 \|_{L^2} + C \nsob{u_1}{-\frac14} \nsob{v_1}{-\frac14}.
    \end{equation}
\end{proposition}

\begin{proof}
    First, by \eqref{eq:hat-u=hat-u1}, $(\widehat{u} \overline{\widehat{v}})(\omega) = \omega^2 (\widehat{u_1} \overline{\widehat{v_1}})(\omega)$, so that
    \begin{equation}
        \bigg| \sum_{j\in\N} c_j
        \left( (\widehat{u} \overline{\widehat{v}}) ( \lambda_j)
        + (\widehat{u} \overline{\widehat{v}}) (- \lambda_j)
        \right) \bigg|
        =
        \bigg| \sum_{j\in\N} \lambda_j^2 c_j
        \left( (\widehat{u_1} \overline{\widehat{v_1}}) ( \lambda_j)
        + (\widehat{u_1} \overline{\widehat{v_1}}) (- \lambda_j)
        \right) \bigg|.
    \end{equation}
    Moreover, by \eqref{eq:lambda_j} and \eqref{cj_asympt}, $\lambda_j^2 c_j = O(1)$.
    Hence the claimed estimate follows directly from \cref{lem:bessel-best} (see below) applied to $u_1$ and $v_1$.
\end{proof}

\begin{lemma}
    \label{lem:bessel-best}
    There exists $C > 0$ such that, for all $T > 0$ and $u, v \in L^2((0,T);\C)$,
    \begin{equation}
        \sum_{j=0}^{+\infty} |(\widehat{u}\overline{\widehat{v}})(\lambda_j)| \leq C T \| u \|_{L^2} \| v \|_{L^2} + C \nsob{u}{-\frac14}\nsob{v}{-\frac14}.
    \end{equation}
\end{lemma}

\begin{proof}
    By Cauchy--Schwarz, each summand is bounded by $T \|u\|_{L^2} \|v\|_{L^2}$.
    In particular, we can safely omit the term $j = 0$.
    For $j \geq 1$ and any $\omega \in \R$, write by Taylor expansion:
    \begin{equation}
        \widehat{u}\overline{\widehat{v}}(\lambda_j)
        = \widehat{u}\overline{\widehat{v}}(\omega) + \int_\omega^{\lambda_j} (\widehat{u} \partial_\omega \overline{\widehat{v}} + \partial_\omega \widehat{u} \overline{\widehat{v}})
    \end{equation}
    Averaging over $\omega \in I_j := (\lambda_j-j\pi^2,\lambda_j+j\pi^2)$, one has
    \begin{equation}
        \begin{split}
            \widehat{u}\overline{\widehat{v}}(\lambda_j)
            & = \frac{1}{2 \pi^2 j} \int_{I_j} \widehat{u}\overline{\widehat{v}}(\omega) + \frac{1}{2 \pi^2 j} \int_{I_j} \int_\omega^{\lambda_j} (\widehat{u} \partial_\omega \overline{\widehat{v}} + \partial_\omega \widehat{u} \overline{\widehat{v}}), \quad \text{so} \\
            |\widehat{u}\overline{\widehat{v}}(\lambda_j)| & \leq \frac{1}{2 \pi^2 j} \int_{I_j} |\widehat{u} \overline{\widehat{v}}| + \int_{I_j} (|\widehat{u} \partial_\omega \overline{\widehat{v}}| + |\partial_\omega \widehat{u} \overline{\widehat{v}}|).
        \end{split}
    \end{equation}
    Since $\lambda_j = (j\pi)^2$, the intervals $I_j$ are disjoint.
    Moreover, there exists $c > 0$ such that, for any $j \in \N^*$ and $\omega \in I_j$, $1/j \leq c \langle\omega\rangle^{-\frac 12}$.
    Hence, since $\partial_\omega \widehat{u}(\omega) = \mathcal{F}[-itu(t)](\omega)$, we have
    \begin{equation}
        \sum_{j=1}^{+\infty} |\widehat{u}\overline{\widehat{v}}(\lambda_j)|
        \leq \frac{c}{\pi} \nsob{u}{-\frac14}\nsob{v}{-\frac14} + 4 \pi (\| u \|_{L^2} \| t v(t) \|_{L^2} + \| t u(t) \|_{L^2} \| v \|_{L^2}),
    \end{equation}
    which concludes the proof using $\| t u(t) \|_{L^2} \leq T \|u \|_{L^2}$.
\end{proof}

\subsection{Estimate of the regular part}

Here, we focus on the regular part $\Hreg$ of the kernel, for which we first provide a pointwise estimate (see \cref{p:Hreg-local}) which entails a continuity estimate for the associated bilinear operator in a weak norm (see \cref{p:Hreg-Hnu}).

\begin{proposition}
    \label{p:Hreg-local}
    There exists $C > 0$ such that, for all $\omega \in \R$,
    \begin{equation}
        | \Hreg(\omega) | \leq C \frac{\ln (1+ \langle \omega \rangle)}{\langle \omega \rangle^{\frac 12}}.
    \end{equation}
\end{proposition}

\begin{proof}
    By symmetry, we can assume that $\omega \geq 0$.
    
    We start with the case of low frequencies $0 \leq \omega < \pi^2 / 2$.
    Hence $j_\omega = 1$.
    For $\omega \in [0,\pi^2/2]$, $|\lambda_j - \omega| \geq \pi^2/2$ for all $j \in \N^*$.
    By \eqref{cj_asympt}, $\sum |\lambda_j c_j| < +\infty$.
    Thus $\Hall$ is bounded on $[0,\pi^2/2]$.
    On this interval, $|\Hpv(\omega)| \leq \lambda_1^2 |c_1|/\pi^2$.
    Hence $\Hreg(\omega) = \omega^2 \Hall(\omega) + a - \Hpv(\omega)$ is bounded on $[0,\pi^2/2]$.

    We now assume that $\omega \geq \frac{\pi^2}{2}$.
    In particular $\langle \omega \rangle \approx j_\omega^2$.
    Since $a = \sum_{j \in \N^*} \lambda_j c_j$, we write
    \begin{equation}
        |\Hreg(\omega)|
        \leq
        \Bigg| \sum_{j\neq j_\omega} \lambda_j c_j \Bigg( \frac{\omega^2}{\lambda_j^2 - \omega^2} + 1 \Bigg) \Bigg|
        + \Bigg| \frac{\lambda_{j_\omega}^3 c_{j_\omega}}{\lambda_{j_\omega}^2 - \omega^2} - \frac{1}{2} \frac{\lambda_{j_\omega}^2 c_{j_\omega}}{\lambda_{j_\omega} - |\omega|} \Bigg|.
    \end{equation}
    For the second term, we write, using \eqref{cj_asympt}, there exists $C_0 > 0$ such that
    \begin{equation}
        \Bigg| \frac{\lambda_{j_\omega}^3 c_{j_\omega}}{\lambda_{j_\omega}^2 - \omega^2} - \frac{1}{2} \frac{\lambda_{j_\omega}^2 c_{j_\omega}}{\lambda_{j_\omega} - |\omega|} \Bigg|
        = \Bigg| \frac{\lambda_{j_{\omega}}^2 c_{j_\omega}}{2(\lambda_{j_\omega} + \omega)} \Bigg|
        \leq C_0 \langle \omega \rangle^{-1}.
    \end{equation}
    Thus, to conclude, it suffices to prove that there exists $C > 0$ such that, for all $\omega \geq \frac{\pi^2}{2}$,
    \begin{equation}
        \left|
        \sum_{j\neq j_\omega} \lambda_j c_j \left( \frac{\omega^2}{\lambda_j^2 - \omega^2}
        + 1 \right)
        \right|
        \leq C \langle \omega \rangle^{-\frac 12} \ln (1+\langle \omega \rangle).
    \end{equation}
    Since $\lambda_j c_j = O(j^{-2})$, there exists $C_1 > 0$ such that
    \begin{equation}
        \left|
        \sum_{j \neq j_\omega} \lambda_j c_j \left( \frac{\omega^2}{\lambda_j^2 - \omega^2}
        + 1 \right) \right|
        \leq \sum_{j \neq j_\omega} \frac{C_1}{j^2} \left|\frac{\omega^2}{\pi^4 j^4-\omega^2}+1\right|
        = \frac{\pi C_1}{|\omega|^{\frac 1 2}} \frac{\pi}{|\omega|^{\frac 12}} \sum_{j \neq j_\omega} g\left(\frac{j\pi}{|\omega|^{\frac 12}}\right)
    \end{equation}
    where $g(s) := s^2 |s^4-1|^{-1}$.

    The function $g$ is increasing on $[0,1)$ and decreasing on $(1,\infty)$.
    Using a classical series-integral comparison argument, we have
    \begin{equation}
        \frac{\pi}{|\omega|^{\frac 12}} \sum_{j = 1}^{j_\omega - 2} g\left(\frac{j \pi}{|\omega|^{\frac 12}}\right)
        \leq \int_{\pi|\omega|^{-\frac 12}}^{\pi(j_\omega - 1) |\omega|^{-\frac 12}} g(s) \dd s.
    \end{equation}
    From \eqref{eq:H_pv-j_om}, we obtain
    \begin{equation}
        \frac{\pi(j_\omega-1)}{|\omega|^{\frac 12}} \leq \sqrt{1+\alpha^2} - \alpha \leq 1 - \frac{\alpha}{2} = 1 - \frac{\pi}{4 |\omega|^{\frac 12}}
    \end{equation}
    where we used the notation $\alpha := \pi / (2 |\omega|^{\frac 12}) \in (0,1)$, and the intermediate inequality is valid for $\alpha \in (0,1)$.
    Since $g(s) \leq \frac{1}{1-s}$ for $s \in (0,1)$, we obtain
    \begin{equation}
        \begin{split}
            \frac{\pi}{|\omega|^{\frac 12}} \sum_{j = 1}^{j_\omega - 2} g\left(\frac{j\pi}{|\omega|^{\frac 12}}\right)
            & \ \leq \int_0^{\pi (j_\omega - 1) |\omega|^{-\frac 12}} \frac{\dd s}{1 - s} = - \ln\left( 1- \frac{\pi(j_\omega-1)}{|\omega|^{\frac 12}} \right) \\
            & \leq \ln\left( \frac{4 |\omega|^{\frac 12}}{\pi} \right) \leq C_2 \ln (1+\langle\omega\rangle),
        \end{split}
    \end{equation}
    for some absolute constant $C_2 > 0$.
    We had to exclude the term $j = j_\omega - 1$. Using the same estimates, we find
    \begin{equation}
        \frac{\pi}{|\omega|^{\frac 12}} g\left(\frac{(j_\omega-1)\pi}{|\omega|^{\frac 12}}\right)
        \leq \frac{\pi}{|\omega|^{\frac 12}} \frac{4 |\omega|^{\frac 12}}{\pi}
        \leq 4.
    \end{equation}
    Eventually, proceeding similarly we can estimate the term for $j = j_\omega + 1$, and the sum for $j \geq j_\omega + 2$ by a series-integral comparison, using the fact that $g$ is integrable at infinity.
\end{proof}

\begin{corollary}
    \label{p:Hreg-Hnu}
    Let $\nu \in [0,\frac 14)$.
    There exists $C > 0$ such that, for all $T > 0$ and $u, v \in \mathcal{H}$,
    \begin{equation}
        \Bigg| \int_\R \frac{\Hreg(\omega)}{\omega^2} \widehat{u}(\omega) \overline{\widehat{v}}(\omega) \dd \omega \Bigg| \leq C \nsob{u_1}{-\nu} \nsob{v_1}{-\nu}.
    \end{equation}
\end{corollary}

\begin{proof}
    By \eqref{eq:hat-u=hat-u1}, for $u,v \in \mathcal{H}$ and $\omega \in \R$, $(\widehat{u} \overline{\widehat{v}})(\omega) = \omega^2 (\widehat{u_1} \overline{\widehat{v_1}})(\omega)$.
    Hence, by \cref{p:Hreg-local}, there exists $C > 0$ such that
    \begin{equation}
        \Bigg| \int_\R \frac{\Hreg(\omega)}{\omega^2} \widehat{u}(\omega) \overline{\widehat{v}}(\omega) \dd \omega \Bigg|
        \leq C \int_\R \frac{\ln(1+\langle\omega\rangle)}{\langle\omega\rangle^{\frac12}} |(\widehat{u_1} \overline{\widehat{v_1}})(\omega)| \dd \omega.
    \end{equation}
    Thus, fixing any $\nu \in [0,\frac 14)$, there exists $C > 0$ such that $\ln(1+\langle\omega\rangle) \leq C \langle\omega\rangle^{\frac12-2\nu}$ for all $\omega \in \R$, which concludes the proof recalling \cref{def:Htilde-nu}.
\end{proof}

\subsection{Final estimate of the bilinear form}

We gather the previous results to obtain a precise estimate for the bilinear form applied to controls chosen in the subspace $\mathcal{H}$.

\begin{proposition}
    \label{p:Q-a}
    Let $\nu \in [0,\frac 14)$.
    There exists $C>0$ such that, for all $T > 0$, and all $u,v \in \mathcal{H}$,
    \begin{equation}
        \label{eq:bound-Q-a}
        | Q(u,v) - 2 i a \langle u_1, v_1 \rangle | \leq C T \| u_1 \|_{L^2} \| v_1 \|_{L^2} + C \nsob{u_1}{-\nu} \nsob{v_1}{-\nu}.
    \end{equation}
\end{proposition}

\begin{proof}
    Going back to its initial definition \eqref{eq:Q_0} in the time domain from $K \in L^\infty(\R;\C)$ (see \eqref{eq:K_0} and \eqref{cj_asympt}), $Q$ is continuous on $L^2 \times L^2$.
    Hence, by density, it suffices to prove the claimed estimate for $u, v \in \mathcal{H} \cap C^\infty_c((0,T);\C)$.
    For such controls, one has the formula \eqref{FQ_Fourier} in Fourier domain.
    By \cref{p:Dirac-bound}, the term involving Dirac masses is bounded by the right-hand side of \eqref{eq:bound-Q-a} with $\nu = \frac 14$.
    Then, decomposing $\Hall$ as in \eqref{eq:Theta=3}, the term involving $\Hpv$ is bounded by \cref{p:Hpv-easy-bound} and the one involving $\Hreg$ by \cref{p:Hreg-Hnu} for any $\nu \in [0,\frac14)$.
    Eventually, recalling \eqref{eq:hat-u=hat-u1}, the term $-a \omega^{-2}$ yields the contribution $-2ia\langle u_1,v_1\rangle$.
\end{proof}

\begin{corollary}
    \label{cor:FQ-L2}
    For all $T > 0$ and $u,v \in \mathcal{H}$, formula \eqref{FQ_Fourier} holds.
\end{corollary}

\begin{proof}
    As in the proof of \cref{p:Q-a} we can argue by density since we proved that all terms in \eqref{FQ_Fourier} are continuous for the $L^2$ topology.
\end{proof}

\section{Proof of the negative result}
\label{sec:negative}

We prove \cref{thm:negative} with the following steps.
From \cref{sec:quad}, we first prove in \cref{sec:Q-coercive} that, when $a_k \neq 0$, the quadratic form $Q_k$ is coercive on $\mathcal{H}$.
We deduce a coercivity property for the second-order expansion of the state $\psi_2(T)$ in \cref{sec:psi2-coercive}, up to a term involving $u_1(T)$.
We prove a closed-loop estimate on $|u_1(T)|$ in \cref{sec:closed-loop}.
We conclude the proof in \cref{sec:proof-negative}.

We use the notation $\lesssim$ of \cref{sec:notations} for constants independent of time and controls.

\subsection{Coercivity of the quadratic form}
\label{sec:Q-coercive}

The detailed study of \cref{sec:quad} yields the following estimate for the reference quadratic form~$Q$.

\begin{proposition}
    \label{p:Q0-coerc}
    Let $\nu \in [0,\frac12)$.
    There exists $C > 0$ such that, for all $T \in (0,1]$ and $u,v \in \mathcal{H}$,
    \begin{equation}
        \label{eq:Q0-coerc}
        \big| Q(u,v) - 2 i a \langle u_1, v_1 \rangle \big| \leq C T^\nu \| u_1 \|_{L^2} \| v_1 \|_{L^2}.
    \end{equation}
\end{proposition}

\begin{proof}
    This is a straightforward consequence of \cref{p:Q-a}.
    By \eqref{eq:bound-Q-a},
    \begin{equation}
        \big| Q(u,v) - 2 i a \langle u_1, v_1 \rangle \big|
        \lesssim T \| u_1 \|_{L^2} \| v_1 \|_{L^2} + \nsob{u_1}{-\frac{\nu}{2}} \nsob{v_1}{-\frac{\nu}{2}}
        \lesssim T^\nu \| u_1 \|_{L^2} \| v_1 \|_{L^2}
    \end{equation}
    thanks to \cref{lem:H-beta-T-beta} and $T \leq T^\nu$ (for $T \leq 1$).
\end{proof}

\begin{remark}
    \label{rk:Q-positive}
    Estimate \eqref{eq:Q0-coerc} proves that, when $a \neq 0$, for $T$ small enough, the quadratic form $\Im (Q)$ has a sign on the hyperplane $\mathcal{H}$ (see \eqref{eq:cH}).
    For instance, if $a>0$, then $\Im(Q)$ is positive on $\mathcal{H}$ when $T$ is small enough:
    \begin{equation}
        \forall u \in \mathcal{H}, \qquad \Im(Q(u,u)) \geq (2a-C T^{\nu})\|u_1\|_{L^2}^2 > 0.
    \end{equation}
    First, it is necessary to restrict to the hyperplane $\mathcal{H}$ i.e.\ this quadratic form is not positive on $L^2((0,T);\R)$.
    Indeed, for example when $k = 0$ (so that $c_j = |\langle \mu,\varphi_j\rangle|^2$) and $u = \mathbbm{1}_{[0,T]}$ (which does not belong to $\mathcal{H}$ since $u_1(T)=T$), an explicit calculation gives
    \begin{equation}
        \Im Q(u,u)
        = 2 \sum_{j \geq 1} \frac{|\langle \mu, \varphi_j \rangle|^2}{\lambda_j^2} \left( \sin (\lambda_j T) - \lambda_j T \right)
        \leq 2 \left(\sin(\lambda_1 T) - \lambda_1 T \right) \sum_{j \geq 1} \frac{\left\vert \left\langle \mu, \varphi_j \right\rangle \right\vert^2}{\lambda_j^2} < 0.
    \end{equation}
    Moreover, it is also necessary to consider small values of $T$, i.e.\ this quadratic form is not positive on $\mathcal{H}$ when $T$ is large.
    Consider again the case $k = 0$ (so that $c_j = |\langle \mu,\varphi_j\rangle|^2$) for simplicity.
    For every $T>0$, the function $u(t) := (t - \frac{T}{2})\mathbbm{1}_{[0,T]}(t)$ belongs to $\mathcal{H}$ and an explicit computation gives, for every $\lambda > 0$,
    \begin{equation}
        \label{eq:int-exp_T-t}
        \begin{split}
            \int_0^T \int_0^T \sin \left( - \lambda \vert t - s \vert \right) & \left( t - \frac{T}{2} \right) \left( s - \frac{T}{2} \right) \dd s \dd t \\
            = & - \frac{T^{3}}{6 \lambda} - \frac{T^{2} \sin{\left( \lambda T \right)}}{2 \lambda^{2}} - \frac{2 T \cos{\left(\lambda T \right)}}{\lambda^{3}} + \frac{2 \sin{\left(\lambda T \right)}}{\lambda^{4}}.
        \end{split}
    \end{equation}
    Thus, there exists $c_0 > 0$ such that, for $T \gg 1$,
    \begin{equation}
        \Im Q(u,u) \leq \sum_{j \geq 1} \left\vert \left\langle \mu, \varphi_j \right\rangle \right\vert^2 \left( - \frac{T^{3}}{6 \lambda_j} + \frac{T^{2}}{2 \lambda_j^{2}} + \frac{2T}{\lambda_j^{3}} + \frac{2}{\lambda_j^{4}} \right) \leq - c_0 T^3.
    \end{equation}
\end{remark}

\cref{p:Q0-coerc} is an estimate on the $H^{-1}$ continuity and coercivity of the reference quadratic form $Q$ associated with the kernel $K(t-s)$.
Using the abstract toolbox exposed in \cref{sec:multiply}, we can transfer this property to the modulated quadratic form $Q_k$ defined in \eqref{eq:Q_k} (see \cref{sec:modulation}).

\begin{corollary}
    \label{p:Qk-coerc}
    Let $\nu \in [0,\frac12)$.
    There exists $C > 0$ such that, for all $T \in (0,1]$ and $u,v \in \mathcal{H}$,
    \begin{equation}
        | Q_k(u,v) - 2 i a_k \langle u_1, v_1 \rangle | \leq C T^\nu \| u_1 \|_{L^2} \| v_1 \|_{L^2}.
    \end{equation}
\end{corollary}

\begin{proof}
    By \cref{p:Qk-Q},
    \begin{equation}
        \begin{split}
            \big| Q_k(u,v) - Q(u,v) + i \lambda_k K(0) \langle u_1, v_1 \rangle \big| & \lesssim T^2 \| u_1 \|_{L^2} \| v_1 \|_{L^2} + \nsob{u_1}{-\frac{\nu}{2}} \nsob{v_1}{-\frac{\nu}{2}} \\
            & \lesssim T^\nu \| u_1 \|_{L^2} \| v_1 \|_{L^2}
        \end{split}
    \end{equation}
    by \cref{lem:H-beta-T-beta} and using $T^2 \leq T^\nu$ for $T \leq 1$.
    
    Then \eqref{eq:Q0-coerc} of \cref{p:Q0-coerc} and \eqref{eq:a-ak} ($2 a_k = 2 a - \lambda_k K(0)$) conclude.
\end{proof}

\subsection{Coercivity of the second-order expansion}
\label{sec:psi2-coercive}

Using the notations of \cref{sec:modulation}, we now transfer the coercivity of $Q_k$ (defined in \eqref{eq:Q_k} from the kernel $K_k$ of \eqref{eq:K_k}), to the second-order expansion of the state.
We must handle two subtleties: first, \cref{p:Qk-coerc} is only valid for controls in $\mathcal{H}$ and second, the formula for $\langle \psi_2(T),\varphi_k\rangle$ involves $Q_k$ evaluated at controls with a modulated phase.
We establish general lemmas which will be re-used in \cref{sec:positive}.

\begin{lemma}
    \label{lem:Q_k-Proj}
    Let $\nu \in [0,\frac12)$.
    There exists $C > 0$ such that, for all $T \in (0,1]$ and $v \in L^2((0,T);\C)$,
    \begin{equation}
        \label{eq:Q_k-Proj}
        | Q_k(v,\bar{v}) - Q_k(\PH v,\PH \bar{v}) | \leq C T^2 \| v_1 \|_{L^2}^2 + C \nsob{v_1}{-\nu}^2 + C |v_1(T)|^2,
    \end{equation}
    where $\PH v$ is the orthogonal projection of $v$ onto $\mathcal{H}$, i.e.\ $\PH v = v - v_1(T)/T$.
\end{lemma}

\begin{proof}
    Using the symmetry of \eqref{eq:Q_k}, \eqref{eq:K_k} and \eqref{eq:K_0},
    \begin{equation}
        \begin{split}
            Q_k(\PH v,\PH \bar{v}) - Q_k(v,\bar{v})
            & = 2 Q_k(v,\PH \bar{v} - \bar{v}) + Q_k(\PH v-v,\PH \bar{v} - \bar{v}) \\
            & = - \frac{2 v_1(T)}{T} Q_k(v,\mathbbm{1}) + \frac{(v_1(T))^2}{T^2} Q_k(\mathbbm{1},\mathbbm{1}).
        \end{split}
    \end{equation}
    Since $K_k \in L^\infty(\R;\C)$, $|Q_k(\mathbbm{1},\mathbbm{1})| \leq T^2 \| K_k \|_{L^\infty}$.
    From \eqref{eq:Q_k},
    \begin{equation}
        Q_k(v,\mathbbm{1}) = v_1(T) \int_0^T K_k(t-T) \dd t + \int_0^T v_1(s) \int_0^T K_k'(t-s) \dd t \dd s.
    \end{equation}
    The first term is bounded by $T |v_1(T)| \| K_k \|_{L^\infty}$.
    For the second term, we write
    \begin{equation}
        K_k'(\sigma) + i a_k \sign \sigma
        = i (\sign \sigma) \sum_{j \in \N} c_j \left(\lambda_j - \frac{\lambda_k}{2} \right) \left(1 - e^{-i (\lambda_j - \frac{\lambda_k}{2})|\sigma|} \right).
    \end{equation}
    Using that $|1-e^{-i\theta}| \leq \min (2,|\theta|)$, \eqref{eq:lambda_j} and \eqref{cj_asympt}, and splitting the sum at $j \approx |\sigma|^{-\frac12}$
    \begin{equation}
        |K_k'(\sigma) + i a_k \sign \sigma| \leq c \sum_{j\in\N} \frac{\min (1, (1+j^2) |\sigma|)}{1+j^2}
        \leq c^2 |\sigma|^{\frac12}
    \end{equation}
    for some $c > 0$.
    Thus, by Cauchy--Schwarz,
    \begin{equation}
        \int_0^T v_1(s) \int_0^T K_k'(t-s) \dd t \dd s = - i a_k \int_0^T v_1(s) (T-2s) \dd s + O(T^2 \|v_1\|_{L^2}).
    \end{equation}
    The first term is bounded by $T |a_k| \nsob{v_1}{-\nu} \nsob{1-2t/T}{\nu}$.
    By \eqref{eq:norm-1-H+nu} and \cref{lem:mul-t}, we have $\nsob{1-2t/T}{\nu} \lesssim 1$ for $T \in [0,1)$ and $\nu \in (0,\frac 12)$, which concludes the proof.
\end{proof}

\begin{lemma}
    \label{p:Qk-continuous}
    Let $\nu \in [0,\frac14)$.
    There exists $C> 0$ such that, for all $T \in (0,1]$, and $u,v \in \mathcal{H}$,
    \begin{equation}
        \label{eq:Qk-bound}
        | Q_k(u,v) | \leq C (|a_k|+T) \| u_1 \|_{L^2} \| v_1 \|_{L^2} + C \nsob{u_1}{-\nu} \nsob{v_1}{-\nu}.
    \end{equation}
\end{lemma}

\begin{proof}
    This is a direct consequence of \cref{p:Q-a}, \cref{p:Qk-Q} and \eqref{eq:a-ak}.
\end{proof}

\begin{lemma}
    \label{lem:Q_k-conj-best}
    Let $\nu \in [0,\frac14)$.
    There exists $C > 0$ such that, for all $T \in (0,1]$ and $u \in L^2((0,T);\R)$,
    \begin{equation}
        | Q_k(u \rho_k, u \overline{\rho_k}) - Q_k(\PH u,\PH u) | \leq C T (|a_k|+T) \| u_1 \|_{L^2}^2 + C \nsob{u_1}{-\nu}^2 + C |u_1(T)|^2,
    \end{equation}
    where $\rho_k(t) = \exp (i \lambda_k t / 2)$ and $\PH u$ is the orthogonal projection of $u$ onto $\mathcal{H}$.
\end{lemma}

\begin{proof}
    Let $v := u \rho_k$ and $w := v - u$.
    Write
    \begin{equation}
        \label{eq:Qk-conj-best-0}
        Q_k(v,\bar{v}) - Q_k(\PH u,\PH u)
        = Q_k(v,\bar{v}) - Q_k(\PH v, \PH \bar{v})
        + Q_k(\PH w,\PH \bar{v}) + Q_k(\PH u,\PH \bar{w}).
    \end{equation}
    First, by \cref{lem:Q_k-Proj}, one can estimate $Q_k(v,\bar{v}) - Q_k(\PH v, \PH \bar{v})$ as in \eqref{eq:Q_k-Proj}.
    Integrating by parts,
    \begin{equation}
        v_1 = u_1 \rho_k - h_1
        \quad \text{where} \quad
        h := u_1 \rho_k'.
    \end{equation}
    By Cauchy--Schwarz, $\|h_1\|_{L^2} \lesssim T \|u_1\|_{L^2}$.
    Hence $\|v_1\|_{L^2} \lesssim \| u_1 \|_{L^2}$ and $\nsob{v_1}{-\nu} \lesssim \nsob{u_1}{-\nu} + T \|u_1\|_{L^2}$.
    One has $|h_1(T)| \leq \nsob{u_1}{-\nu} \nsob{\rho_k'}{\nu}$ and $\nsob{\rho_k'}{\nu} \lesssim 1$ for $T \in (0,1]$ by \eqref{eq:norm-eit-H+nu}.
    Hence $|v_1(T)| \lesssim |u_1(T)| + \nsob{u_1}{-\nu}$.
    Thus,
    \begin{equation}
        \label{eq:Qk-conj-best-1}
        | Q_k(v,\bar{v}) - Q_k(\PH v,\PH \bar{v}) | \lesssim T^2 \| u_1 \|_{L^2}^2 + \nsob{u_1}{-\nu}^2 + |u_1(T)|^2.
    \end{equation}
    Second, by \cref{p:Qk-continuous},
    \begin{align}
        \label{eq:qkpwpv-1}
        |Q_k(\PH w,\PH \bar{v})| & \lesssim (|a_k|+T) \| (\PH w)_1 \|_{L^2} \| (\PH v)_1 \|_{L^2} + \nsob{(\PH w)_1}{-\nu} \nsob{(\PH v)_1}{-\nu}, \\
        \label{eq:qkpwpv-2}
        |Q_k(\PH u,\PH \bar{w})| & \lesssim (|a_k|+T) \| (\PH u)_1 \|_{L^2} \| (\PH w)_1 \|_{L^2} + \nsob{(\PH u)_1}{-\nu} \nsob{(\PH w)_1}{-\nu}.
    \end{align}
    Terms involving $\PH u$ are estimated by \eqref{eq:puu1-u1}.
    For $w$, integrations by parts lead to
    \begin{equation}
        (\PH w)_1(t) = u_1(t) (\rho_k(t)-1) - \int_0^t u_1 \rho_k' + \frac{t}{T} u_1(T)(1-\rho_k(T)) + \frac{t}{T}\int_0^T u_1 \rho_k'.
    \end{equation}
    Hence, since $\|\rho_k-1\|_{L^\infty(0,T)} \lesssim T$,
    \begin{equation}
        \label{eq:pw1-u1}
        \nsob{(\PH w)_1}{-\nu} \leq \|(\PH w)_1\|_{L^2} \lesssim T \|u_1\|_{L^2} + T |u_1(T)|.
    \end{equation}
    Terms involving $\PH v$ will be estimated by triangular inequality from $\PH u$ and $\PH w$.
    Hence, gathering \eqref{eq:qkpwpv-1}, \eqref{eq:qkpwpv-2}, \eqref{eq:puu1-u1} and \eqref{eq:pw1-u1},
    \begin{equation}
        \label{eq:Qk-conj-best-2}
        |Q_k(\PH w,\PH \bar{v})| + |Q_k(\PH u,\PH \bar{w})| \lesssim T (|a_k|+T) \|u_1\|_{L^2}^2 + \nsob{u_1}{-\nu}^2 + |u_1(T)|^2.
    \end{equation}
    Eventually, the result follows from \eqref{eq:Qk-conj-best-0}, \eqref{eq:Qk-conj-best-1} and \eqref{eq:Qk-conj-best-2}.
\end{proof}

\begin{corollary}
    \label{cor:psi2-coerc}
    Let $\nu \in [0,\frac12)$.
    There exists $C > 0$ such that, for all $T \in (0,1]$, $u \in L^2((0,T);\R)$,
    \begin{equation}
        \label{eq:psi2-coerc}
        \left| \langle \psi_2(T), \varphi_k \rangle + i a_k \| u_1 \|_{L^2}^2 \right| \leq C T^\nu \|u_1\|_{L^2}^2 + C |u_1(T)|^2.
    \end{equation}
\end{corollary}

\begin{proof}
    By \eqref{eq:psi2=Qk} of \cref{sec:modulation}, $\langle \psi_2(T), \varphi_k e^{-i\lambda_k T}\rangle = - \frac 1 2 Q_k(u \rho_k, u \overline{\rho_k})$.

    By \cref{lem:Q_k-conj-best}, 
    \begin{equation}
        |Q_k(u \rho_k, u \overline{\rho_k}) - Q_k(\PH u, \PH u)| \lesssim T \|u_1\|_{L^2}^2 + \nsob{u_1}{-\frac{\nu}{2}}^2 + |u_1(T)|^2.
    \end{equation}
    By \cref{p:Qk-coerc} applied to $(\PH u, \PH u)$,
    \begin{equation}
        \big| Q_k(\PH u, \PH u) - 2 i a_k \| (\PH u)_1 \|_{L^2}^2 \big| \lesssim T^\nu \|(\PH u)_1\|_{L^2}^2.
    \end{equation}
    By \eqref{eq:puu1-u1} and Young's inequality,
    \begin{equation}
        \big| \|u_1\|_{L^2}^2 - \| (\PH u)_1 \|_{L^2}^2 \big| \leq T \|u_1\|_{L^2}^2 + (1+T) |u_1(T)|^2.
    \end{equation}
    Hence,
    \begin{equation}
        \begin{split}
            \left| \langle \psi_2(T), \varphi_k e^{-i\lambda_k T}\rangle + i a_k \| u_1 \|_{L^2}^2 \right| & \lesssim T^\nu \|u_1\|_{L^2}^2 + \nsob{u_1}{-\frac{\nu}{2}}^2 + |u_1(T)|^2 \\
            & \lesssim T^\nu \|u_1\|_{L^2}^2 + |u_1(T)|^2
        \end{split}
    \end{equation}
    thanks to \cref{lem:H-beta-T-beta} (applied to $u_1$).
    The conclusion follows by writing $|e^{i \lambda_k T} - 1| \lesssim T$.
\end{proof}

\subsection{Closed-loop estimate}
\label{sec:closed-loop}

As illustrated in \cref{rk:Q-positive}, the quadratic form is not positive definite on $L^2((0,T);\R)$ but only on the subspace of controls such that $u_1(T) = 0$.
As for the basic ODE example \eqref{eq:x-W1} of the introduction, this still prevents small-time local controllability, because $u_1(T)$ can be estimated in a ``closed-loop'' manner from the state of the system (so it cannot be chosen freely independently on the target).
We refer to \cite[Section 1.6.4]{BeauchardMarbach2022} for more details on this notion.

\begin{proposition}
    \label{p:closed-loop}
    Let $k\in\N$ and $\mu \in H^2((0,1);\R)$ such that $a_k \neq 0$ (see \eqref{def:cj_a}).
    There exists $T^* > 0$ such that, for any $T \in (0,T^*)$ and $u \in L^2((0,T);\R)$ such that $\|u\|_{L^2} \leq 1$,
    \begin{equation}
        |u_1(T)| \lesssim T^{\frac 12} \|u_1\|_{L^2}
        + \|\psi(T)-\varphi_0\|_{L^2}.
    \end{equation}
\end{proposition}

\begin{proof}
    If $T \leq 1$, by the quadratic remainder estimate \eqref{eq:quad-rem} of \cref{cor:quad-rem} (see \cref{sec:estimates-quad}), and the assumption that $\|u\|_{L^2} \leq 1$,
    \begin{equation}
        \label{O1bis}
        \| \psi(T)-\varphi_0-\psi_1(T) \|_{L^2}
        \lesssim T^{\frac 12} |u_1(T)| + T^{\frac 12} \|u_1\|_{L^2}.
    \end{equation}
    Since $a_k \neq 0$, one has $\mu \neq 0$.
    Fix $j \in \N$ such that $\langle \mu,\varphi_j \rangle \neq 0$.
    By \eqref{linearise_explicit},
    \begin{equation}
        \langle (\psi-\varphi_0-\psi_1)(T),\varphi_j \rangle
        =
        \langle \psi(T)-\varphi_0,\varphi_j \rangle - i \langle \mu,\varphi_j\rangle \int_0^T u(t) e^{-i\lambda_j(T-t)} \dd t.
    \end{equation}
    Thus \eqref{O1bis} and $\langle \mu,\varphi_j \rangle \neq 0$ imply that
    \begin{equation}
        \label{estim_moment}
        \left| \int_0^T u(t) e^{i\lambda_j t } \dd t \right| \lesssim
        \|\psi(T)-\varphi_0\|_{L^2} + T^{\frac 12} |u_1(T)| + T^{\frac 12} \|u_1\|_{L^2}.
    \end{equation}
    Using integration by parts and the Cauchy--Schwarz inequality, we get
    \begin{equation}
        |u_1(T)| = \left| \int_0^T \big ( u(t) + i \lambda_j u_1(t) \big) e^{i \lambda_j t} \dd t \right|
        \lesssim
        \left| \int_0^T u(t) e^{i\lambda_j t} \dd t \right|+ T^{\frac 12} \|u_1\|_{L^2}.
    \end{equation}
    Thus incorporating \eqref{estim_moment} in the previous estimate gives the conclusion for small enough times.
\end{proof}

\subsection{Conclusion of the proof}
\label{sec:proof-negative}

We can now conclude the proof of \cref{thm:negative}.

\begin{proof}
    We write
    \begin{equation}
        \langle \psi(T) - \varphi_0, \varphi_k \rangle
        = \langle \psi_1(T), \varphi_k \rangle
        + \langle \psi_2(T), \varphi_k \rangle
        + \langle (\psi - \varphi_0 - \psi_1 - \psi_2)(T), \varphi_k \rangle.
    \end{equation}
    By assumption, $\langle \mu, \varphi_k \rangle = 0$, so $\langle \psi_1(T), \varphi_k \rangle = 0$ by \eqref{linearise_explicit}.
    Thanks to \cref{cor:psi2-coerc} with $\nu = \frac 18$,
    \begin{equation}
        \left| \langle \psi_2(T), \varphi_k \rangle + i a_k \| u_1 \|_{L^2}^2 \right| \lesssim T^{\frac 18} \|u_1\|_{L^2}^2 + |u_1(T)|^2.
    \end{equation}
    By the cubic remainder estimate \eqref{eq:cubic-rem} of \cref{cor:cubic-rem} (see \cref{sec:estimates-cub}), for $T \leq 1$ and $\|u\|_{L^2} \leq 1$,
    \begin{equation}
        | \langle (\psi - \varphi_0 - \psi_1 - \psi_2)(T), \varphi_k \rangle | \lesssim |u_1(T)|^3 + T^{\frac 18} \| u_1 \|_{L^2}^2.
    \end{equation}
    Eventually, if $T < T^*$ given by \cref{p:closed-loop},
    \begin{equation}
        |u_1(T)|^3 \lesssim |u_1(T)|^2 \lesssim T \|u_1\|_{L^2}^2 + \|\psi(T)-\varphi_0\|_{L^2}^2.
    \end{equation}
    Gathering these estimates yields the claimed drift estimate \eqref{drift} with $\nu = \frac 18$.
\end{proof}

\section{Proof of the positive result}
\label{sec:positive}

We prove \cref{thm:positive}.
In all this section, we therefore assume that $k \in \N^*$ is fixed with $\langle \mu, \varphi_k \rangle = 0$, $a_k = 0$ (see \cref{def:cj_a}) and that assumption \eqref{eq:hyp-c0-k} holds (see \cref{rk:k=0} for the case $k = 0$).
The core of the proof lies in the analysis of the quadratic form $Q$ to achieve motions in the directions $\pm i \varphi_k$ (see \cref{sec:approximate}).
The other parts of the argument are relatively standard (motion in the directions $\pm \varphi_k$, control in projection, and nonlinear fixed-point argument with estimate of the cost).

\subsection{Prequel: solvability of moment problems}

Using assumption \eqref{eq:hyp-c0-k}, most of the components of the state will be controlled at the linearized level.
Starting from \eqref{linearise_explicit}, this typically relies on the solvability of moment problems.

\begin{definition}
    Let $m \in \Z$.
    Let $h^m_r(\N;\C)$ be the set of sequences $(d_j)_{j \in \N}$ such that $d_0 \in \R$ and
    \begin{equation}
        \|d\|_{h^m} := \| \langle j \rangle^{m} d_j \|_{\ell^2} < +\infty.
    \end{equation}
\end{definition}

Given a $d \in h^0_r(\N;\C)$, a \emph{moment problem} consists in finding a control $u$ such that
\begin{equation}
    \label{eq:moments}
    \forall j \in \N, \qquad \int_0^T u(t) e^{- i \lambda_j t} \dd t = d_j.
\end{equation}
It is known that this problem is solvable and that for $m \in \Z$, the $H^m$ norm of $u$ scales like the $h^{2m}$ norm of $d$.
We will need the following results, which are particular cases of \cite[Theorem 3.10]{Bournissou2023_Lin}.

\begin{proposition}
    \label{p:Megane-1}
    For every $T>0$, there exists $C_T>0$ and a continuous linear map $L_T:h^0_r(\N;\C) \rightarrow L^2((0,T);\R)$ such that, for every $d \in h^0_r(\N;\C)$, the control $u:=L_T(d)$ satisfies \eqref{eq:moments},
    \begin{equation}
        \|u\|_{L^2} \leq C_T \|d\|_{h^0}
        \qquad \text{and} \qquad
        \| u_1 \|_{L^2} + |u_1(T)| \leq C_T \|d\|_{h^{-2}}.
    \end{equation}
\end{proposition}

\begin{proposition}
    \label{p:Megane+1}
    For every $T > 0$, there exists $C_T > 0$ and a continuous linear map $L_T : h^2_r(\N;\C) \to H^1_0((0,T);\R)$ such that, for every $d \in h^2_r(\N;\C)$, the control $u := L_T(d)$ satisfies \eqref{eq:moments},
    \begin{equation}
        \label{eq:Megane+1}
        \|u\|_{L^2} \leq C_T \|d\|_{h^0}
        \qquad \text{and} \qquad
        \|u\|_{H^1_0} \leq C_T \|d\|_{h^{2}}.
    \end{equation}
\end{proposition}

\subsection{Approximate movement in the main direction}
\label{sec:approximate}

Let $T > 0$.
We now use the analysis of \cref{sec:quad} to construct controls $u^{\pm i}$ such that the state moves approximately in the directions $\pm i \varphi_k$.
The goal of this section is to prove \cref{p:vect-tgt-Im}.
Heuristically, we will use a control such that
\begin{equation}
    u_1(t) \propto \frac{1}{T^{\frac12}} \cos (t\omega_0) \chi\left(\frac{t}{T}\right),
\end{equation}
where $\chi$ is a smooth cut-off function and $\omega_0$ is a very carefully chosen frequency, located near (but not too close to) a singularity of $\Hpv$ (depending on the desired sign of the motion).

\begin{lemma}
    \label{lem:chi}
    There exists $\chi \in C^\infty_c((0,1);\R)$ such that $\| \chi \|_{L^2} = 1$ and
    \begin{equation}
        \forall p \in \Z \setminus \{0\}, \qquad \widehat{\chi}(4 p \pi) = 0.
    \end{equation}
\end{lemma}

\begin{proof}
    Let $f := \mathbbm{1}_{[\frac 14, \frac 34]}$.
    A direct integration shows that
    \begin{equation}
        \widehat{f}(\omega)
        = \int_{\frac 14}^{\frac 34} e^{-i \omega t} \dd t
        = 2 e^{-i \frac{\omega}{2}} \frac{\sin (\omega / 4)}{\omega}.
    \end{equation}
    In particular, for any $p \in \Z \setminus \{0\}$, $\widehat{f}(4 p \pi) = 0$.
    Now take $\rho \in C^\infty_c(\R;\R)$ a smooth mollifier with $\supp \rho \subset [-\frac 18, \frac 18]$.
    Let $\chi_0 := f \star \rho$.
    By standard properties of convolution, $\chi_0 \in C^\infty$, $\supp \chi_0 \subset [\frac 18, \frac 78] \subset (0,1)$, and $\widehat{\chi_0} = \widehat{f} \widehat{\rho}$.
    Thus, for every $p \in \Z \setminus \{0\}$, $\widehat{\chi_0}(4 p \pi) = 0$.
    Eventually, setting $\chi := \chi_0 / \| \chi_0 \|_{L^2}$ concludes the proof.
\end{proof}

\begin{lemma}
    \label{lem:e-i-omega0-H-s}
    Let $\nu \in [0,\frac 12)$ and $\chi \in H^1_0((0,1);\R)$.
    There exists $C(\chi,\nu) > 0$ such that, for all $T > 0$ and $\omega_0 \geq 1/T$,
    \begin{equation}
        \nsob{T^{-\frac 12} e^{i \omega_0 t} \chi(t/T)}{-\nu} \leq C(\chi,\nu) \omega_0^{-\nu}.
    \end{equation}
\end{lemma}

\begin{proof}
    Let $f(t) := T^{-\frac 12} e^{i \omega_0 t} \chi(t/T)$.
    Its Fourier transform is $\widehat{f}(\omega) = T^{\frac 12} \widehat{\chi}(T(\omega - \omega_0))$.
    Using the definition \eqref{eq:Hnu} and the change of variables $\sigma = T(\omega - \omega_0)$, the squared norm is
    \begin{equation}
        \nsob{f}{-\nu}^2
        = \frac{1}{2 \pi} \int_{\R} \frac{|\widehat{\chi}(\sigma)|^2 }{\langle \omega_0 + \sigma/T \rangle^{2\nu}} \dd\sigma
        \leq \frac{T^{2\nu}}{2 \pi} \int_{\R} \frac{|\widehat{\chi}(\sigma)|^2}{|A + \sigma|^{2\nu}} \dd\sigma,
    \end{equation}
    using $\langle \omega_0 + \sigma/T \rangle \geq |A+\sigma|/T$, where $A := T \omega_0 \geq 1$.
    We split the integral in two regions.

    First, when $|A+\sigma| \geq A/2$, we have $|A + \sigma|^{-2\nu} \leq 2^{2\nu} A^{-2\nu}$.
    Hence
    \begin{equation}
        \int_{|A+\sigma| \geq A/2} \frac{|\widehat{\chi}(\sigma)|^2}{|A + \sigma|^{2\nu}} \dd\sigma
        \leq 2^{2\nu} A^{-2\nu} \| \widehat{\chi} \|_{L^2}^2.
    \end{equation}
    Second, when $|A+\sigma| \leq A/2$, $|\sigma| \geq A/2$.
    Since $\chi \in H^1_0((0,1);\R)$, for all $\sigma \in \R$, $|\widehat{\chi}(\sigma)| \leq \| \chi' \|_{L^2} / |\sigma|$.
    Thus
    \begin{equation}
        \int_{|A+\sigma| \leq A/2} \frac{|\widehat{\chi}(\sigma)|^2}{|A + \sigma|^{2\nu}} \dd\sigma
        \leq \left(\frac{\| \chi' \|_{L^2}}{A/2}\right)^2 \int_{|y| < A/2} \frac{\dd y}{|y|^{2\nu}}
        = \frac{4 \| \chi' \|_{L^2}^2}{A^2} \frac{2^{2\nu} A^{1-2\nu}}{1-2\nu},
    \end{equation}
    where we used that $\nu < \frac 12$ to obtain the convergence of the integral.
    Recalling that $A \geq 1$, we conclude the proof with $C(\chi,\nu) := 2^{\nu} \| \chi \|_{L^2} + 2^{\nu} \| \chi' \|_{L^2} (1-2\nu)^{-\frac 12}$.
\end{proof}

\begin{proposition}
    \label{p:two-signs}
    There exists $C > 0$ such that the following property holds.
    For all $\nu \in (0,\frac12)$, $T > 0$ and $\eta > 0$, there exist $V^\pm \in C^\infty_c((0,T);\R)$ with $\| V^\pm \|_{L^2} \leq C$ such that
    \begin{equation}
        \label{eq:two-signs}
        \frac{1}{2\pi} \int_\R \Hpv |\widehat{V^\pm}|^2 = \pm T,
        \qquad
        \nsob{V^\pm}{-\nu} \leq \eta
        \qquad \text{and} \qquad
        \sum_{j\in \N} \langle j \rangle^4 | \widehat{V^\pm}(\lambda_j) |^2 \leq \eta^2.
    \end{equation}
\end{proposition}

\begin{proof}
    Let $\chi$ be given by \cref{lem:chi}.
    In the sequel of this proof, we will denote by $C_\chi$ various constants depending only on $\chi$.
    In particular, since $\chi \in C^\infty_c((0,1);\R)$, $| \widehat{\chi}(\omega) | \leq C_\chi \langle \omega \rangle^{-4}$.

    For $T > 0$, let $\omega_0 := \lambda_{j_0} + \frac{\eps_0\beta}{T}$ where $j_0 \in \N^*$, $\beta > 0$ and $\eps_0 \in \{ \pm 1 \}$ are to be chosen later.
    Define
    \begin{equation}
        u(t) := 2 T^{-\frac 12} \cos (t \omega_0) \chi (t/T).
    \end{equation}
    Thus,
    \begin{equation}
        \label{eq:hat-u-cos}
        \widehat{u}(\omega) = T^{\frac 12} \widehat{\chi}(T \omega + T \omega_0) + T^{\frac 12} \widehat{\chi}(T \omega - T \omega_0).
    \end{equation}
    We claim that one can choose $\omega_0$ in such a way that $u$ will satisfy the requirements for $V^\pm$.
    First, one has $\|u\|_{L^2(0,T)} \leq 2 \|\chi\|_{L^2(0,1)} = 2$.
    We now study the three items of \eqref{eq:two-signs}.

    \step[st:two-signs-1]{Smallness of weak Sobolev norms of the candidate control.}
    For any fixed $\nu \in (0,\frac 12)$, by \cref{lem:e-i-omega0-H-s}, there exists a constant $C(\chi,\nu)$ such that, for all $T > 0$ and $\omega_0 \geq 1/T$,
    \begin{equation}
        \nsob{u}{-\nu} \leq 2 C(\chi,\nu) \omega_0^{-\nu}.
    \end{equation}
    Thus, for any fixed $(T, \varepsilon_0,\beta,\eta)$, for $j_0$ large enough, $\nsob{u}{-\nu} \leq \eta$.
    In particular, for any fixed $(T,\varepsilon_0,\beta)$, for $j_0$ large enough, $\nsob{u}{-\frac14}^2 \leq T/\beta^2$ (which will be used in \cref{st:two-signs-2}).

    \step[st:two-signs-2]{Computation of $\Hpv$.
    We prove that there exists a constant $M(\chi,\mu)$ such that, for any fixed tuple $(T,\varepsilon_0,\beta)$ with $\beta > 2^{\frac 13}$, for $j_0$ large enough,}
    \begin{equation}
        \label{eq:step2}
        \Bigg| \frac{1}{2\pi} \int_\R \Hpv |\widehat{u}|^2 + \varepsilon_0 \lambda_{j_0}^2 c_{j_0} \frac{T}{\beta} \Bigg| \leq \frac{M T}{\beta^2}.
    \end{equation}
    Recalling the notation \eqref{eq:F-eps}, write
    \begin{equation}
        \label{eq:Hpv-b2T}
        \int_\R \Hpv |\widehat{u}|^2 =
        \int_{F_{1/(\beta^2T)}} \Hpv |\widehat{u}|^2 +
        \int_{\R \setminus F_{1/(\beta^2T)}} \Hpv |\widehat{u}|^2.
    \end{equation}
    We now estimate both terms.
    \begin{itemize}
        \item \emph{Regions close to the poles.}
        By \cref{p:Hpv-bound} with $\delta = 1/\beta^2$, there exists $C_1 > 0$ such that
        \begin{equation}
            \label{eq:step2-close-1}
            \Bigg| \int_{\R \setminus F_{1/(\beta^2T)}} \Hpv |\widehat{u}|^2 \Bigg| \leq \frac{C_1 T}{\beta^2} \|u\|_{L^2}^2 + C_1 \nsob{u}{-\frac14}^2
            \leq \frac{5 C_1 T}{\beta^2}
        \end{equation}
        for $j_0$ large enough (for fixed $(T, \varepsilon_0, \beta)$), using $\|u\|_{L^2} \leq 2$ and \cref{st:two-signs-1}.

        \item \emph{Regions far from the poles.}
        We focus on the first term of \eqref{eq:Hpv-b2T}.
        Let $I_0 := [\omega_0 - \frac{\beta}{2T}, \omega_0 + \frac{\beta}{2T}]$.
        First, let us prove that there exists $C_2 > 0$ such that
        \begin{equation}
            \label{eq:step2-far-1}
            \Bigg| \int_{F_{1/(\beta^2 T)}} \Hpv |\widehat{u}|^2
            - \lambda_{j_0}^2 c_{j_0} \int_{I_0} \frac{|\widehat{u}(\omega)|^2}{\lambda_{j_0}-\omega} \dd \omega \Bigg|
            \leq \frac{C_2 T}{\beta^2}.
        \end{equation}
        Taking $\beta > 2^{\frac 13}$ and $j_0>\frac{3\beta}{2\pi^2 T}$, we have $I_0 \subset F_{1/(\beta^2T)}$ and $\Hpv(\omega) = \frac 12 \lambda_{j_0}^2 c_{j_0} / (\lambda_{j_0} - \omega)$ on $I_0$.

        By \eqref{cj_asympt}, there exists $A_\mu > 0$ such that, for all $j \in \N$, $|\lambda_j^2 c_j| \leq A_\mu$.
        In particular, by \eqref{eq:H_pv-j_om}, for $\omega \in F_{1/(\beta^2T)}$, $|\Hpv(\omega)|\leq A_\mu \beta^2 T$.
        Hence, using that $\Hpv |\widehat{u}|^2$ is even,
        \begin{equation}
            \Bigg| \int_{F_{1/(\beta^2 T)}} \Hpv |\widehat{u}|^2
            - \lambda_{j_0}^2 c_{j_0} \int_{I_0} \frac{|\widehat{u}(\omega)|^2}{\lambda_{j_0}-\omega} \dd \omega \Bigg|
            \leq A_\mu \beta^2 T \int_{||\omega| - \omega_0| \geq \frac{\beta}{2T}} |\widehat{u}(\omega)|^2 \dd \omega.
        \end{equation}
        Since $\chi \in C^\infty_c((0,1);\R)$, using \eqref{eq:hat-u-cos},
        \begin{equation}
            \int_{||\omega| - \omega_0| \geq \frac{\beta}{2T}} |\widehat{u}(\omega)|^2 \dd \omega
            \leq 4 \int_{|\omega| \geq \frac{\beta}{2}} |\widehat{\chi}(\omega)|^2 \dd \omega
            \leq \frac{C_\chi}{\beta^4}.
        \end{equation}
        Combining these two estimates yields \eqref{eq:step2-far-1} with $C_2 := A_\mu C_\chi$.

        Second, using again the decay of $\widehat{\chi}$, \eqref{eq:hat-u-cos} and Cauchy--Schwarz, one proves that there exists $C_3 > 0$ such that
        \begin{equation}
            \label{eq:step2-far-2}
            \Bigg| \int_{I_0} \frac{|\widehat{u}(\omega)|^2}{\lambda_{j_0}-\omega} \dd \omega -
            \int_{I_0} \frac{T |\widehat{\chi}(T(\omega-\omega_0))|^2}{\lambda_{j_0}-\omega} \dd \omega \Bigg| \leq \frac{C_3 T}{\beta^2}.
        \end{equation}

        Third, by a change of variables,
        \begin{equation}
            \int_{I_0} \frac{T | \widehat{\chi}(T(\omega-\omega_0))|^2}{\lambda_{j_0}-\omega} \dd \omega
            = - \frac{\eps_0 T}{\beta} \int_{-\frac{\beta}{2}}^{\frac{\beta}{2}} \frac{|\widehat{\chi}(\sigma)|^2}{1+\frac{\eps_0\sigma}{\beta}} \dd \sigma.
        \end{equation}
        Using $\left\vert \frac{1}{1+x} - 1 \right\vert \leq 2 \vert x \vert$ for all $\vert x \vert \leq \frac{1}{2}$ and the decay of $\chi$, one finds
        \begin{equation}
            \label{eq:step2-far-3_bis}
            \begin{split}
                \left\vert \int_{I_0} \frac{T | \widehat{\chi}(T(\omega-\omega_0))|^2}{\lambda_{j_0}-\omega} \dd \omega + 2\pi \frac{\eps_0 T}{\beta} \right\vert
                & \leq \frac{T}{\beta} \int_{-\frac{\beta}{2}}^{\frac{\beta}{2}} 2 \frac{\vert \sigma \vert}{\beta} |\widehat{\chi}(\sigma)|^2 \dd \sigma
                + \frac{T}{\beta} \int_{\vert \sigma \vert > \frac{\beta}{2}} |\widehat{\chi}(\sigma)|^2 \dd \sigma
                \\ & \leq \frac{C_4 T}{\beta^2},
            \end{split}
        \end{equation}
        for some $C_4 > 0$.
    \end{itemize}
    Eventually, combining \eqref{eq:step2-close-1}, \eqref{eq:step2-far-1}, \eqref{eq:step2-far-2} and \eqref{eq:step2-far-3_bis} proves \eqref{eq:step2}.

    \step{Choice of the parameters.}
    Using \cref{lem:J-mu}, let $a_\mu > 0$ such that $|\lambda_j^2 c_j| \geq a_\mu$ for $j$ large enough.
    We now fix a number $\beta \in 4 \pi \N^*$ such that $\beta \geq 2 M / a_\mu$.
    Then, for any fixed $T > 0$ and $\varepsilon_0 \in \{ \pm 1 \}$, for $j_0$ large enough, by \eqref{eq:step2}, one has
    \begin{equation}
        \frac{1}{2 \pi} \int_\R \Hpv |\widehat{u}|^2 = - \varepsilon_0 T \left( \frac{\lambda_{j_0}^2 c_{j_0}}{\beta}(1+\delta_u) \right),
    \end{equation}
    where $|\delta_u| \leq \frac 12$.
    The absolute value of the term inside the parenthesis is uniformly bounded below and above.
    Thus, one can define $V^\pm$ by choosing $\varepsilon_0$ and rescaling $u$, yielding a uniform constant $C$ for the $L^2$ norm of $V^\pm$.

    \step{Estimate of the moments sum.}
    To estimate the weighted sum of the moments in \eqref{eq:two-signs}, we must separate the one for $j = j_0$.
    \begin{itemize}
        \item For $j = j_0$, recalling that $\omega_0 = \lambda_{j_0} + \frac{\eps_0 \beta}{T}$, by \eqref{eq:hat-u-cos},
        \begin{equation}
            \widehat{u}(\lambda_{j_0}) = T^{\frac 12} \widehat{\chi}(2 T \lambda_{j_0} + \eps_0 \beta) + T^{\frac 12} \widehat{\chi}(- \eps_0 \beta)
        \end{equation}
        We have chosen $\beta \in 4 \pi \N^*$, so that, by \cref{lem:chi}, $\widehat{\chi}(-\eps_0 \beta) = 0$.
        If $\pi^2 j_0^2 \geq \beta/T$, we have $2T\lambda_{j_0} + \eps_0 \beta \geq T \lambda_{j_0}$, hence $|\widehat{u}(\lambda_{j_0})| \leq C_\chi T^{1/2} \langle T \lambda_{j_0} \rangle^{-4}$.
        Thus $\langle j_0 \rangle^4 |\widehat{u}(\lambda_{j_0})|^2 \leq C_\chi j_0^{4-16} \leq \frac 12 \eta^2$ for $j_0$ large enough.

        \item For $j \neq j_0$, if $j_0 \geq \beta/T$, then $|\lambda_j \pm \omega_0| \geq (j+j_0)$.
        So, using \eqref{eq:hat-u-cos},
        \begin{equation}
            |\widehat{u}(\lambda_{j})|^2 \leq C_\chi T \langle T(j+j_0) \rangle^{-8}.
        \end{equation}
        Hence
        \begin{equation}
            \sum_{j \neq j_0} \langle j \rangle^4 |\widehat{u}(\lambda_{j})|^2
            \leq \frac{C_\chi}{T^{15} j_0^3} \leq \frac 12 \eta^2
        \end{equation}
        for $j_0$ large enough.\qedhere
    \end{itemize}
\end{proof}

\begin{lemma}
    \label{lem:J-mu}
    If $\mu \in H^2((0,1);\R)$ satisfies \eqref{eq:hyp-c0-k}, there exist $a_\mu > 0$ and $J_\mu \in \N$ such that
    \begin{equation}
        \label{eq:Jmu}
        \forall j \geq J_\mu, \quad
        |\lambda_j^2 c_j| \geq a_\mu.
    \end{equation}
\end{lemma}

\begin{proof}
    Since $\mu \in H^2$, $\langle \mu'',\varphi_j\rangle = o(1)$.
    Thus, assumption \eqref{eq:hyp-c0-k} and equality \eqref{eq:mu-phi_j} imply that $a_\pm := \mu'(1) \pm \mu'(0) \neq 0$.
    Applying equality \eqref{eq:mu-phi_j} with $\mu \gets \mu \varphi_k$, we obtain
    \begin{equation}
        \langle \mu \varphi_k, \varphi_j \rangle
        = 2 \frac{(-1)^{j+k}\mu'(1)-\mu'(0)}{(j\pi)^2} - \frac{\langle (\mu \varphi_k)'',\varphi_j \rangle}{(j\pi)^2}.
    \end{equation}
    Since $\mu \varphi_k \in H^2$, $\langle (\mu \varphi_k)'',\varphi_j \rangle = o(1)$.
    Combined with \eqref{eq:mu-phi_j} and the definition \eqref{def:cj_a} of the coefficients $c_j$, this implies that for any fixed $a_\mu < \min \{ a_-^2, a_+^2 \}$, there exists a $J_\mu$ satisfying \eqref{eq:Jmu}.
\end{proof}

\begin{proposition}
    \label{p:two-signs-null-moments}
    Let $\nu \in (0,\frac12)$.
    There exists $C > 0$ such that the following property holds.
    For all $T > 0$ and $\eta > 0$, there exist $U^\pm \in H^1_0((0,T);\R)$ with $\| U^\pm \|_{L^2} \leq C$ such that
    \begin{equation}
        \label{eq:two-signs-null-moments}
        \frac{1}{2\pi} \int_\R \Hpv |\widehat{U^\pm}|^2 = \pm T,
        \qquad
        \nsob{U^\pm}{-\nu} \leq \eta
        \qquad \text{and} \qquad
        \forall j \in \N, \quad \widehat{U^\pm}(\lambda_j) = 0.
    \end{equation}
\end{proposition}

\begin{proof}
    Fix $\nu \in (0,\frac 12)$.
    Let $C_1 > 0$ be given by \cref{p:two-signs}.
    Let $T > 0$ and $\eta > 0$.
    We construct $U^+$.
    We will consider a control of the form $U^+ := V^+ + W$ where $V^+ \in H^1_0((0,T);\R)$ is given by \cref{p:two-signs} for some $\eta' > 0$ to be chosen later, and $W := L_T(d)$, where $L_T : h^2_r(\N;\C) \to H^1_0((0,T);\R)$ is the map given by \cref{p:Megane+1}, associated with a continuity constant $C_T$ from $h^2 \to H^1_0$ as in \eqref{eq:Megane+1}, and $d_j := - \widehat{V^+}(\lambda_j)$.
    This guarantees that, for all $j \in \N$, $\widehat{U^+}(\lambda_j) = 0$.
    We have, by \eqref{eq:Megane+1} and \eqref{eq:two-signs},
    \begin{equation}
        \| W \|_{H^1} \leq C_T \| d \|_{h^2} \leq C_T \eta'.
    \end{equation}
    Thus, choosing $\eta'$ small enough, by triangular inequality, $\nsob{U^+}{-\nu} \leq \nsob{V^+}{-\nu} + \|W\|_{L^2} \leq \frac{\eta}{2}$.
    Now,
    \begin{equation}
        \int_\R \Hpv |\widehat{U^+}|^2
        = \int_\R \Hpv |\widehat{V^+}|^2 + 2 \Re \int_\R \Hpv \widehat{V^+} \overline{\widehat{W}} + \int_\R \Hpv |\widehat{W}|^2.
    \end{equation}
    By \cref{p:Hpv-CT}, there exists $C_0 > 0$ such that
    \begin{equation}
        \begin{split}
            \Big| 2 \Re \int_\R \Hpv \widehat{V^+} \overline{\widehat{W}} + \int_\R \Hpv |\widehat{W}|^2 \Big|
            & \leq C_0 T \| V^+ \|_{L^2} \| W \|_{L^2} + C_0 T \| W \|_{L^2}^2 \\
            & \leq C_0 T C_1 C_T \eta' + C_0 T (C_T \eta')^2 \leq \pi T
        \end{split}
    \end{equation}
    up to choosing $\eta'$ small enough.
    Hence, there exists $c \in (\frac12,\frac32)$ such that $\frac{1}{2\pi} \int_\R \Hpv |\widehat{U^+}|^2 = c T$.
    Eventually, choosing $U^+ := \frac{1}{\sqrt{c}}(V^+ + W)$ satisfies all conditions of \eqref{eq:two-signs-null-moments} and $C := 2 C_1$.
\end{proof}

\begin{proposition}
    \label{p:vect-tgt-Im}
    There exist $T^*, C > 0$ such that, for all $T \in (0,T^*)$, there exist controls $u^{\pm i} \in L^2((0,T);\R)$ with $u_1^{\pm i}(T) = 0$ satisfying $\| u^{\pm i}_1 \|_{L^2} \leq C$ such that
    \begin{equation}
        \label{eq:vect-tgt-Im}
        \psi_1(T) = 0, \quad
        \Im \langle \psi_2(T), \varphi_k \rangle = \pm T, \quad
        |\Re \langle \psi_2(T), \varphi_k \rangle| \leq C T^2,
    \end{equation}
    where $\psi_1$ and $\psi_2$ are the linear and quadratic order approximations, solutions to \eqref{eq:psi1} and \eqref{eq:psi2}.
\end{proposition}

\begin{proof}
    Let $T \in (0,1)$.
    We use the notation $\lesssim$ of \cref{sec:notations} for constants independent of $T$, $u$.
    
    We construct $u^{+i}$, denoted simply by $u$ in the proof (the construction of $u^{-i}$ is identical).
    We intend to use the control $u := \partial_t U^+$, where $U^+$ is constructed in \cref{p:two-signs-null-moments}, with the choice $\eta = T$.
    In particular, since $U^+ \in H^1_0((0,T);\R)$, $u \in \mathcal{H}$.

    By \cref{linearise_explicit}, one has $\psi_1(T) = 0$ if and only if, for all $j \in \N$, $\langle \mu, \varphi_j \rangle \widehat{u}(\lambda_j) = 0$.
    Since $\widehat{u}(\omega) = i \omega \widehat{U^+}(\omega)$ and $\widehat{U^+}(\lambda_j) = 0$ for all $j \in \N$ by \eqref{eq:two-signs-null-moments}, we get $\psi_1(T) = 0$.

    By \eqref{eq:psi2=Qk} in \cref{sec:modulation}, $\langle \psi_2(T), \varphi_k e^{-i \lambda_k T} \rangle = - \frac 12 Q_k (u \rho_k, u \overline{\rho_k})$ where $\rho_k(t):= \exp (i \lambda_k t / 2)$.
    By \cref{lem:Q_k-conj-best}, using that $a_k = 0$ and that, for $u \in \mathcal{H}$, $u_1(T) = 0$ and $\PH u = u$,
    \begin{equation}
        |Q_k(u\rho_k,u\overline{\rho_k}) - Q_k(u,u)| \lesssim T^2 \|u_1\|_{L^2}^2 + \nsob{u_1}{-\nu}^2.
    \end{equation}
    Applying \cref{p:Qk-Q},
    \begin{equation}
        | Q_k(u,u) - Q(u,u) + i \lambda_k K(0) \langle u_1, u_1 \rangle | \lesssim T^2 \| u_1 \|_{L^2}^2 + \nsob{u_1}{-\nu}^2.
    \end{equation}
    By \eqref{eq:two-signs-null-moments}, for all $j \in \N$, $\widehat{u}(\pm \lambda_j) = \pm i \lambda_j \widehat{U^+}(\pm \lambda_j) = 0$.
    By \cref{cor:FQ-L2} and \eqref{eq:Theta=3}, since $\widehat{u_1} = \widehat{U^+}$,
    \begin{equation}
        2 \pi Q(u,u) = - 2 i \int_\R (-a + \Hpv + \Hreg) |\widehat{U^+}|^2.
    \end{equation}
    By \cref{p:Hreg-Hnu}, $| \int_\R \Hreg |\widehat{U^+}|^2 | \lesssim \nsob{U^+}{-\nu}^2$.
    Since $a_k = 0$, recalling \eqref{eq:a-ak}, we obtain:
    \begin{equation}
        \Big| Q_k(u \rho_k, u \overline{\rho_k}) + \frac{2i}{2\pi} \int_\R \Hpv |\widehat{U^+}|^2 \Big|
        \lesssim T^2 \| U^+ \|_{L^2}^2 + \nsob{U^+}{-\nu}^2.
    \end{equation}
    By \cref{p:two-signs-null-moments}, $\|u_1\|_{L^2} = \|U^+\|_{L^2} \lesssim 1$.
    By \eqref{eq:two-signs-null-moments} with $\eta = T$, $\nsob{U^+}{-\nu} \leq T$ and
    \begin{equation}
        \Big| Q_k(u \rho_k, u \overline{\rho_k}) + 2i T \Big|
        \lesssim T^2.
    \end{equation}
    Recalling \eqref{eq:psi2=Qk} and writing $e^{i \lambda_k T} = 1 + O(T)$, there exists $C > 0$ such that
    \begin{equation}
        \Big| \langle \psi_2(T), \varphi_k \rangle - i T \Big|
        \leq C T^2.
    \end{equation}
    Hence, if $T < T^* := \min \{ 1, \frac{1}{2C} \}$, we can rescale the control to obtain \eqref{eq:vect-tgt-Im} and $\|u_1^{\pm i}\|_{L^2} \lesssim 1$.
\end{proof}

\subsection{Time concatenation to recover the complex direction}
\label{sec:concat}

As described in \cref{sec:positive-heuristic}, now that we are able to move in the directions $\pm i \varphi_k$, we use a classical \emph{tangent vector} argument to move in the directions $\pm \varphi_k$ and in fact any $z \varphi_k$ for $z \in \C$.

In \cref{sec:concat,sec:projection,sec:cost}, when there might be an ambiguity on the control or the initial data, we write $\psi(t;u,\psi_0)$ to denote the solution to \eqref{eq:psi} with a control $u$ and an initial data $\psi_0$.
We use similar notations for the linear and quadratic orders $\psi_1$ and $\psi_2$ solutions to \eqref{eq:psi1} and \eqref{eq:psi2}.

\subsubsection{A concatenation lemma}

\begin{lemma}
    \label{lem:concat}
    Let $T, T' > 0$, $u \in L^2((0,T);\R)$, $v \in L^2((0,T');\R)$.
    If $\psi_1(T;u,0) = 0$, then
    \begin{equation}
        \label{eq:concat-psi2}
        \langle \psi_2(T+T';u \diamond v,0), \varphi_k \rangle =
        \langle \psi_2(T;u,0), \varphi_k \rangle e^{-i\lambda_k T'} +
        \langle \psi_2(T';v,0), \varphi_k \rangle,
    \end{equation}
    where $u \diamond v \in L^2((0,T+T'); \R)$ denotes the time-concatenation of $u$ and $v$.
\end{lemma}

\begin{proof}
    For $t \in [0,T]$, by causality, $\psi_1(t;u \diamond v,0) = \psi_1(t;u,0)$ and $\psi_2(t;u \diamond v,0) = \psi_2(t;u,0)$.
    For $t \in [T,T+T']$, the assumption $\psi_1(T;u,0) = 0$ implies that $\psi_1(t;u \diamond v,0) = \psi_1(t-T;v,0)$.
    For $t \in [T,T+T']$, the evolution equation \eqref{eq:psi2} for $\psi_2$ reads
    \begin{equation}
        \begin{split}
            \frac{\dd}{\dd t} \langle \psi_2(t;u \diamond v,0), \varphi_k \rangle
            & = - i \lambda_k \langle \psi_2(t;u \diamond v,0), \varphi_k \rangle
            + i (u \diamond v)(t) \langle \mu \psi_1(t;u \diamond v,0), \varphi_k \rangle
            \\
            & = - i \lambda_k \langle \psi_2(t;u \diamond v,0), \varphi_k \rangle
            + i v(t-T) \langle \mu \psi_1(t-T;v,0), \varphi_k \rangle,
        \end{split}
    \end{equation}
    which implies \eqref{eq:concat-psi2}.
\end{proof}

\subsubsection{Free evolution of a tangent vector}

\begin{proposition}
    \label{p:mvt-Im-to-Re}
    There exist $T^*, C > 0$ such that, for every $T \in (0,T^*)$, there exist controls $u^{\pm 1} \in L^2((0,T);\R)$ with $u_1^{\pm 1}(T) = 0$ satisfying $\| u^{\pm 1}_1 \|_{L^2} \leq C$ such that
    \begin{equation}
        \psi_1(T) = 0, \quad
        |\Im \langle \psi_2(T), \varphi_k \rangle| \leq C T^3, \quad
        \Re \langle \psi_2(T), \varphi_k \rangle = \pm T^2.
    \end{equation}
\end{proposition}

\begin{proof}
    Let $T^*, C_1 > 0$ given by \cref{p:vect-tgt-Im}.
    Given $T \in (0,T^*)$ and $M \geq 1$, we will apply \cref{p:vect-tgt-Im} with a time $T_M := T/(2M)$.
    We construct $u^{+1} := u^{+i} \diamond 0_{[T_M,T-T_M]} \diamond u^{-i}$.
    First, using the property $u^{\pm i}_1(T_M) = 0$,
    we have $\| u^{+1}_1 \|_{L^2}^2 = \| u^{+i}_1 \|_{L^2}^2 + \| u^{-i}_1 \|_{L^2}^2 \leq 2 C_1^2$.
    Second, using \cref{lem:concat}, we have
    \begin{equation}
        \langle \psi_2(T), \varphi_k \rangle
        = \langle \psi_2(T_M;u^{+i}), \varphi_k \rangle e^{-i \lambda_k (T-T_M)} + \langle \psi_2(T_M;u^{-i}), \varphi_k \rangle.
    \end{equation}
    Thus, by \eqref{eq:vect-tgt-Im},
    \begin{equation}
        \left| \Re \langle \psi_2(T), \varphi_k \rangle
        - T_M \sin (\lambda_k (T-T_M)) \right| \leq 2 C_1 T_M^2.
    \end{equation}
    \begin{equation}
        \left| \Im \langle \psi_2(T), \varphi_k \rangle \right| \leq T_M (1 - \cos (\lambda_k (T-T_M))) + C_1 T_M^2 \left| \sin (\lambda_k(T-T_M)) \right|.
    \end{equation}
    Using $|\sin \theta| \leq \theta$, $|\sin \theta - \theta| \leq \theta^2$ and $|1-\cos \theta| \leq \theta^2$, we obtain
    \begin{equation}
        \begin{split}
            \left| \Re \langle \psi_2(T), \varphi_k \rangle
            - \lambda_k T T_M \right|
            & \leq (2 C_1 + \lambda_k) T_M^2 + T_M \lambda_k^2 T^2
            \\ & \leq \lambda_k T T_M \left( \lambda_k T + \frac{2 C_1 + \lambda_k}{2 \lambda_k M} \right)
            \leq \frac{1}{2} \lambda_k T T_M
        \end{split}
    \end{equation}
    provided that $\lambda_k T \leq \frac 1 4$ (which holds, up to reducing $T^*$) and that $(2C_1+\lambda_k) / (2\lambda_k M) \leq \frac 1 4$ (we use $k \neq 0$ so $\lambda_k \neq 0$, and take $M$ large enough) and also
    \begin{equation}
        \left| \Im \langle \psi_2(T), \varphi_k \rangle \right| \leq T_M \lambda_k^2 T^2 + C_1 T_M^2 \lambda_k T \leq (\lambda_k^2 + C_1 \lambda_k) T^3.
    \end{equation}
    Thus, we obtained
    \begin{equation}
        \Re \langle \psi_2(T), \varphi_k \rangle = \frac{\alpha}{2} \frac{\lambda_k}{2M} T^2
    \end{equation}
    for some $\alpha \in \left[ 1,3 \right]$.
    Eventually, we can use the homogeneity of the linear and quadratic systems $\psi_1$ and $\psi_2$ to conclude using a control of the form $\sqrt{(4M)/(\alpha \lambda_k)} u^{+1}$.
\end{proof}

\subsubsection{Complex quadratic motions along $\varphi_k$}

\begin{proposition}
    \label{p:v^z}
    There exists $T^* > 0$ such that, for all $T \in (0,T^*)$, there exists $C_T > 0$ and a continuous map $v : \C \to L^2((0,T);\R)$ such that, for all $z \in \C$,
    \begin{equation}
        v^z_1(T)=0, \qquad
        \psi_1(T;v^z,0)=0, \qquad
        \langle \psi_2(T;v^z,0), \varphi_k \rangle = z
    \end{equation}
    and one has the estimates
    \begin{equation}
        \|v^z\|_{L^2} \leq C_T |z|^{\frac 12}, \qquad
        \|v^z_1\|_{L^2} \lesssim \left( \frac{|\Re(z)|}{T^2} \right)^{\frac 12} + \left( \frac{|\Im(z)|}{T} \right)^{\frac 12}.
    \end{equation}
\end{proposition}

\begin{proof}
    Let $T^* > 0$ given by \cref{p:vect-tgt-Im} and \cref{p:mvt-Im-to-Re}.
    
    Let $T \in (0,T^*)$.
    We define the complex numbers
    \begin{equation}
        \label{def:bpmbpmi}
        \begin{aligned}
            b_{\pm 1}(T) & := \frac{1}{T^2} \langle \psi_2(T;u^{\pm 1},0) , \varphi_k \rangle e^{-i \lambda_k T} = \pm 1 + \underset{T \to 0}{O}(T), \\
            b_{\pm i}(T) & := \frac{1}{T} \langle \psi_2(T;u^{\pm i},0) , \varphi_k \rangle = \pm i + \underset{T \to 0}{O}(T).
        \end{aligned}
    \end{equation}
    For $z \in \C$ there exists a unique $(\alpha,\beta,\epsilon,\epsilon') \in \R_+^2 \times \{\pm 1 \} \times \{ \pm i \}$, such that $z=\alpha b_{\epsilon}(T) + \beta b_{\epsilon'}(T)$ and we define the control $v^z \in L^2((0,2T);\R)$ by
    \begin{equation}
        v^z:=\left( \frac{\alpha}{T^2} \right)^{\frac 12} u^{\epsilon}\, \diamond\, \left( \frac{\beta}{T} \right)^{\frac 12} u^{\epsilon'}.
    \end{equation}
    Note that $z \mapsto v^z$ is continuous because $\alpha, \beta$ vanish when $\eps, \eps'$ switch signs.
    
    Since
    $\psi_1(T;\left( \alpha/T^2 \right)^{\frac 12} u^{\epsilon},0)
    = \left( \alpha/T^2 \right)^{\frac 12}
    \psi_1(T;u^{\epsilon},0)
    =0$, by \cref{lem:concat}, we get
    \begin{equation}
        \begin{aligned}
            \langle \psi_2(2T;v^z) , \varphi_k \rangle
            & =
            \frac{\alpha}{T^2} \langle \psi_2(T;u^{\epsilon}) , \varphi_k \rangle e^{-i\lambda_k T}
            + \frac{\beta}{T} \langle \psi_2(T;u^{\epsilon'}) , \varphi_k \rangle
            \\ & = \alpha b_{\epsilon}(T) + \beta b_{\epsilon'}(T)=z.
        \end{aligned}
    \end{equation}
    We have $\|v^z\|_{L^2}\leq C_T(\alpha^{\frac 12}+\beta^{\frac 12}) \leq C_T'|z|^{\frac 12}$. We deduce from the relation
    $z=\alpha b_{\epsilon} + \beta b_{\epsilon'}$ and \eqref{def:bpmbpmi} that
    $\alpha \lesssim |\Re(z)| + T\alpha+T\beta$ and
    $\beta \lesssim |\Im(z)| + T\alpha+T\beta$.
    Thus, for $T$ small enough,
    $\alpha \lesssim |\Re(z)|+T\beta$ and
    $\beta \lesssim |\Im(z)|+T\alpha$.
    Incorporating the second estimate in the first one gives $\alpha \lesssim |\Re(z)|+T|\Im(z)|$ and similarly $|\beta| \lesssim |\Im(z)|+T|\Re(z)|$.
    Finally,
    \begin{equation}
        \|v^z_1\|_{L^2}
        \lesssim
        \left( \frac{|\Re(z)|+T|\Im(z)|}{T^2} \right)^{\frac 12} +
        \left( \frac{|\Im(z)|+T|\Re(z)|}{T} \right)^{\frac 12}
        \lesssim
        \left( \frac{|\Re(z)|}{T^2} \right)^{\frac 12} +
        \left( \frac{|\Im(z)|}{T} \right)^{\frac 12},
    \end{equation}
    which concludes the proof, up to reducing $T^*$.
\end{proof}

\subsection{Control in projection by the linear test}
\label{sec:projection}

Using assumption \eqref{eq:hyp-c0-k}, we will control all directions except $\varphi_k$ using linear theory.
We prove \cref{p:control-proj-NL} which will be used in the proof of \cref{thm:positive}.

\begin{definition}[Norm $N_T$]
    For $u \in L^2((0,T);\R)$, we define
    \begin{equation}
        N_T(u) := |u_1(T)| + \|u_1\|_{L^2}.
    \end{equation}
\end{definition}

The following elementary result will be used several times.

\begin{lemma}
    \label{Lem:NT}
    Let $T,T'>0$, $u \in L^2((0,T);\R)$ and $v \in L^2((0,T');\R)$.
    If $u_1(T)=0$ then
    \begin{equation}
        N_{T+T'}(u \diamond v) \leq N_T(u) + N_{T'}(v).
    \end{equation}
\end{lemma}

\begin{proof}
    The assumption $u_1(T)=0$ implies $(u \diamond v)_1 =u_1 \diamond v_1$, which gives the conclusion.
\end{proof}

We introduce the vector space
\begin{equation}
    V_k:=\{ \psi \in L^2((0,1);\C) \mid
    \langle \psi,\varphi_k\rangle=\Re \langle \psi,\varphi_0 \rangle=0 \}
\end{equation}
and the orthogonal projection $\mathbb{P}_k:L^2((0,1);\C) \rightarrow V_k$ defined by
\begin{equation}
    \mathbb{P}_k \psi:=\psi-\langle \psi,\varphi_k\rangle \varphi_k - \Re \langle \psi,\varphi_0 \rangle \varphi_0.
\end{equation}

\begin{proposition}
    \label{p:control-proj-NL}
    For $T \in (0,1]$, there exist $\delta_T, C_T>0$ such that, for all $\psi_0 \in H^2_N((0,1);\C) \cap \sphere$, $\widetilde{\psi}^* \in H^2_N((0,1);\C) \cap V_k$ with $\| \psi_0-\varphi_0 \|_{H^2} + \|\widetilde{\psi}^*\|_{H^2}<\delta_T$, there exists $u \in L^2((0,T);\R)$, depending continuously on $\psi_0$ and $\widetilde{\psi}^*$, such that
    \begin{align}
        & \mathbb{P}_k \psi(T;u,\psi_0)=\widetilde{\psi}^*, \label{cible}
        \\
        & \|u\|_{L^2} \leq C_T \left( \| \psi_0-\varphi_0\|_{H^2} +\|\widetilde{\psi}^*\|_{H^2}\right), \label{estimL2}
        \\
        & N_T(u) \leq C_T \left( \| \psi_0-\varphi_0\|_{L^2} + \|\widetilde{\psi}^*\|_{L^2} \right). \label{estimH-1}
    \end{align}
\end{proposition}

\begin{proof}
    Let $T \in (0,1]$.
    In this proof $C_T$ denotes a constant that may change from one expression to another.
    To lighten the notations of functional spaces, the letter $\psi$ implicitly denotes a complex-valued function on $(0,1)$ and the letter $u$ implicitly denotes a real-valued function $(0,T)$.
    In order to define an end-point map between vector spaces (not involving the manifold $\sphere$), we introduce
    \begin{equation}
        \tgt := \{ \psi \in L^2 \mid \Re \langle \psi,\varphi_0 \rangle =0 \},
        \qquad
        \Omega_0:=\{ \widetilde{\psi}_0 \in H^2_N \cap \tgt \mid \| \widetilde{\psi}_0 \|_{L^2}<1 \}
    \end{equation}
    and for any $\widetilde{\psi}_0 \in \Omega_0$, we define
    $\psi_0:= (1-\|\widetilde{\psi}_0\|_{L^2}^2)^{\frac 12} \varphi_0 + \widetilde{\psi}_0 \in \sphere$.
    The end-point map
    \begin{equation}
        \label{Def:FT}
        F_T :
        \begin{cases}
            \Omega_0 \times L^2 & \to \Omega_0 \times (H^2_N \cap V_k) \\
            (\widetilde{\psi}_0, u) & \mapsto (\widetilde{\psi}_0, \mathbb{P}_k \psi(T;u,\psi_0))
        \end{cases}
    \end{equation}
    is of class $C^1$, satisfies $F_T(0,0)=(0,0)$ and
    \begin{equation}
        DF_T(0,0) :
        \begin{cases}
            (H^2_N \cap \tgt) \times L^2 & \to (H^2_N \cap \tgt) \times (H^2_N \cap V_k) \\
            (\widetilde{\psi}_0, u) & \mapsto (\widetilde{\psi}_0, \mathbb{P}_k \psi_1(T;u,\widetilde{\psi}_0)).
        \end{cases}
    \end{equation}

    \step{We prove that $DF_T(0,0)$ has a linear right inverse, denoted $DF_T(0,0)^{-1}$
    such that, for all $(\widetilde{\psi}_0,\widetilde{\psi}^*) \in (H^2_N \cap \tgt) \times (H^2_N \cap V_k)$, the second component $u$ of $DF_T(0,0)^{-1} \cdot (\widetilde{\psi}_0,\widetilde{\psi}^*)$ satisfies}
    \begin{equation}
        \label{eq:proof-L2-NT-0f}
        \| u \|_{L^2} \leq C_T(\| \widetilde{\psi}_0\|_{H^2} + \|\widetilde{\psi}^* \|_{H^2})
        \quad \text{and} \quad
        N_T(u) \leq C_T(\| \widetilde{\psi}_0\|_{L^2} +
        \|\widetilde{\psi}^* \|_{L^2}).
    \end{equation}
    By \eqref{linearise_explicit}, the equality $DF_T(0,0)\cdot(\widetilde{\psi}_0,u)=(\widetilde{\psi}_0,\widetilde{\psi}^*)$ is equivalent to the moment problem
    \begin{equation}
        \forall j \in \N \setminus \{k\}, \qquad\int_0^T u(t) e^{i\lambda_j t} \dd t = \frac{\langle \widetilde{\psi}^*,\varphi_j\rangle e^{i\lambda_j T}-\langle \widetilde{\psi}_0,\varphi_j\rangle}{i \langle \mu , \varphi_j \rangle},
    \end{equation}
    where, for $j = 0$, the right-hand side is a real number since $\widetilde{\psi}_0 \in \tgt$, $\lambda_0 = 0$ and $\widetilde{\psi}^* \in V_k$.
    Thus the conclusion follows from \cref{p:Megane-1} and \eqref{eq:mu-phi_j} (where we also set $\int_0^T u(t) e^{i \lambda_k t} \dd t = 0$).

    \step{We prove equality \eqref{cible}.}
    The vector subspace
    $\mathcal{V}:=\text{Range}(DF_T(0,0)^{-1})$ is closed in $\left( H^2_N \cap \tgt \right) \times L^2((0,T);\R)$ and $DF_T(0,0)$ is an isomorphism from $\mathcal{V}$ to $(H^2_N \cap \tgt) \times (H^2_N \cap V_k)$.
    By the inverse mapping theorem, there exists
    an open neighborhood $\Omega_0'$ of $(0,0)$ in $\mathcal{V}$,
    an open neighborhood $\Omega_T$ of $(0,0)$ in $(H^2_N \cap \tgt) \times (H^2_N \cap V_k)$ such that $F_T$ is a $C^1$ diffeomorphism from $\Omega_0'$ to $\Omega_T$. 
    One may assume $\Omega_T$ of the form
    \begin{equation}
        \Omega_T:=\{ (\widetilde{\psi}_0,\widetilde{\psi}^*) \in (H^2_N \cap \tgt) \times (H^2_N \cap V_k) \mid
        \| \widetilde{\psi}_0 \|_{H^2_N} + \|\widetilde{\psi}^*\|_{H^2_N}<\delta_T \}
    \end{equation}
    where $\delta_T>0$. In particular (even if it means reducing $\delta_T$), for every $(\psi_0,\widetilde{\psi}^*) \in (H^2_N \cap \sphere) \times (H^2_N \cap V_k)$ such that $\|\psi_0-\varphi_0\|_{H^2_N}+\|\widetilde{\psi}^*\|_{H^2_N}<\delta_T$, then $\widetilde{\psi}_0 := \psi_0-\Re \langle \psi_0,\varphi_0\rangle \varphi_0$ satisfies $(\widetilde{\psi}_0,\widetilde{\psi}^*) \in \Omega_T$ and
    $\psi_0 = (1-\|\widetilde{\psi}_0\|_{L^2}^2)^{\frac 12} \varphi_0 + \widetilde{\psi}_0$, thus, if $u \in L^2$ denotes the second component of $F_T^{-1}(\widetilde{\psi}_0,\widetilde{\psi}^*)$ then
    \eqref{cible} holds, and $u$ depends continuously on $\psi_0$ and $\widetilde{\psi}^*$.

    \step[st:DF-1]{We prove \eqref{estimL2}.}
    The map $F_T^{-1}$ is $C^1$ thus locally Lipschitz: there exists $C_T>0$ such that (even if it means reducing $\delta_T$),
    \begin{equation}
        \|u\|_{L^2} \leq C_T ( \| \widetilde{\psi}_0\|_{H^2} + \|\widetilde{\psi}^*\|_{H^2} ) \leq C_T ( \| \psi_0-\varphi_0\|_{H^2} + \|\widetilde{\psi}^*\|_{H^2} ).
    \end{equation}

    \step[st:DF-2]{We prove \eqref{estimH-1}.}
    In the proof of the inverse mapping theorem, the point $(\widetilde{\psi}_0,u)$ is obtained as a fixed point in $\mathcal{V}$
    \begin{equation}
        (\widetilde{\psi}_0,u)=(\widetilde{\psi}_0,u)-DF_T(0,0)^{-1}\cdot( F_T(\widetilde{\psi}_0,u) - (\widetilde{\psi}_0,\widetilde{\psi}^*) ).
    \end{equation}
    This equality can also be written in the following ways
    \begin{equation}
        \begin{split}
            (0,u) & =(0,u)-DF_T(0,0)^{-1} \cdot ( F_T(\widetilde{\psi}_0,u) - (\widetilde{\psi}_0,\widetilde{\psi}^*) ) \\
            & = - DF_T(0,0)^{-1}\cdot( F_T(\widetilde{\psi}_0,u) - DF_T(0,0) \cdot (0,u) - (\widetilde{\psi}_0,\widetilde{\psi}^*) ) \\
            & = - DF_T(0,0)^{-1} \cdot ( 0,
            \mathbb{P}_k \psi(T;u,\psi_0)-\mathbb{P}_k \psi_1(T;u,0)-\widetilde{\psi}^*).
        \end{split}
    \end{equation}
    Thus, by \cref{st:DF-1} and \eqref{eq:quad-rem} (even if it means reducing $\delta_T$ so that $\|u\|_{L^2}<C_T \delta_T<1$ and $T \leq 1$),
    \begin{equation}
        \begin{aligned}
            N_T(u) & \leq C_T \| \mathbb{P}_k \psi(T;u,\psi_0)-\mathbb{P}_k \psi_1(T;u,0)-\widetilde{\psi}^* \|_{L^2}
            \\ & \leq C_T (\|\psi(T;u,\psi_0)-\varphi_0-\psi_1(T;u,0)\|_{L^2} + \|\widetilde{\psi}^*\|_{L^2} )
            \\ & \leq C_T ( \|\psi(T;u,\psi_0)-\psi(T;u,\varphi_0)\|_{L^2} +
            \|\psi(T;u,\varphi_0)-\varphi_0-\psi_1(T;u,0)\|_{L^2} + \|\widetilde{\psi}^*\|_{L^2})
            \\ & \leq C_T ( \|\psi_0-\varphi_0 \|_{L^2} +
            N_T(u)^{\frac 32} + \|\widetilde{\psi}^*\|_{L^2} ).
        \end{aligned}
    \end{equation}
    Even if it means reducing $\delta_T$ (so that $C_T N_T(u)^{\frac 12} \leq C_T' \|u\|_{L^2}^{\frac 12} \leq C_T'' \delta_T^{\frac 12}<1$), we deduce \eqref{estimH-1}.
\end{proof}

\subsection{Small-time nonlinear control and cost estimate}
\label{sec:cost}

We combine the previous arguments to conclude the proof with the Brouwer fixed-point theorem.

\begin{proof}[Proof of \cref{thm:positive}]
    Let $T^* \in (0,1]$ be given by \cref{p:v^z} and $T \in (0,T^*)$.
    For larger $T$, since $\varphi_0$ is an equilibrium, one can start with a null control.
    
    Let $\delta_T':=\frac 12 \delta_{T/2}>0$, where $\delta_{T/2}$ is given by \cref{p:control-proj-NL}.
    Let $\psi^* \in H^2_N \cap \sphere$ be such that
    $\|\psi^*-\varphi_0\|_{H^2_N}< \delta_T'$. 
    Then $\widetilde{\psi}^*:=\mathbb{P}_k(\psi^*)$ belongs to $H^2_N \cap V_k$ and satisfies $\|\widetilde{\psi}^*\|_{H^2}<\delta_T'$.

    \bigskip

    \step[st:pos-1]{Reformulation of the control problem.}
    For $z \in \C$, consider the control $v^z \in L^2((0,\frac{T}{2});\R)$ given by \cref{p:v^z} (for a time $T \gets \frac T 2$ and a target complex number $z \gets z e^{i \lambda_k \frac T 2}$), which satisfies:
    \begin{equation}
        \label{vz_propriete_T/2}
        \begin{aligned}
            & v^z_1(T/2)=0, \quad
            \psi_1(T/2;v^z,0)=0, \quad
            \langle \psi_2(T/2;v^z,0),\varphi_k \rangle = z e^{i \lambda_k \frac T 2},
            \\
            & \|v^z\|_{L^2} \leq C_T |z|^{\frac 12},
            \qquad
            N_{T/2}(v^z) \lesssim \frac{|\Re(z)|^{\frac 12}}{T} + \frac{|\Im(z)|^{\frac 12}}{T^{\frac 12}} .
        \end{aligned}
    \end{equation}
    Thus, recalling estimate \eqref{eq:est-psi-phi0-psi1-H2}, we obtain
    \begin{equation}
        \|\psi(T/2;v^z,\varphi_0)-\varphi_0\|_{H^2}
        =
        \|\psi(T/2;v^z,\varphi_0)-\varphi_0-\psi_1(T/2;v^z,0)\|_{H^2}
        \leq C_T \|v^z\|_{L^2}^2 \leq C_T |z|.
    \end{equation}
    In particular, if $z$ is small enough, then
    $\|\psi(T/2;v^z,\varphi_0)-\varphi_0\|_{H^2_N}<\delta_T'$.
    Thus, by \cref{p:control-proj-NL}, there exists $w^z \in L^2((0,\frac{T}{2});\R)$ such that
    \begin{gather}
        \mathbb{P}_k \psi(T/2;w^z,\psi(T/2;v^z,\varphi_0))=\widetilde{\psi}^*, \\
        \label{eq:wz-prop-2}
        \|w^z\|_{L^2} \leq C_T (|z| + \|\widetilde{\psi}^*\|_{H^2}), \\
        \label{eq:wz-prop-3}
        N_{T/2}(w^z) \leq C_T (|z| + \|\widetilde{\psi}^*\|_{L^2}).
    \end{gather}
    Let $u^z := v^z \diamond w^z \in L^2((0,T),\R)$.
    Then $\mathbb{P}_k \psi(T;u^z,\varphi_0)=\widetilde{\psi}^*$.
    Thus
    \begin{equation}
        \psi(T;u^z,\varphi_0)=\psi^*
        \qquad \Leftrightarrow \qquad
        \langle \psi(T;u^z,\varphi_0) , \varphi_k \rangle = z^* := \langle \psi^*,\varphi_k \rangle.
    \end{equation}

    \step[st:pos-2]{Let $\theta := \frac{1}{16}$. We prove that, for $z$ small enough}
    \begin{equation}
        \label{eq:step2-psizz}
        \left| \langle \psi(T;u^z,\varphi_0) , \varphi_k \rangle
        -z \right| \leq C_T \left( |z|^{1+\theta} + \|\widetilde{\psi}^*\|_{L^2}^2 \right).
    \end{equation}
    Using \cref{Lem:NT}, \eqref{vz_propriete_T/2} and \eqref{eq:wz-prop-3}, we have
    \begin{equation}
        \label{uz_estim}
        N_T(u^z) \leq
        C_T ( |z|^{\frac 12} + \|\widetilde{\psi}^*\|_{L^2}).
    \end{equation}
    Since $\langle \mu, \varphi_k \rangle = 0$, we have $\langle \psi(T;u^z,\varphi_0) , \varphi_k \rangle = A_z + B_z$ where
    \begin{equation}
        A_z:=\langle \psi_2(T;u^z,0) , \varphi_k \rangle,
        \quad B_z:= \langle \psi(T;u^z,\varphi_0) - \varphi_0 - \psi_1(T;u^z,0)-\psi_2(T;u^z,0) , \varphi_k \rangle.
    \end{equation}
    Using \cref{lem:concat} we obtain
    \begin{equation}
        A_z=\langle \psi_2(T/2;v^z,0) , \varphi_k \rangle e^{-i\lambda_k \frac T2} + \langle \psi_2(T/2;w^z,0) , \varphi_k \rangle.
    \end{equation}
    Thus, using \eqref{vz_propriete_T/2}, \eqref{eq:psi2-coerc} and \eqref{eq:wz-prop-3},
    \begin{equation}
        |A_z-z|
        \leq C N_T(w^z)^2
        \leq C_T(|z|^2+\|\widetilde{\psi}^*\|_{L^2}^2).
    \end{equation}
    Using \eqref{eq:cubic-rem} of \cref{cor:cubic-rem} (see \cref{sec:estimates-cub}) since $T \leq 1$ and \eqref{uz_estim} we obtain 
    \begin{equation}
        |B_z| \leq C_T N_T(u^z)^{2+2\theta} \leq C_T (|z|^{1+\theta} + \|\widetilde{\psi}^*\|_{L^2}^{2+2\theta} )
    \end{equation}
    which gives the conclusion.

    \step[st:pos-3]{Brouwer fixed-point theorem.}
    Let $C_T$ be given by \cref{st:pos-2} and $r>0$ such that $C_T r^{1+\theta}<\frac{r}{3}$.
    We may assume that $\|\psi^*-\varphi_0\|_{H^2}$ is small enough so that $C_T \|\widetilde{\psi}^*\|_{L^2}^2<\frac r3$ and $|z^*|<\frac r3$.
    The map
    \begin{equation}
        G :
        \begin{cases}
            \overline{B}_{\C}(0,r) \to \C \\
            z \mapsto z-\langle \psi(T;u^z,\varphi_0),\varphi_k \rangle + z^*
        \end{cases}
    \end{equation}
    is continuous (by continuity of $z \mapsto v^z$ from \cref{p:v^z} and $z \mapsto w^z$ from \cref{p:control-proj-NL}, and continuous dependence of the state on the control and initial data) and takes values in $\overline{B}_{\C}(0,r)$ because
    \begin{equation}
        |G(z)| \leq | z-\langle \psi(T;u^z,\varphi_0),\varphi_k \rangle| + |z^*| \leq C_T ( r^{1+\theta} + \|\widetilde{\psi}^*\|_{L^2}^2)+\frac r3 < r
    \end{equation}
    By the Brouwer fixed-point theorem, there exists $z \in \overline{B}_{\C}(0,r)$ such that $G(z)=z$.
    By \cref{st:pos-1}, this implies that $\psi(T;u^z,\varphi_0)=\psi^*$.

    \step{Proof of the estimates.}
    At the fixed point, we have
    \begin{equation}
        \label{fixed_point}
        z=z^*+z-\langle \psi(T;u^z,\varphi_0),\varphi_k \rangle.
    \end{equation}
    Hence \eqref{eq:step2-psizz} implies that
    \begin{equation}
        \begin{split}
            |z - z^*| 
            & \leq C_T |z|^{1+\theta} + C_T \|\widetilde{\psi}^*\|_{L^2}^2 \\ 
            & \leq C_T |z - z^*|^{1+\theta} + C_T |z^*|^{1+\theta} + C_T \|\widetilde{\psi}^*\|_{L^2}^2 \\
            & \lesssim C_T (2r)^\theta |z - z^*| + C_T |z^*|^{1+\theta} + C_T \|\widetilde{\psi}^*\|_{L^2}^2.
        \end{split}
    \end{equation}
    Up to reducing $r$ so that $(2r)^\theta C_T < 1$, we obtain
    \begin{equation}
        \label{eq:zz*}
        |z - z^*| \lesssim C_T ( |z^*|^{1+\theta} + \|\widetilde{\psi}^*\|_{L^2}^2 ).
    \end{equation}
    Together with \eqref{vz_propriete_T/2} and \eqref{eq:wz-prop-2}, this implies
    \begin{equation}
        \label{eq:cost-L2-precise}
        \|u^z\|_{L^2}
        \leq C_T(|z^*|^{\frac 12} + \|\widetilde{\psi}^*\|_{H^2}).
    \end{equation}
    Using \cref{Lem:NT}, \eqref{vz_propriete_T/2}, \eqref{eq:wz-prop-3} and \eqref{eq:zz*}, we obtain
    \begin{equation}
        \label{eq:cost-H1-precise}
        \begin{split}
            N_T(u^z) & \lesssim \frac{|\Re z |^{\frac 12}}{T} + \frac{|\Im z|^{\frac 12}}{T^{\frac 12}} +C_T ( |z| + \|\widetilde{\psi}^*\|_{L^2}) \\
            & \lesssim \frac{|\Re z^* |^{\frac 12}}{T} + \frac{|\Im z^*|^{\frac 12}}{T^{\frac 12}} +C_T ( |z^*|^{\frac12 + \frac{\theta}{2}} + \|\widetilde{\psi}^*\|_{L^2}) \\
            & \lesssim \frac{|\Re z^* |^{\frac 12}}{T} + \frac{|\Im z^*|^{\frac 12}}{T^{\frac 12}} +C_T \|\widetilde{\psi}^*\|_{L^2}
        \end{split}
    \end{equation}
    for $z^*$ small enough.
    Eventually \eqref{eq:cost-L2-precise} implies \eqref{eq:cost-L2} and \eqref{eq:cost-H1-precise} implies \eqref{eq:cost-H-1}, concluding the proof of \cref{thm:positive}.
\end{proof}

\section{Multiplying kernels by regular functions}
\label{sec:multiply}

In this paragraph, we investigate the following question, in a general abstract setting.
Let $K$ be the kernel of an integral operator, of which we have estimated the operator norm.
Let $\chi$ be a regular function vanishing on the diagonal.
What estimates can we obtain on the operator norm of the integral operator associated with $\chi K$?

In our context, this will enable us in \cref{p:Qk-Q} to transfer the results obtained in \cref{sec:quad} on the reference quadratic form $Q$ of \cref{eq:Q_0} to the quadratic form $Q_k$ of \eqref{eq:Q_k}.

For $T > 0$, we use the following notations:
\begin{equation}
    \label{eq:I-I2-Delta}
    I := (0,T), \quad
    I^2 := (0,T)^2, \quad
    \Delta_{\pm} := \{ (s,t) \in I^2 \mid \pm (t-s) > 0 \}.
\end{equation}

\subsection{Multiplying a bounded operator on \texorpdfstring{$L^2$}{L2}}
\label{sec:mul-L2}

Given a kernel $K$ of a bounded integral operator on $L^2((0,T);\C)$, and a regular function $\chi$ vanishing on the diagonal, we estimate the operator norm of the integral operator associated with $\chi K$.

\begin{proposition}
    \label{lem:mul-chi-L2}
    Let $T > 0$, $K \in L^2(I^2;\C)$ and $C_K$ such that, for all $u,v\in L^2((0,T);\C)$,
    \begin{equation}
        \label{eq:chi-hypo}
        \left| \int_{I^2} K(s,t) u(s) \bar{v}(t) \dd s \dd t \right|
        \leq C_K \| u\|_{L^2} \| v \|_{L^2}.
    \end{equation}
    Let $\chi \in W^{1,\infty}(I^2;\C)$ vanishing on the diagonal.
    Then, for all $u, v \in L^2((0,T);\C)$,
    \begin{equation}
        \label{eq:chi-ccl}
        \left| \int_{I^2} \chi(s,t) K(s,t) u(s) \bar{v}(t) \dd s \dd t \right|
        \leq T C_K \left( \| \partial_1 \chi \|_{L^\infty} + \| \partial_2 \chi \|_{L^\infty} \right) \| u\|_{L^2} \| v \|_{L^2}.
    \end{equation}
\end{proposition}

\begin{proof}
    Since $W^{1,\infty}$ functions are continuous, the assumption that $\chi$ vanishes on the diagonal makes sense.
    Since $\chi \in L^\infty(I^2;\C)$, $K \in L^2(I^2;\C)$ and $u, v \in L^2(I;\C)$, $\chi(s,t) K(s,t) u(s) \bar{v}(t) \in L^1(I^2;\C)$.
    Thus the contributions of $\Delta_{+}$ and $\Delta_{-}$ can be estimated separately; we treat $\Delta_{+}$.
    Using $\chi(s,s) = 0$, we write
    \begin{equation}
        \label{eq:cake}
        \mathbbm{1}_{\Delta_+}(s,t) \chi(s,t)
        = \mathbbm{1}_{s < t} \int_s^t \partial_2 \chi(s,\theta) \dd \theta
        = \int_0^T \mathbbm{1}_{(0,\theta)}(s) \mathbbm{1}_{(\theta,T)}(t) \partial_2 \chi(s,\theta) \dd \theta.
    \end{equation}
    Introduce, for $\theta \in (0,T)$, the projections for $\phi \in L^2((0,T);\C)$,
    \begin{equation}
        \label{eq:proj-PQ}
        (P_\theta \phi)(s) := \mathbbm{1}_{(0,\theta)}(s) \phi(s)
        \quad \text{and} \quad
        (Q_\theta \phi)(t) := \mathbbm{1}_{(\theta,T)}(t) \phi(t).
    \end{equation}
    Then, using \eqref{eq:cake} and Fubini's theorem,
    \begin{equation}
        \label{eq:chi-caked}
        \int_{\Delta_+} \chi(s,t) K(s,t) u(s) \bar{v}(t) \dd s \dd t
        = \int_0^T \int_{I^2} K(s,t) (P_\theta (\partial_2 \chi(\cdot,\theta) u))(s) (Q_\theta \bar{v})(t) \dd s \dd t \dd \theta
    \end{equation}
    For each fixed $\theta \in (0,T)$, assumption \eqref{eq:chi-hypo} yields
    \begin{equation}
        \label{eq:cake-local}
        \begin{split}
            \Bigg| \int_{I^2} K(s,t) (P_\theta (\partial_2 \chi(\cdot,\theta) u))(s) (Q_\theta \bar{v})(t) \dd s \dd t \Bigg|
            & \leq C_K \| P_\theta (\partial_2 \chi(\cdot,\theta) u) \|_{L^2} \| Q_\theta \bar{v} \|_{L^2} \\
            & \leq C_K \| \partial_2 \chi \|_{L^\infty} \| u \|_{L^2} \| \bar{v} \|_{L^2}.
        \end{split}
    \end{equation}
    Substituting estimate \eqref{eq:cake-local} in \eqref{eq:chi-caked} and integrating over $\theta \in (0,T)$ bounds the integral over $\Delta_+$ by $T C_K \| \partial_2 \chi \|_{L^\infty} \| u \|_{L^2} \| v \|_{L^2}$.
    By symmetry, one has an analogous bound with $\|\partial_1 \chi\|_{L^\infty}$ in place of $\|\partial_2 \chi\|_{L^\infty}$ for the integral over $\Delta_-$.
    Combining the two estimates proves \eqref{eq:chi-ccl}.
\end{proof}

\begin{remark}
    Using Cauchy--Schwarz, the hypothesis $K \in L^2(I^2;\C)$ immediately allows to bound left-hand side of \eqref{eq:chi-ccl} by $\| \chi \|_{L^\infty} \| K \|_{L^2} \| u \|_{L^2} \| v \|_{L^2}$, where $\|\chi\|_{L^\infty} \leq T \| \chi\|_{W^{1,\infty}}$ since $\chi$ vanishes on the diagonal.
    The interest of the bound \eqref{eq:chi-ccl} is to involve the operator norm $C_K$ of assumption~\eqref{eq:chi-hypo}, instead of the Hilbert--Schmidt norm $\|K\|_{L^2}$, which can be much larger.
\end{remark}

\begin{remark}
    We choose to present \cref{lem:mul-chi-L2} with the assumption $K \in L^2(I^2;\C)$ in order to be able to manipulate usual Lebesgue integrals.
    However, one could also adopt the following abstract viewpoint.
    Take $K$ a continuous linear form of norm $C_K$ on the projective tensor product $E := L^2(I;\C) \otimes L^2(I;\C)$.
    Then the claim is that $\chi K$ is a continuous linear form of norm at most $T \| \nabla \chi \|_{L^\infty} C_K$ on $E$.
    By duality, this amounts to proving that multiplication by $\chi$ is a linear bounded operator on the completion of $E$, which can be done using the same integral representation of $\chi$ which allows to separate variables.
    An advantage of this abstract viewpoint is that we don't need to assume that $K \in L^2(I^2;\C)$, which circumvents the necessity of a regularization argument in the proof of \cref{p:R-chi} below.
\end{remark}

\subsection{Multiplying a bounded operator on a fractional Sobolev space}
\label{sec:mul-Hnu}

We generalize \cref{lem:mul-chi-L2} to continuous bilinear forms on $\sobC{\nu} \times \sobC{\nu}$.

\begin{proposition}
    \label{lem:mul-chi-Hnu}
    Let $\nu \in (-\frac 12,\frac12)$.
    There exists a constant $C_\nu > 0$ such that the following property holds.
    Let $T > 0$, $K \in L^2(I^2;\C)$ and $C_K$ such that, for all $u,v\in L^2 \cap \sobC{\nu}$,
    \begin{equation}
        \label{eq:mul-chi-Hnu-hyp}
        \left| \int_{I^2} K(s,t) u(s) \bar{v}(t) \dd s \dd t \right|
        \leq C_K \nsob{u}{\nu} \nsob{v}{\nu}.
    \end{equation}
    Let $\chi \in W^{1,\infty}(I^2;\C)$ vanishing on the diagonal.
    Assume that $\chi \in W^{2,\infty}(\Delta_\pm;\C)$ (but not across the diagonal).
    Then, for all $u, v \in L^2 \cap \sobC{\nu}$,
    \begin{equation}
        \label{eq:mul-chi-Hnu-ccl}
        \left| \int_{I^2} \chi(s,t) K(s,t) u(s) \bar{v}(t) \dd s \dd t \right|
        \leq T C_K C_\nu \left( M_\chi + T \| \partial_{12} \chi \|_{L^\infty} \right) \nsob{u}{\nu} \nsob{v}{\nu}
    \end{equation}
    where $M_\chi := \sup_{t \in [0,T]} |\partial_2 \chi(t-0,t)| + \sup_{t\in [0,T]} |\partial_1 \chi (t+0,t)|$.
\end{proposition}

\begin{proof}
    We argue as in the proof of \cref{lem:mul-chi-L2}.
    Note that, since $\chi$ has $W^{2,\infty}$ regularity within each triangle, $\partial_1 \chi$ and $\partial_2 \chi$ have traces on each side of the diagonal.

    The key additional argument is that there exists a constant $c_\nu > 0$ such that, for all $0 < \theta < T$ and $\phi \in L^2 \cap \sobC{\nu}$, $\nsob{P_\theta \phi}{\nu} \leq c_\nu \nsob{\phi}{\nu}$ and $\nsob{Q_\theta \phi}{\nu} \leq c_\nu \nsob{\phi}{\nu}$, where $P_\theta$ and $Q_\theta$ are defined in \eqref{eq:proj-PQ}.
    Indeed, by \eqref{eq:Hnu} and \cref{lem:mul-1-Hnu}, denoting by $\widetilde{\phi}$ the extension of $\phi$ by $0$ on $\R \setminus (0,T)$, $\nsob{P_\theta \phi}{\nu} = \| P_\theta \widetilde{\phi} \|_{H^\nu(\R)} \leq c_\nu \| \widetilde{\phi} \|_{H^\nu(\R)} = c_\nu \nsob{\phi}{\nu}$.

    To circumvent the lack of regularity of $\chi$ across the diagonal, we write, for $0 < s < \theta < T$,
    \begin{equation}
        \partial_2 \chi(s,\theta) = \partial_2 \chi(\theta-0,\theta) - \int_0^\theta \partial_{12} \chi(\tau,\theta) \mathbbm{1}_{(0,\tau)}(s) \dd \tau.
    \end{equation}
    Thus, starting from \eqref{eq:proj-PQ}, we obtain
    \begin{equation}
        \begin{split}
            \int_{\Delta_+} \chi(s,t) K(s,t) & u(s) \bar{v}(t) \dd s \dd t
            = \int_0^T \partial_2 \chi(\theta-0,\theta) \int_{I^2} K(s,t) (P_\theta u)(s) (Q_\theta \bar{v})(t) \dd s \dd t \dd \theta \\
            & - \int_0^T \int_0^\theta \partial_{12}\chi(\tau,\theta) \int_{I^2} K(s,t) (P_{\min (\tau, \theta)} u)(s) (Q_\theta \bar{v})(t) \dd s \dd t \dd \tau \dd \theta
        \end{split}
    \end{equation}
    Using assumption \eqref{eq:mul-chi-Hnu-hyp}, layer by layer, for each fixed $\theta$ and $\tau$, we obtain
    \begin{equation}
        \Bigg| \int_{\Delta_+} \chi(s,t) K(s,t) u(s) \bar{v}(t) \dd s \dd t \Bigg| \leq
        c_\nu^2 \nsob{u}{\nu} \nsob{v}{\nu} \left( T M_\chi + T^2 \| \partial_{12} \chi \|_{L^\infty} \right).
    \end{equation}
    Together with a similar estimate on $\Delta_-$, this gives the conclusion with $C_\nu := 2 c_\nu^2$.
\end{proof}

\begin{corollary}
    \label{lem:mul-chi-mixte}
    Let $\nu \in (-\frac 12,\frac12)$.
    There exists a constant $C_\nu > 0$ such that the following property holds.
    Let $T \in (0,1]$, $K \in L^2(I^2;\C)$ and $C_K,D_K > 0$ such that, for all $u, v \in L^2 \cap \sobC{\nu}$,
    \begin{equation}
        \label{eq:mul-chi-hyp}
        \left| \int_{I^2} K(s,t) u(s) \bar{v}(t) \dd s \dd t \right|
        \leq C_K \| u \|_{L^2} \| v \|_{L^2} + D_K \nsob{u}{\nu} \nsob{v}{\nu}.
    \end{equation}
    Let $\chi \in W^{1,\infty}(I^2;\C)$ vanishing on the diagonal, with $\chi \in W^{2,\infty}(\Delta_\pm;\C)$ (but not across the diagonal).
    There exists $C_\chi > 0$ depending only on its $W^{2,\infty}$ norm such that, for all $u, v \in L^2 \cap \sobC{\nu}$,
    \begin{equation}
        \label{eq:mul-chi-ccl}
        \Bigg| \int_{I^2} \chi(s,t) K(s,t) u(s) \bar{v}(t) \dd s \dd t
        \Bigg|
        \leq T C_\nu C_\chi \left( C_K \| u \|_{L^2} \| v \|_{L^2} + D_K \nsob{u}{\nu} \nsob{v}{\nu} \right).
    \end{equation}
\end{corollary}

\begin{proof}
    This result is easily obtained by working as in the proofs of \cref{lem:mul-chi-L2} and \cref{lem:mul-chi-Hnu}, using the mixed continuity estimate \eqref{eq:mul-chi-hyp} on each layer, to obtain estimate \eqref{eq:mul-chi-ccl}.
    The assumption $T \leq 1$ is used to bound the constant $M_\chi + T \| \partial_{12} \chi \|_{L^\infty}$ appearing in \eqref{eq:mul-chi-Hnu-ccl}.
\end{proof}

\subsection{Multiplying a bounded operator acting on primitives}
\label{sec:mul-H-1}

In this paragraph, we only consider controls $u,v \in \mathcal{H}$ (see \eqref{eq:cH}), i.e.\ such that $u_1(T) = v_1(T) = 0$, where $u_1$ and $v_1$ are the primitives of $u$ and $v$ vanishing at $0$.

We now handle a kernel $K$ for which the natural continuity assumption is expressed at the level of the primitives $u_1$ and $v_1$ of the controls.
We prove a counterpart of \cref{lem:mul-chi-mixte} in this setting.

\begin{proposition}
    \label{p:R-chi}
    Let $\nu \in (-\frac 12,\frac12)$.
    There exists a constant $C_\nu > 0$ such that the following property holds.
    Let $T \in (0,1]$, $K \in H^1(I^2;\C)$ and $C_K,D_K > 0$ such that, for all $u, v \in \mathcal{H}$,
    \begin{equation}
        \label{eq:R-chi-hyp}
        \left| \int_{I^2} K(s,t) u(s) \bar{v}(t) \dd s \dd t \right|
        \leq C_K \| u_1 \|_{L^2} \| v_1 \|_{L^2} + D_K \nsob{u_1}{\nu} \nsob{v_1}{\nu}.
    \end{equation}
    Let $\chi \in W^{1,\infty}(I^2;\C)$ vanishing on the diagonal, with $\chi \in W^{2,\infty}(\Delta_\pm;\C)$ (but not across the diagonal).
    There exists $C_\chi > 0$ depending only on its $W^{2,\infty}$ norm such that, for all $u, v \in \mathcal{H}$,
    \begin{equation}
        \label{eq:R-chi-ccl}
        \begin{split}
            \Bigg| \int_{I^2} \chi(s,t) K(s,t) u(s) \bar{v}(t) \dd s \dd t
            & - \int_I w u_1 \bar{v}_1
            - \int_{\Delta_+\cup\Delta_-} \Upsilon(s,t) u_1(s) \bar{v}_1(t) \dd s \dd t \Bigg| \\
            & \leq T C_\nu C_\chi \left( C_K \| u_1 \|_{L^2} \| v_1 \|_{L^2} + D_K \nsob{u_1}{\nu} \nsob{v_1}{\nu} \right).
        \end{split}
    \end{equation}
    where $w(s) := (\partial_1 \chi(s,s+0)-\partial_1 \chi(s,s-0)) K(s,s)$ and $\Upsilon := K \partial_{21}\chi + \partial_2 \chi \partial_1 K + \partial_1 \chi \partial_2 K$.
\end{proposition}

\begin{proof}
    \step{Preliminary comments on \eqref{eq:R-chi-ccl}.}
    All terms in the left-hand side of \eqref{eq:R-chi-ccl} make sense for $K \in H^1(I^2;\C)$.
    Indeed, let $u,v \in \mathcal{H}$.
    First, $\int_{I^2} \chi K (u \otimes \bar{v})$ makes sense since $\chi \in L^\infty(I^2)$, $K \in L^2(I^2)$ and $u \otimes v \in L^2(I^2)$.
    Second, $\partial_1 \chi$ and $K$ are in $H^1$ of each triangle, so have $L^2$ traces on each side of the diagonal, so $w \in L^1(I)$, and $u_1 \bar{v}_1 \in L^\infty(I)$.
    Third, since $u_1 \otimes \bar{v}_1 \in L^\infty(I^2)$, it suffices to check that $\Upsilon \in L^1(I^2)$, which is the case.
    Thus, the left-hand side of \eqref{eq:R-chi-ccl} in fact depends continuously on $K \in H^1(I^2;\C)$, and we can use a regularization argument.

    \step{Case of a regular kernel $K \in H^2(I^2;\C)$.}
    Then $\chi K \in C^0([0,T]^2;\C)$ and $\chi K$ is $H^2$ in $\Delta_\pm$.
    Thus, we can apply \cref{lem:IPP-Armand} to the kernel $\chi K$.
    Fix $u,v \in \mathcal{H}$.
    Using that $u_1(T) = v_1(T) = 0$, equality \eqref{eq:IPP-Armand} reads
    \begin{equation}
        \label{eq:R-chi-ipp}
        \begin{split}
            \int_{I^2} \chi(s,t) K(s,t) u(s) \bar{v}(t) \dd s \dd t = \int_{I} w u_1 \bar{v}_1 & + \int_{\Delta_+ \cup \Delta_-} \Upsilon(s,t) u_1(s) \bar{v}_1(t) \dd s \dd t
            \\ & + \int_{\Delta_+ \cup \Delta_-} \chi(s,t) \partial_{21}K(s,t) u_1(s) \bar{v}_1(t) \dd s \dd t,
        \end{split}
    \end{equation}
    where $w(s) = \partial_1 (\chi K)(s,s+0) - \partial_1(\chi K)(s,s-0) = (\partial_1 \chi(s,s+0)-\partial_1 \chi(s,s-0)) K(s,s)$ using that $\chi$ vanishes on the diagonal, and we used that $\partial_{21} (\chi K) = \Upsilon + \chi \partial_{21} K$.

    Moreover, for any $f,g \in C^\infty_c((0,T);\R)$, integrating by parts,
    \begin{equation}
        \begin{split}
            \Bigg| \int_{I^2} \partial_{21} K(s,t) f(s) \bar{g}(t) \dd s \dd t \Bigg|
            & =
            \Bigg| \int_{I^2} K(s,t) f'(s) \bar{g}'(t) \dd s \dd t \Bigg| \\
            & \leq C_K \| f \|_{L^2} \| g \|_{L^2} + D_K \nsob{f}{\nu} \nsob{g}{\nu}
        \end{split}
    \end{equation}
    by hypothesis \eqref{eq:R-chi-hyp}.
    By density, this estimate is preserved for $f,g \in L^2 \cap \sobC{\nu}$
    and the kernel $\partial_{21} K \in L^2$ satisfies the assumption \eqref{eq:mul-chi-hyp}.
    Thus, by \cref{lem:mul-chi-mixte} (using $T \leq 1$),
    \begin{equation}
        \begin{split}
            \Bigg| \int_{\Delta_+ \cup \Delta_-} \chi(s,t) \partial_{21} K(s,t) & u_1(s) \bar{v}_1(t) \dd s \dd t \Bigg| \\
            & \leq T C_\nu C_\chi \left( C_K \| u_1 \|_{L^2} \| v_1 \|_{L^2} + D_K \nsob{u_1}{\nu} \nsob{v_1}{\nu} \right),
        \end{split}
    \end{equation}
    which, together with \eqref{eq:R-chi-ipp}, yields the conclusion \eqref{eq:R-chi-ccl}.

    \step{Regularization argument.}
    We now construct a sequence $K_N \in H^2(I^2;\C)$ such that $K_N \to K$ in $H^1(I^2;\C)$ and, for each $N \in \N^*$, $K_N$ satisfies assumption \eqref{eq:R-chi-hyp} with constants $(C_K, D_K) \gets (C_K,c_\nu^2 D_K)$ (see below).
    By the second step, \eqref{eq:R-chi-ccl} holds for each $K_N$.
    By the first step, the left-hand side of \eqref{eq:R-chi-ccl} corresponding to $K_N$ tends to the one corresponding to $K$.

    Consider the 1D Neumann Laplacian on $L^2(I)$.
    It possesses an orthonormal basis of eigenfunctions $(\phi_k)_{k=0}^\infty$ in $L^2(I)$, given by $\phi_0(t) = 1/\sqrt{T}$ and $\phi_k(t) = \sqrt{2/T} \cos(k\pi t/T)$ for $k \geq 1$ with eigenvalues $\beta_k^2 = (k\pi/T)^2 \geq 0$.
    For $N \in \N$, let $\mathbb{P}_N$ be the orthogonal projector in $L^2(I)$ onto the subspace spanned by $\{\phi_0, \dotsc, \phi_N\}$.
    We define
    \begin{equation}
        K_N := (\mathbb{P}_N \otimes \mathbb{P}_N) K \in C^\infty([0,T]^2;\C).
    \end{equation}
    Since $K \in H^1(I^2;\C)$, we have $K_N \to K$ in $H^1(I^2;\C)$.
    Moreover, for all $u,v\in \mathcal{H}$,
    \begin{equation}
        \begin{split}
            \bigg| \int_{I^2} K_N(s,t) u(s) \bar{v}(t) & \dd s \dd t \bigg|
            = \bigg|\int_{I^2} K(s,t) (\mathbb{P}_N u)(s) (\overline{\mathbb{P}_N v})(t) \dd s \dd t \bigg|
            \\ & \leq C_K \| (\mathbb{P}_N u)_1 \|_{L^2} \| (\mathbb{P}_N v)_1 \|_{L^2} + D_K \nsob{(\mathbb{P}_N u)_1}{\nu} \nsob{(\mathbb{P}_N v)_1}{\nu},
        \end{split}
    \end{equation}
    by hypothesis \eqref{eq:R-chi-hyp} since $\mathbb{P}_N u, \mathbb{P}_N v \in \mathcal{H}$.
    For $u \in \mathcal{H}$, $\| (\mathbb{P}_N u)_1 \|_{L^2} \leq \| u_1 \|_{L^2}$ and by \cref{lem:P_N-H_nu}, $\nsob{(\mathbb{P}_N u)_1}{\nu} \leq c_\nu \nsob{u_1}{\nu}$, so all approximate kernels $K_N$ satisfy \eqref{eq:R-chi-hyp} with the same constants $(C_K,D_K) \gets (C_K, c_\nu^2 D_K)$ which are independent of $N$.
\end{proof}

\begin{lemma}
    \label{lem:P_N-H_nu}
    Let $\nu \in (-\frac12,\frac12)$.
    There exists $c_\nu > 0$ such that, with the notations of the proof of \cref{p:R-chi}, for every $N \in \N$ and $u \in \mathcal{H}$, $\nsob{(\mathbb{P}_N u)_1}{\nu} \leq c_\nu \nsob{u_1}{\nu}$.
\end{lemma}

\begin{proof}
    For $u \in \mathcal{H}$, write $u = \sum_{k\geq1} c_k \phi_k$ ($\langle u, \phi_0 \rangle = 0$ since $u \in \mathcal{H}$).
    Thus $u_1 = \sum_{k\geq1} (c_k/\beta_k) e_k$ where $e_k(t) := \sqrt{2/T} \sin (\beta_kt)$ and $(\mathbb{P}_N u)_1 = \sum_{k=1}^N (c_k/\beta_k) e_k$.
    Now we use that the Sobolev norm $\nsob{\cdot}{\nu}$ can also be computed spectrally (see \cref{sec:spectral}).
    Hence, applying \cref{lem:H-nu<=Hspec} and \cref{lem:Hspec<=Hnu}
    \begin{equation}
        \nsob{(\mathbb{P}_N u)_1}{\nu}^2
        \leq c_\nu \sum_{k=1}^N \langle \beta_k \rangle^{2\nu} \left|\frac{c_k}{\beta_k}\right|^2
        \leq c_\nu \sum_{k=1}^\infty \langle \beta_k \rangle^{2\nu} \left|\frac{c_k}{\beta_k}\right|^2
        \leq c_\nu^2 \nsob{u_1}{\nu}^2,
    \end{equation}
    which concludes the proof.
\end{proof}

\subsection{Application to our case}
\label{sec:mul-our-case}

In this paragraph, we only consider controls $u,v \in \mathcal{H}$ (see \eqref{eq:cH}), i.e.\ such that $u_1(T) = v_1(T) = 0$, where $u_1$ and $v_1$ are the primitives of $u$ and $v$ vanishing at $0$.

\begin{lemma}
    \label{lem:kernel-abc}
    Let $\alpha, \beta, \gamma \in \C$.
    There exists $C > 0$ such that, for all $T \in (0,1]$ and $u,v \in \mathcal{H}$,
    \begin{equation}
        \Bigg| \int_{I^2} \big( e^{i \gamma |t-s|} - 1 \big) \big(\alpha |t-s| + \beta \big) u(s) \overline{v}(t) \dd s \dd t + 2 i \gamma \beta \langle u_1, v_1 \rangle \Bigg| \leq C \nsob{u_1}{-\frac12}\nsob{v_1}{-\frac12}.
    \end{equation}
\end{lemma}

\begin{proof}
    We apply \cref{lem:IPP-Armand} to the kernel $K(s,t) := f(|t-s|)$ where $f(\sigma) := ( e^{i \gamma \sigma} - 1 ) (\alpha \sigma + \beta)$.
    By \cref{lem:IPP-Armand}, recalling that $u_1(T) = v_1(T) = 0$ since $u,v \in \mathcal{H}$,
    \begin{equation}
        \int_{I^2} K(s,t) u(s) \overline{v}(t) \dd s \dd t = - 2 f'(0) \langle u_1, v_1 \rangle - \int_{\Delta_- \cup \Delta_+} f''(|t-s|) u_1(s) \overline{v}_1(t) \dd s \dd t.
    \end{equation}
    We have $f'(0) = i \gamma \beta$.
    Since $f''(|\cdot|) \in C^{0,1}_{\mathrm{loc}}(\R;\C)$, \cref{lem:convol-holder} implies that
    \begin{equation}
        \Bigg| \int_{I^2} f''(|t-s|) u_1(s) \overline{v}_1(t) \dd s \dd t \Bigg| \leq C \nsob{u_1}{-\frac12}\nsob{v_1}{-\frac12},
    \end{equation}
    which proves the claim.
\end{proof}

\begin{proposition}
    \label{p:Qk-Q}
    Let $\nu \in [0,\frac14)$.
    There exists $C > 0$ such that, for all $T \in (0,1]$ and $u,v \in \mathcal{H}$,
    \begin{equation}
        \label{eq:Qk-Q}
        | Q_k(u,v) - Q(u,v) + i \lambda_k K(0) \langle u_1, v_1 \rangle | \leq C T^2 \| u_1 \|_{L^2} \| v_1 \|_{L^2} + C \nsob{u_1}{-\nu} \nsob{v_1}{-\nu}.
    \end{equation}
\end{proposition}

\begin{proof}
    For $u, v \in \mathcal{H}$, by definition \eqref{eq:Q_k} of $Q_k$ and \eqref{eq:Q_0} of $Q$,
    \begin{equation}
        Q_k(u,v) - Q(u,v) = \int_{I^2} \left(e^{i\frac{\lambda_k}{2}|t-s|} - 1\right) K(t-s) u(s) \overline{v}(t) \dd s \dd t.
    \end{equation}
    To apply \cref{p:R-chi}, we need to replace $K$ (up to a remainder term smaller than the right-hand side of \eqref{eq:Qk-Q}) with a kernel that satisfies a continuity estimate of the form \eqref{eq:R-chi-hyp}.
    By \cref{lem:kernel-abc}, there exists $C_1 > 0$ such that, for all $T \in (0,1]$,
    \begin{equation}
        \label{eq:Qk-Q-1}
        \begin{split}
            \Bigg| \int_{I^2} \left( e^{i \frac{\lambda_k}{2} |t-s|} - 1 \right) \left( - i a |t-s| + K(0) \right) u(s) \overline{v}(t) \dd s \dd t & + i \lambda_k K(0) \langle u_1, v_1 \rangle \Bigg| \\
            & \leq C_1 \nsob{u_1}{-\frac12} \nsob{v_1}{-\frac12}.
        \end{split}
    \end{equation}
    Let $R(s,t) := K(t-s) - K(0) + i a |t-s| \in H^1(I^2;\C)$.
    Let $\nu \in [0,\frac14)$.
    By \cref{p:Q-a} (and \cref{ex:kernel-ab}), there exists $C_{R,\nu} > 0$ such that, for all $T \in (0,1]$ and $u,v \in \mathcal{H}$,
    \begin{equation}
        \Bigg| \int_{I^2} R(s,t) u(s) \overline{v}(t) \dd s \dd t \Bigg|
        \leq C_{R,\nu} T \| u_1 \|_{L^2} \| v_1 \|_{L^2} + C_{R,\nu} \nsob{u_1}{-\nu} \nsob{v_1}{-\nu}
    \end{equation}
    so $R$ satisfies the continuity estimate \eqref{eq:R-chi-hyp} with constants $(C_K,D_K) = (C_{R,\nu} T, C_{R,\nu})$.
    
    Let $\chi(s,t) := \exp( i \frac{\lambda_k}{2}|t-s| ) - 1$.
    Then $\chi \in W^{1,\infty}((0,T)^2;\C)$ with $\chi \in W^{2,\infty}(\Delta_\pm;\C)$.
    Applying \cref{p:R-chi}, we obtain the existence of $C_{\chi,\nu}$ such that, by \eqref{eq:R-chi-ccl}, for all $u, v \in \mathcal{H}$,
    \begin{equation}
        \label{eq:Qk-Q-2}
        \begin{split}
            \Bigg| \int_{[0,T]^2} \chi(s,t) R(s,t) u(s) \bar{v}(t) \dd s \dd t
            & - \int_{\Delta_+\cup\Delta_-} \Upsilon(s,t) u_1(s) \bar{v}_1(t) \dd s \dd t \Bigg| \\
            & \leq T C_{\chi,\nu} C_{R,\nu} \left( T \| u_1 \|_{L^2} \| v_1 \|_{L^2} + \nsob{u_1}{-\nu} \nsob{v_1}{-\nu} \right),
        \end{split}
    \end{equation}
    where $\Upsilon := R \partial_{21} \chi + \partial_2 \chi \partial_1 R + \partial_1 \chi \partial_2 R$ and we used that $R(s,s) = 0$ on $[0,T]$.

    Expanding the definitions yields, for $(s,t) \in [0,T]^2$, $\Upsilon(s,t) = \rho(t-s)$ where
    \begin{align}
        \rho(\sigma) & := e^{i \frac{\lambda_k}{2} |\sigma|} \left( \frac{\lambda_k^2}{4} f(|\sigma|) - i \lambda_k f'(|\sigma|) \right), \\
        f(\sigma) & := \sum_{j \in \N} c_j \left( e^{-i \lambda_j \sigma} - 1 + i \lambda_j \sigma \right).
    \end{align}
    One checks easily that $\rho$ is locally $\frac 12$-Hölder on $\R$.
    The key estimate is that, for $\sigma_1,\sigma_2 \in \R$,
    \begin{equation}
        |f'(\sigma_1) - f'(\sigma_2)| = \left|\sum_{j\in \N} \lambda_j c_j (e^{-i\lambda_j \sigma_2} - e^{-i\lambda_j \sigma_1})\right|.
    \end{equation}
    Writing $|e^{-i\lambda_j \sigma_2} - e^{-i\lambda_j \sigma_1}| \leq \min (2, \lambda_j|\sigma_2-\sigma_1|)$, we obtain recalling \eqref{eq:lambda_j} and \eqref{cj_asympt}
    \begin{equation}
        |f'(\sigma_1) - f'(\sigma_2)| \leq C \sum_{j \geq 1} \frac{\min (1, j^2|\sigma_2-\sigma_1|)}{j^2}.
    \end{equation}
    We conclude by splitting the sum at $j \approx |\sigma_2-\sigma_1|^{-\frac12}$.

    Since $\rho \in C^{0,\frac12}_{\mathrm{loc}}(\R;\C)$, by \cref{lem:convol-holder}, there exists $C_\rho > 0$ such that, for all $T \in (0,1]$,
    \begin{equation}
        \label{eq:Qk-Q-3}
        \Bigg| \int_{\Delta_+\cup\Delta_-} \Upsilon(s,t) u_1(s) \bar{v}_1(t) \dd s \dd t \Bigg| \leq C \nsob{u_1}{-\frac14}\nsob{v_1}{-\frac14}.
    \end{equation}
    The conclusion follows by combining \eqref{eq:Qk-Q-1}, \eqref{eq:Qk-Q-2} and \eqref{eq:Qk-Q-3}.
\end{proof}

\section{Remainder estimates for the non-linear system}
\label{sec:remainder}

We prove estimates quantifying the discrepancy between the solution $\psi(t)$ to the nonlinear system~\eqref{eq:psi} and its Taylor approximations.
In \cref{sec:estimates-classical}, we start with straightforward classical energy estimates expressed in terms of $\| u \|_{L^2}$.
Then, in \cref{sec:aux}, we present an auxiliary system which will allow to obtain estimates involving $\| u_1 \|_{L^2}$ instead of $\| u \|_{L^2}$.
Since this auxiliary system involves inhomogeneous Neumann boundary conditions, we derive general energy estimates for such systems in \cref{sec:inhomogeneous-new}.
Thanks to this tool, we derive linear, quadratic and cubic estimates respectively in \cref{sec:estimates-lin,sec:estimates-quad,sec:estimates-cub}.

We use the notation $\lesssim$ of \cref{sec:notations} for constants independent of time and controls.

\subsection{Classical estimates}
\label{sec:estimates-classical}

The following estimates are straightforward and classical.
They are estimates in the strong spatial norm $H^2(0,1)$, but hence involve the strong time norm of the control $\|u\|_{L^2}$.

\begin{lemma}
    \label{p:estimates-classical}
    Let $T \in (0,1]$ and $u \in L^2((0,T);\R)$ such that $\|u\|_{L^2} \leq 1$.
    Then
    \begin{align}
        \label{eq:est-psi-H2}
        \| \psi \|_{L^\infty H^2} & \lesssim 1, \\
        \label{eq:est-psi-phi0-H2}
        \| \psi - \varphi_0 \|_{L^\infty H^2} & \lesssim \| u \|_{L^2}, \\
        \label{eq:est-psi-phi0-psi1-H2}
        \| \psi - \varphi_0 - \psi_1 \|_{L^\infty H^2} & \lesssim \| u \|_{L^2}^2, \\
        \label{eq:est-psi-phi0-psi1-psi2-H2}
        \| \psi - \varphi_0 - \psi_1 - \psi_2 \|_{L^\infty H^2} & \lesssim \| u \|_{L^2}^3.
    \end{align}
\end{lemma}

\begin{proof}
    By \cref{p:WP}, $\|\psi\|_{L^\infty H^2} \leq C_\mu(T,\|u\|_{L^2})$, where $C_\mu$ is a non-decreasing function of both of its arguments. This implies \eqref{eq:est-psi-H2}. Using the initial condition $\psi(0) = \varphi_0$, we have the Duhamel formula
    \begin{equation}
        \psi(t) - \varphi_0 = i \int_0^t u(s) e^{-i A (t-s)} \mu \psi(s) \dd s
    \end{equation}
    which implies $\| \psi - \varphi_0 \|_{L^\infty H^2} \lesssim \| u \|_{L^2} \| \psi \|_{L^\infty H^2} \lesssim \| u \|_{L^2}$ by \cref{lem:smoothing}.
    Writing the Duhamel formula for $\psi - \varphi_0 - \psi_1$ and $\psi - \varphi_0 - \psi_1 - \psi_2$, one then iteratively obtains from \cref{lem:smoothing} the estimates~\eqref{eq:est-psi-phi0-psi1-H2} and \eqref{eq:est-psi-phi0-psi1-psi2-H2} in a bootstrap manner.
\end{proof}

\subsection{Auxiliary system}
\label{sec:aux}

To obtain energy estimates involving $\|u_1\|_{L^2}$, we study an \emph{auxiliary system} for a conjugate unknown~$\widetilde{\psi}$, whose evolution equation only involves the primitive $u_1$ (and not $u$).
This change of unknown is classical in this setting (see \cite[Section 7]{BeauchardLeBorgneMarbach2023} for the case of ODEs and \cite[Section 4.2]{Bournissou2023_Quad} for the Schrödinger PDE with Dirichlet boundary conditions). 
However, here, the estimates are of different nature (compare \cref{eq:cubic-rem} and \cref{cub_i}) and their proof is more subtle.

\medskip

From now on $T > 0$, $u \in L^2((0,T);\R)$ and $\psi$ is the solution to \eqref{eq:psi}. 
We set
\begin{equation}
    \widetilde{\psi}(t,x) := \psi(t,x) e^{-i u_1(t) \mu(x)}.
\end{equation}
Then $\widetilde{\psi} \in C^0([0,T];H^2 \cap \sphere) \cap H^1((0,T);L^2)$ because so does $\psi$ (see \cref{p:WP}). 
Since $\partial_t \widetilde{\psi}=(\partial_t \psi - i u \mu \psi)e^{-iu_1 \mu}=i (\partial_x^2 \psi) e^{-iu_1 \mu} \in C^0([0,T];L^2)$, one also has $\widetilde{\psi} \in C^1([0,T];L^2)$.
Thus $\widetilde{\psi}$ is a strong solution to the following auxiliary system
\begin{equation}
    \label{eq:tilde-psi}
    \begin{cases}
        i \partial_t \widetilde{\psi} = - \partial_x^2 \widetilde{\psi} - i u_1 ( 2 \mu^\prime \partial_x \widetilde{\psi} + \mu^{\prime \prime} \widetilde{\psi} ) + (u_1)^2 (\mu^\prime)^2 \widetilde{\psi} \\
        (\partial_x \widetilde{\psi} + i u_1 \mu^\prime \widetilde{\psi})_{\vert x = 0,1} = 0
    \end{cases}
\end{equation}
with initial data $\widetilde{\psi}(0) = \varphi_0$.
By \emph{strong solution}, we mean that the first equation is an equality in $C^0([0,T];L^2)$ and the traces at $x = 0$ and $x = 1$ hold pointwise in time.

\medskip

Let $\psi_1, \psi_2$ be the solutions to \eqref{eq:psi1}, \eqref{eq:psi2}.
We set
\begin{align}
    \label{eq:psi1-psit1}
    \widetilde{\psi}_1 & := \psi_1 - i u_1 \mu, \\
    \widetilde{\psi}_2 & := \psi_2 - i u_1 \mu \psi_1 - \frac 12 (u_1)^2 \mu^2.
\end{align}
which are the terms of order $1$ and $2$ in the power series expansion of $\widetilde{\psi}$ with respect to the control~$u$.
For the same reasons, $\widetilde{\psi}_1 , \widetilde{\psi}_2 \in C^0([0,T];H^2) \cap C^1([0,T];L^2)$ are strong solutions to the systems
\begin{equation}
    \label{eq:tilde-psi1}
    \begin{cases}
        i \partial_t \widetilde{\psi}_1 = - \partial_x^2 \widetilde{\psi}_1 - i u_1 \mu^{\prime \prime} \\
        (\partial_x \widetilde{\psi}_1 + i u_1 \mu^\prime)_{\vert x = 0,1} = 0
    \end{cases}
\end{equation}
and
\begin{equation}
    \label{eq:tilde-psi2}
    \begin{cases}
        i \partial_t \widetilde{\psi}_2 = - \partial_x^2 \widetilde{\psi}_2 - i u_1 ( 2 \mu^\prime \partial_x \widetilde{\psi}_1 + \mu^{\prime \prime} \widetilde{\psi}_1 ) + (u_1)^2 (\mu^\prime)^2 \\
        (\partial_x \widetilde{\psi}_2 + i u_1 \mu^\prime \widetilde{\psi}_1)_{\vert x = 0,1} = 0
    \end{cases}
\end{equation}
with initial conditions $\widetilde{\psi}_1(0) = \widetilde{\psi}_2(0) = 0$.

\paragraph{Remainders.}
Our goal is to prove energy estimates for the remainders between $\widetilde{\psi}$ and its Taylor approximations.
Therefore, we set $R_{-1} = R_0 := \widetilde{\psi}$, $R_1 := \widetilde{\psi} - \varphi_0$, $R_2 := \widetilde{\psi} - \varphi_0 - \widetilde{\psi}_1$ and $R_3 := \widetilde{\psi} - \varphi_0 - \widetilde{\psi}_1 - \widetilde{\psi}_2$.
With these notations, the successive remainders $R_1$, $R_2$ and $R_3$
belong to $C^0([0,T];H^2) \cap C^1([0,T];L^2)$ and are strong solutions to
\begin{equation}
    \label{eq:RN}
    \begin{cases}
        i \partial_t R_n = - \partial_x^2 R_n + F_n \\
        (\partial_x R_n + i u_1 \mu^\prime R_{n-1})_{\vert x = 0,1} = 0
    \end{cases}
\end{equation}
with initial data $R_n(0) = 0$ and where
\begin{equation}
    F_n := - i u_1 \left( 2 \mu^\prime \partial_x + \mu^{\prime \prime} \right) R_{n-1} + (u_1)^2 (\mu^\prime)^2 R_{n-2}.
\end{equation}

\subsection{Solutions with inhomogeneous Neumann conditions}
\label{sec:inhomogeneous-new}

We derive an \emph{a priori} estimate in $L^\infty L^2$ for solutions with inhomogeneous Neumann conditions.

\begin{proposition}[A priori estimate]
    \label{lem:inhomo}
    Let $T \in (0,1]$, $h_0, h_1 \in L^2((0, T); \C)$, $f \in L^1((0, T); L^2)$ and $\phi \in C^0([0,T];H^2) \cap C^1([0,T];L^2)$. 
    Assume that $\phi$ is a strong solution to
    \begin{equation}
        \label{eq:inhomo}
        \begin{cases}
            i \partial_t \phi = - \partial_x^2 \phi + f \\
            \partial_x \phi(\cdot, 0) = h_0 \\
            \partial_x \phi(\cdot, 1) = h_1
        \end{cases}
    \end{equation}
    with initial data $\phi(0) = 0$.
    Then
    \begin{equation}
        \label{eq:inhomo-bound}
        \| \phi \|_{L^\infty L^2} \lesssim \| h_0 \|_{L^2} + \| h_1 \|_{L^2} + \| f \|_{L^1 L^2}.
    \end{equation}
\end{proposition}

\begin{proof}
    Since $\phi$ is a strong solution to \eqref{eq:inhomo}, for $t \in [0,T]$ and $j \in \N$,
    \begin{equation}
        \begin{split}
            \frac{\dd}{\dd t} \langle \phi(t), \varphi_j \rangle
            & = i \int_0^1 \partial_x^2 \phi(t,x) \varphi_j(x) \dd x - i \langle f(t), \varphi_j \rangle \\
            & = i [\partial_x \phi(t) \varphi_j]_0^1 - i [\phi(t) \partial_x \varphi_j]_0^1 + i \int_0^1 \phi(t,x) \partial_x^2 \varphi_j(x) \dd x - i \langle f(t), \varphi_j \rangle \\
            & = i \left( \varphi_j(1) h_1(t) - \varphi_j(0) h_0(t) \right) - i \lambda_j \langle \phi(t), \varphi_j \rangle - i \langle f(t), \varphi_j \rangle.
        \end{split}
    \end{equation}
    By the Duhamel formula, since $\phi(0) = 0$,
    \begin{equation}
        \langle \phi(t), \varphi_j \rangle = i \int_0^t e^{-i \lambda_j(t-s)} \left( \varphi_j(1) h_1(s) - \varphi_j(0) h_0(s) - \langle f(s), \varphi_j \rangle \right) \dd s.
    \end{equation}
    Using that $|\varphi_j| \leq \sqrt{2}$, for any $t \in [0,T]$,
    \begin{equation}
        \| \phi(t) \|_{L^2}^2 \leq
        6 \sum_{k \in \{0,1\}} \sum_{j \in \N} \left| \int_0^t e^{i \lambda_j s} h_k(s) \dd s \right|^2
        + 3 \sum_{j \in \N} \left| \int_0^t e^{i \lambda_j s} \langle f(s), \varphi_j \rangle \dd s \right|^2.
    \end{equation}
    The first two sums are bounded by Bessel's inequality of \cref{lem:bessel}.
    Let $F(t) := \int_0^t e^{i A s} f(s) \dd s$.
    We have
    \begin{equation}
        \sum_{j \in \N} \left| \int_0^t e^{i \lambda_j s} \langle f(s), \varphi_j \rangle \dd s \right|^2
        = \sum_{j \in \N} \left| \langle F(t), \varphi_j \rangle \right|^2
        = \| F(t) \|_{L^2}^2
        \leq \| f \|_{L^1 L^2}^2
    \end{equation}
    by triangular inequality since $e^{i A s}$ is an isometry on $L^2((0,1);\C)$.
\end{proof}

\begin{remark}
    In \cite{BatalOzsari2015}, for the Schrödinger equation on a half-line, the authors obtain an $L^\infty H^{\frac 12}$ bound on $\phi$ for $L^2$ boundary data.
    The relationship between the regularity of $\phi$ and that of $h_0$ and $h_1$ obtained in \eqref{eq:inhomo-bound} on the interval $[0,1]$ is weaker.
    This fact was already observed for Schrödinger equations with inhomogeneous Dirichlet boundary conditions (see \cite[(2.23) and (2.25)]{BonaSunZhang2016}).
\end{remark}

\subsection{Linear estimates}
\label{sec:estimates-lin}

\begin{lemma}
    \label{lem:estimate-psi1}
    Let $T \in (0,1]$ and $u \in L^2((0,T);\R)$.
    Then
    \begin{align}
        \label{eq:estimate-psi1t-L2}
        \| \widetilde{\psi}_1 \|_{L^\infty L^2} & \lesssim \| u_1 \|_{L^2}, \\
        \label{eq:estimate-psi1t-H1}
        \| \widetilde{\psi}_1 \|_{L^\infty H^1} & \lesssim \| u_1 \|_{L^2}^{\frac 12} \| u \|_{L^2}^{\frac 12}, \\
        \label{eq:estimate-psi1t-H2}
        \| \widetilde{\psi}_1 \|_{L^\infty H^2} & \lesssim \| u \|_{L^2}.
    \end{align}
\end{lemma}

\begin{proof}
    First, by applying \cref{lem:inhomo} to system \eqref{eq:tilde-psi1}, we obtain \eqref{eq:estimate-psi1t-L2}. Second, \eqref{eq:estimate-psi1t-H2} follows from equality \eqref{eq:psi1-psit1} and the estimate $\| \psi_1 \|_{L^\infty H^2} \lesssim \| u \|_{L^2}$, which is \cref{lem:smoothing} applied to \eqref{eq:psi1}.
    Third, for \eqref{eq:estimate-psi1t-H1}, we write $\| \widetilde{\psi}_1(t) \|_{H^1}^2 \lesssim \| \widetilde{\psi}_1(t) \|_{L^2} \| \widetilde{\psi}_1(t) \|_{H^2}$.
\end{proof}

\begin{lemma}
    \label{lem:estimate-R1}
    Let $T \in (0,1]$ and $u \in L^2((0,T);\R)$ be such that $\left\Vert u \right\Vert_{L^2} \leq 1$.
    Then
    \begin{align}
        \label{eq:estimate-R1-L2}
        \| R_1 \|_{L^\infty L^2} & \lesssim \| u_1 \|_{L^2}, \\
        \label{eq:estimate-R1-H1}
        \| R_1 \|_{L^\infty H^1} & \lesssim \| u_1 \|_{L^2}^{\frac 12} \| u \|_{L^2}^{\frac 12}.
    \end{align}
\end{lemma}

\begin{proof}
    First, for \eqref{eq:estimate-R1-L2}, since $R_1$ is the solution to \eqref{eq:RN} with $n=1$, by \cref{lem:inhomo}, one has
    \begin{equation}
        \| R_1 \|_{L^\infty L^2} \lesssim \| u_1 (\mu' \widetilde{\psi})_{\vert x=0} \|_{L^2} + \| u_1 (\mu' \widetilde{\psi})_{\vert x=1} \|_{L^2} + \| F_1 \|_{L^1 L^2}.
    \end{equation}
    Moreover, since $F_1 = - i u_1 ( 2 \mu^\prime \partial_x \widetilde{\psi} + \mu^{\prime \prime} \widetilde{\psi} ) + (u_1)^2 (\mu^\prime)^2 \widetilde{\psi}$,
    \begin{equation}
        \| F_1 \|_{L^1 L^2} \lesssim \| u_1 \|_{L^2} \left( \| \mu' \|_{L^\infty} \| \widetilde{\psi} \|_{L^\infty H^1} + \| \mu'' \|_{L^2} \| \widetilde{\psi} \|_{L^\infty L^\infty} \right) + \| u_1 \|_{L^2}^2 \| \mu' \|_{L^\infty}^2 \| \widetilde{\psi} \|_{L^\infty L^2}.
    \end{equation}
    Using the embedding $H^1(0,1) \hookrightarrow L^\infty(0,1)$, we obtain
    \begin{equation}
        \| R_1 \|_{L^\infty L^2} \lesssim \| u_1 \|_{L^2} \| \widetilde{\psi} \|_{L^\infty H^1} \lesssim \| u_1 \|_{L^2} \| \psi \|_{L^\infty H^1} \lesssim \| u_1 \|_{L^2}
    \end{equation}
    using \cref{p:WP}, which proves \eqref{eq:estimate-R1-L2}.

    Second, to prove \eqref{eq:estimate-R1-H1}, we write $R_1 = \widetilde{\psi} - 1 = \psi(e^{-iu_1 \mu}-1) + (\psi-1)$.
    Hence,
    \begin{equation}
        \label{eq:estimate-R1-H2}
        \| R_1(t) \|_{H^2}
        \lesssim \|u_1\|_{L^\infty} \|\mu\|_{H^2} \| \psi(t) \|_{H^2} + \| \psi(t) - 1 \|_{H^2}
        \lesssim \| u_1 \|_{L^\infty} + \| u \|_{L^2} \lesssim \| u \|_{L^2}
    \end{equation}
    by \eqref{eq:est-psi-H2} and \eqref{eq:est-psi-phi0-H2}.
    Since $\| R_1(t) \|_{H^1} \lesssim \| R_1(t) \|_{L^2}^{\frac 12} \| R_1(t) \|_{H^2}^{\frac 12}$, \eqref{eq:estimate-R1-L2} and \eqref{eq:estimate-R1-H2} entail \eqref{eq:estimate-R1-H1}.
\end{proof}

\subsection{Quadratic estimates}
\label{sec:estimates-quad}

\begin{lemma}
    \label{lem:estimate-psi2}
    Let $T \in (0,1]$ and $u \in L^2((0,T);\R)$ be such that $\| u \|_{L^2} \leq 1$.
    Then
    \begin{equation}
        \label{eq:estimate-psi2-L2}
        \| \widetilde{\psi}_2 \|_{L^\infty L^2} \lesssim \| u_1 \|_{L^2}^{\frac 32} \| u \|_{L^2}^{\frac 12}.
    \end{equation}
\end{lemma}

\begin{proof}
    Since $\widetilde{\psi}_2$ is the solution to \eqref{eq:tilde-psi2}, by \cref{lem:inhomo}, one has
    \begin{equation}
        \begin{split}
            \| \widetilde{\psi}_2 \|_{L^\infty L^2} & \lesssim \| u_1 \mu' \partial_x \widetilde{\psi}_1 \|_{L^1 L^2} + \| u_1 \mu'' \widetilde{\psi}_1 \|_{L^1 L^2} + \| u_1^2 (\mu')^2 \|_{L^1 L^2} \\
            & \qquad \qquad + \| u_1 (\mu' \widetilde{\psi}_1)_{\vert x = 0} \|_{L^2} + \| u_1 (\mu' \widetilde{\psi}_1)_{\vert x= 1} \|_{L^2} \\
            & \lesssim \| u_1 \|_{L^2} \| \widetilde{\psi}_1 \|_{L^\infty H^1} +\| u_1 \|_{L^2} \| \widetilde{\psi}_1 \|_{L^\infty L^\infty}.
        \end{split}
    \end{equation}
    Therefore, the conclusion follows from \cref{lem:estimate-psi1} and the embedding $H^1(0,1) \hookrightarrow L^\infty(0,1)$.
\end{proof}

\begin{lemma}
    \label{lem:estimate-R2}
    Let $T \in (0,1]$ and $u \in L^2((0,T);\R)$ be such that $\| u \|_{L^2} \leq 1$.
    Then
    \begin{align}
        \label{eq:estimate-R2-L2}
        \| R_2 \|_{L^\infty L^2} & \lesssim \| u_1 \|_{L^2}^{\frac 32} \| u \|_{L^2}^{\frac 12}, \\
        \label{eq:estimate-R2-Linf}
        \| R_2 \|_{L^\infty L^\infty} & \lesssim \| u_1 \|_{L^2}^{\frac 98} \| u \|_{L^2}^{\frac 78}.
    \end{align}
\end{lemma}

\begin{proof}
    First, for \eqref{eq:estimate-R2-L2}, since $R_2$ is the solution to \eqref{eq:RN} with $n=2$, by \cref{lem:inhomo}, one has
    \begin{equation}
        \begin{split}
            \| R_2 \|_{L^\infty L^2} & \lesssim \| u_1 (\mu' R_1)_{\vert x=0} \|_{L^2} + \| u_1 (\mu' R_1)_{\vert x=1} \|_{L^2} + \| F_2 \|_{L^1 L^2} \\
            & \lesssim \| u_1 \|_{L^2} \| R_1 \|_{L^\infty L^\infty} + \| F_2 \|_{L^1 L^2}.
        \end{split}
    \end{equation}
    Moreover, since $F_2 = - i u_1 ( 2 \mu^\prime \partial_x R_1 + \mu^{\prime \prime} R_1) + (u_1)^2 (\mu^\prime)^2 \widetilde{\psi}$,
    \begin{equation}
        \| F_2 \|_{L^1 L^2} \leq \| u_1 \|_{L^2} \left( 2 \| \mu' \|_{L^\infty} \| R_1 \|_{L^\infty H^1} + \| \mu'' \|_{L^2} \| R_1 \|_{L^\infty L^\infty} \right) + \| u_1 \|_{L^2}^2 \| \mu' \|_{L^\infty}^2 \| \widetilde{\psi} \|_{L^\infty L^\infty}.
    \end{equation}
    By \cref{p:WP}, $\| \widetilde{\psi} \|_{L^\infty H^2} \lesssim 1$.
    By~\eqref{eq:estimate-R1-H1} of \cref{lem:estimate-R1} and $H^1(0,1) \hookrightarrow L^\infty(0,1)$, we obtain
    \begin{equation}
        \| u_1 \|_{L^2} \| R_1 \|_{L^\infty L^\infty} + \| F_2 \|_{L^2 L^2} \lesssim \| u_1 \|_{L^2}^{\frac 32} \| u \|_{L^2}^{\frac 12},
    \end{equation}
    from which estimate \eqref{eq:estimate-R2-L2} follows.

    Second, we prove \eqref{eq:estimate-R2-Linf}.
    By Gagliardo--Nirenberg interpolation inequalities (see \cite{Gagliardo1959,Nirenberg1959}), there exists $C > 0$ such that, for every $f \in H^2(0,1)$, $\| f \|_{L^\infty} \leq C \| f \|_{H^2}^{\frac14} \| f \|_{L^2}^{\frac 34}$.
    Hence, for each $t \in [0,T]$, $\| R_2(t) \|_{L^\infty} \leq C \| R_2(t) \|_{H^2}^{\frac 14} \| R_2(t) \|_{L^2}^{\frac 34}$.
    Thus, it suffices to prove that
    \begin{equation}
        \label{eq:estimate-R2-H2}
        \| R_2(t) \|_{H^2} \lesssim \| u \|_{L^2}^2.
    \end{equation}
    One has
    \begin{equation}
        \label{eq:R2-R2t}
        \psi - \varphi_0 - \psi_1
        = \left(e^{i u_1 \mu} - 1 - i u_1 \mu\right) \varphi_0
        + \left(e^{i u_1 \mu} - 1 \right) \widetilde{\psi}_1
        + e^{i u_1 \mu} R_2.
    \end{equation}
    Since $\| e^{i u_1 \mu} - 1 - i u_1 \mu \|_{H^2} \lesssim \| u_1 \|_{L^\infty}^2 \lesssim \| u \|_{L^2}^2$ and $\| (e^{i u_1 \mu} - 1) \widetilde{\psi}_1 \|_{H^2} \lesssim \|u\|_{L^2}^2$ by \eqref{eq:estimate-psi1t-H2}, \eqref{eq:estimate-R2-H2} follows from \eqref{eq:est-psi-phi0-psi1-H2}.
    Thus, we deduce \eqref{eq:estimate-R2-Linf}.
\end{proof}

\begin{corollary}
    \label{cor:quad-rem}
    Let $T \in (0,1]$ and $u \in L^2((0,T);\R)$ be such that $\|u\|_{L^2} \leq 1$.
    Then
    \begin{equation}
        \label{eq:quad-rem}
        \| \psi(T) - \varphi_0 - \psi_1(T) \|_{L^2} \lesssim \| u_1 \|_{L^2}^{\frac 32} \| u \|_{L^2}^{\frac12} + |u_1(T)|^2.
    \end{equation}
\end{corollary}

\begin{proof}
    This follows from the decomposition \eqref{eq:R2-R2t} and estimates \eqref{eq:estimate-R2-L2} and \eqref{eq:estimate-psi1t-L2}.
\end{proof}

\subsection{Cubic estimates}
\label{sec:estimates-cub}

We now deduce an estimate on $\langle R_3, \varphi_k \rangle$, which immediately entails the main objective of this section, the cubic estimate of \cref{cor:cubic-rem}.
We start with the following weak estimate.

\begin{lemma}
    \label{lem:R3-bis}
    Let $T \in (0,1]$ and $u \in L^2((0,T);\R)$ be such that $\|u\|_{L^2} \leq 1$.
    Then
    \begin{equation}
        \label{eq:R3-bis}
        | \langle R_3(T), \varphi_k \rangle | \lesssim \|u_1\|_{L^2}^{2+\frac18} \|u\|_{L^2}^{\frac 78}.
    \end{equation}
\end{lemma}

\begin{proof}
    Since $R_3$ is a strong solution to \eqref{eq:RN} with $F_3 = \partial_x (-2 i u_1 \mu' R_2) + i u_1 \mu'' R_2 + u_1^2 (\mu')^2 R_1$, integrations by parts lead to
    \begin{equation}
        \begin{split}
            \frac{\dd}{\dd t} \langle R_3, \varphi_k \rangle & = \langle i \partial_x^2 R_3, \varphi_k \rangle - i \langle F_3, \varphi_k \rangle \\
            & = - i \lambda_k \langle R_3, \varphi_k \rangle - u_1 [ \mu' R_2 \varphi_k ]_0^1 + u_1 \langle R_2 , 2\mu' \varphi_k' + \mu'' \varphi_k \rangle
            - i u_1^2 \langle R_1 , (\mu')^2 \varphi_k \rangle.
        \end{split}
    \end{equation}
    Thus
    \begin{equation}
        | \langle R_3(T), \varphi_k \rangle | \lesssim \|u_1\|_{L^2} \| R_2 \|_{L^\infty L^\infty} + \|u_1\|_{L^2}^2 \|R_1\|_{L^\infty L^2},
    \end{equation}
    which implies the claimed estimate by \eqref{eq:estimate-R2-Linf} and \cref{eq:estimate-R1-L2}.
\end{proof}

\begin{corollary}
    \label{cor:cubic-rem}
    Let $T \in (0,1]$ and $u \in L^2((0,T);\R)$ be such that $\|u\|_{L^2} \leq 1$.
    Then
    \begin{equation}
        \label{eq:cubic-rem}
        | \langle (\psi - \varphi_0 - \psi_1 - \psi_2)(T), \varphi_k \rangle | \lesssim |u_1(T)|^3 + \|u_1\|_{L^2}^{2+\frac18} \|u\|_{L^2}^{\frac78}.
    \end{equation}
\end{corollary}

\begin{proof}
    One has
    \begin{equation}
        \label{eq:R3-R3t}
        \begin{split}
            \psi - \varphi_0 - \psi_1 - \psi_2
            = \left(e^{i u_1 \mu} - 1 - i u_1 \mu + \frac{u_1^2}{2} \mu^2\right) \varphi_0
            & + \left(e^{i u_1 \mu} - 1 - i u_1 \mu \right) \widetilde{\psi}_1 \\
            & + \left(e^{i u_1 \mu} - 1 \right) \widetilde{\psi}_2
            + e^{i u_1 \mu} R_3.
        \end{split}
    \end{equation}
    Hence \cref{lem:estimate-psi1} and \cref{lem:estimate-psi2} yield
    \begin{equation}
        \begin{split}
            | \langle (\psi - \varphi_0 & - \psi_1 - \psi_2)(T), \varphi_k \rangle | \\
            & \lesssim |u_1(T)|^3 + |u_1(T)|^2 \| \widetilde{\psi}_1(T) \|_{L^2} + |u_1(T)| \| \widetilde{\psi}_2(T) \|_{L^2} + | \langle (e^{i u_1 \mu} R_3)(T), \varphi_k \rangle | \\
            & \lesssim |u_1(T)|^3 + \| u_1 \|_{L^2}^{\frac 94} \| u \|_{L^2}^{\frac 34} + | \langle (e^{i u_1 \mu} R_3)(T), \varphi_k \rangle |.
        \end{split}
    \end{equation}
    One has
    \begin{equation}
        | \langle (e^{i u_1 \mu} R_3)(T), \varphi_k \rangle | \lesssim | \langle R_3(T), \varphi_k \rangle | + |u_1(T)| \| R_3(T) \|_{L^2}.
    \end{equation}
    As $R_3 = R_2 - \widetilde{\psi}_2$, \eqref{eq:estimate-psi2-L2} and \eqref{eq:estimate-R2-L2} imply $\| R_3(T) \|_{L^2} \lesssim \|u_1\|_{L^2}^{\frac32} \|u\|_{L^2}^{\frac12}$.

    Together with \cref{lem:R3-bis}, we obtain
    \begin{equation}
        | \langle (e^{i u_1 \mu} R_3)(T), \varphi_k \rangle | \lesssim \|u_1\|_{L^2}^{2+\frac18} \|u\|_{L^2}^{\frac78} + |u_1(T)| \|u_1\|_{L^2}^{\frac32} \|u\|_{L^2}^{\frac 12}
    \end{equation}
    and this completes the proof of \eqref{eq:cubic-rem}.
\end{proof}

\begin{remark}
    Assuming that $\mu \in H^3((0, 1);\R)$ and using a slightly different method (based on a more detailed analysis of the regularity of solutions with inhomogeneous Neumann conditions), one can replace $\|u_1\|_{L^2}^{2+\frac18} \|u\|_{L^2}^{\frac78}$ with $\|u_1\|_{L^2}^{2+\frac14} \|u\|_{L^2}^{\frac34}$ in the above estimates.
    See also \cref{sec:easy-estimates} below where the assumption $\mu \in H^3_N((0,1);\R)$ entails an estimate by $\|u_1\|_{L^2}^3$.
\end{remark}

\appendix

\section{Drift quantified by integer negative Sobolev norms}
\label{sec:drift_integer}

The goal of this section is to prove the following drift estimate, via the usual integration by parts approach.
The statement and its proof are adaptations to Neumann boundary conditions of \cite[Theorem 1.3]{Bournissou2023_Quad}, written for Dirichlet boundary conditions in the bilinear Schrödinger equation. 

For $m \in \N^*$, we use the notation 
\begin{equation}
    H^m_N((0,1);\C) := D(A^{\frac m 2}) = \left\{ \phi \in H^{m}((0,1);\C) \mid (\phi^{(2j+1)})_{\vert x = 0,1} = 0 \text{ for } 0 < 2j+1 < m \right\}.
\end{equation}
Extending the notation $u_1$ for the primitive of $u$, we let $u_{n+1}(t) := \int_0^t u_n$ for all $n \geq 0$ with $u_0 = u$.

\begin{theorem}
    \label{thm:integer_drift}
    Let $k \in \N$, $p\in\N^*$, $\mu \in H^{2p+1}_N((0,1);\R)$ such that $\langle \mu ,\varphi_k \rangle=0$.
    For $1 \leq n \leq p$, let
    \begin{equation}
        \label{Wp_Schro}
        a_{k}^n :=
        \sum_{j \in \N} c_j \lambda_j^{n-1}(\lambda_j-\lambda_k)^{n-1} (\lambda_j-\lambda_k/2),
    \end{equation}
    where $c_j$ is defined by \eqref{def:cj_a}.
    We assume $p$ is the minimal integer for which $a_k^p \neq 0$.
    There exists $T^*,C>0$ such that, for every $T \in (0,T^*)$ and $u \in L^2 \cap H^{2p-3} ((0,T);\R)$ with $\|u\|_{L^2} \leq 1$,
    \begin{equation}
        \label{eq:drift-Wp}
        \left| \langle \psi(T) - \varphi_0, \varphi_k \rangle + i a_k^p \|u_p\|_{L^2}^2 \right|
        \leq C \Gamma_p(T, u) \|u_p\|_{L^2}^2 + C \|\psi(T)-\varphi_0\|_{L^2}^2,
    \end{equation}
    where $\Gamma_p(T, u) := T + \|u_1^{(2p-2)}\|_{L^2} + T^{2-2p} \|u_1\|_{L^2}$.
\end{theorem}

As in \cref{rq:imp_motion}, estimate \eqref{eq:drift-Wp} prevents particular motions in $\vect_{\C} \{ \varphi_0,\varphi_k \}$, when $T$ and $u$ are small enough so that $\Gamma_p(T,u) < |a^p_k|$.
In the case $p=1$, we recover \cref{thm:negative} under the stronger assumption $\mu \in H^3_N((0,1);\R)$, which allows a much simpler proof.

The assumption $\mu \in H^{2p+1}_N$ implies that $\lambda_j^p \langle \mu , \varphi_j \rangle = \langle A^p \mu , \varphi_j \rangle$. 
An additional integration by parts justifies $\lambda_j^p \langle \mu , \varphi_j \rangle = o(1/j)$.
This estimate also holds with $\mu \leftarrow \mu \varphi_k \in H^{2p+1}_N$.
Thus
\begin{equation}
    \label{cj=1/Lj^2p+1}
    c_j = \underset{j \to \infty}{o} ( 1/ \lambda_j^{2p+1} )
\end{equation}
and, for any $1 \leq n \leq p$, the series $a_k^n$ of \eqref{Wp_Schro} converges absolutely. 

The following lemma gives an interpretation of $a_k^n$ in terms of Lie brackets (with the convention $\ad_{A}(\mu) = [A, \mu] = A \mu - \mu A$). 
We recover the usual drift direction of ODEs (see \cite[Theorem 6.1]{BeauchardMarbach2022}).

\begin{lemma}
    If $\mu \in C^{\infty}_c((0,1);\R)$, then, for every $k \in \N$, $n \in \N^*$,
    \begin{equation}
        a_{k}^n = \frac{(-1)^{n-1}}{2} \langle [ \ad_{A}^{n-1}(\mu) , \ad_{A}^n(\mu)] \varphi_0 , \varphi_k \rangle.
    \end{equation}
\end{lemma}

\begin{proof}
    The assumption $\mu \in C^{\infty}_c((0,1);\R)$ guarantees that the iterated Lie brackets are well defined: $A$ is only applied to functions in $D(A)=H^2_N$.
    For every $\ell \in \N$ and $m \in \N$ we have
    \begin{equation}
        \ad_{A}^m(\mu) \varphi_{\ell} 
        = \ad_{A-\lambda_{\ell}}^m(\mu) \varphi_{\ell} \\
        = \sum_{q=0}^m (-1)^q \binom{m}{q} (A-\lambda_{\ell})^{m-q} \mu (A-\lambda_{\ell})^q \varphi_{\ell}
        = (A-\lambda_{\ell})^m(\mu \varphi_{\ell}).
    \end{equation}
    Thus for every $j \in \N$,
    \begin{equation}
        (\lambda_j-\lambda_{\ell})^m \langle \varphi_j , \mu \varphi_{\ell} \rangle
        = \langle (A-\lambda_{\ell})^m \varphi_j , \mu \varphi_{\ell} \rangle
        = \langle \varphi_j , (A-\lambda_{\ell})^m (\mu \varphi_{\ell}) \rangle
        = \langle \varphi_j , \ad_{A}^m(\mu) \varphi_{\ell} \rangle,
    \end{equation}
    where the second equality results from integrations by parts and the boundary terms vanish because $\mu \in C^{\infty}_c(0,1)$. 
    Therefore, using $2(\lambda_j-\lambda_k/2)=\lambda_j + (\lambda_j-\lambda_k)$ and Parseval's identity, we obtain
    \begin{equation}
        2 a_k^n = \langle \ad_{A}^{n}(\mu) \varphi_0 , \ad_{A}^{n-1}(\mu)\varphi_k \rangle + \langle \ad_A^{n-1}(\mu)\varphi_0 , \ad_{A}^{n}(\mu)\varphi_k\rangle
    \end{equation}
    which concludes the proof using that $(\ad_A^q(\mu))^* = (-1)^q \ad_A^q(\mu)$.
\end{proof}

\subsection{Coercivity of the second order expansion}
\label{subsec:coer_app}

\begin{proposition}
    \label{Prop:coer}
    Let $k \in \N$, $p\in\N^*$, $\mu \in H^{2p+1}_N((0,1);\R)$ such that $p$ is the minimal integer for which $a_k^p \neq 0$.
    There exists $C>0$ such that, for every $T \in (0,1]$ and $u \in L^2((0,T);\R)$,
    \begin{equation}
        \left| \langle \psi_2(T), \varphi_k\rangle + i a_k^p \|u_p\|_{L^2}^2 \right| \leq C \left( T \|u_p\|_{L^2}^2 + |u_1(T)|^2 + \dotsb + |u_p(T)|^2 \right).
    \end{equation}
\end{proposition}

\begin{proof}
    We deduce from \eqref{eq:psi2=Qk} that, for every $T >0$ and $u \in L^2((0,T);\R)$,
    \begin{equation}
        \label{def:H(t,s)}
        \langle \psi_2(T), \varphi_k e^{-i\lambda_k T}\rangle = - \int_0^T u(t) \int_0^t u(s) H(t,s) \dd s \dd t
        \quad \text{ where } \quad
        H(t,s):=e^{i \lambda_k t} K(t-s)
    \end{equation}
    and $K(\sigma):=\sum_{j\in\N} c_j e^{-i\lambda_j \sigma}$ defines $K \in C^{2p}(\R;\C)$ (see \cref{cj=1/Lj^2p+1}), thus $H \in C^{2p}(\R^2;\C)$.

    \medskip
    \step[st-Qn-1]{We prove by induction on $0 \leq n \leq p$, that, for every $T>0$, there exists a quadratic form $\widetilde{Q}_{n,T}$ on $\C^{2n}$, with uniformly bounded coefficients when $T \to 0$, such that, for all $u \in L^2((0,T);\R)$,
    \begin{equation}
        \label{H_IPP}
        \begin{aligned}
            \int_0^T u(t) \int_0^t u(s) H(t,s) \dd s \dd t
            =
            \sum_{q=1}^n \int_0^T u_q(t)^2 \left(
            \partial_1^{q-1} \partial_2^{q} H(t,t)
            - \frac{1}{2} \frac{\dd}{\dd t}\left[ \partial_1^{q-1} \partial_2^{q-1} H(t,t) \right]
            \right) \dd t
            \\
            + \int_0^T u_n(t) \int_0^t u_n(s) \partial_1^n \partial_2^n H(t,s) \dd s \dd t + \widetilde{Q}_{n,T}(u_1(T),\dots,u_n(T),\gamma_0^n,\dots,\gamma_{n-1}^n)
        \end{aligned}
    \end{equation}
    where, for $0 \leq q \leq n-1$, $\gamma_q^n := \int_0^T u_n(s) \partial_1^q \partial_2^n H(T,s) \dd s$.}
    The equality holds for $n=0$ with $\widetilde{Q}_{0,T}=0$.
    The heredity $n \to (n+1)$ relies on integrations by parts.
    In the double integral above, we perform two integrations by parts, to produce a double integral involving $u_{n+1}(t) u_{n+1}(s)$ and boundary terms involving $u_{n+1}(T)$ and $\gamma_n^{n+1}$. 
    In each $\gamma_q^n$, we perform one integration by parts to produce an integral involving $u_{n+1}(s)$ and a boundary term involving $u_{n+1}(T)$.

    \step[st-Qn-2]{Coercivity of the quadratic form.}
    We deduce from \eqref{def:H(t,s)} that for every $1 \leq n \leq p$,
    \begin{equation}
        \partial_1^{n-1} \partial_2^{n} H(t,t)
        - \frac{1}{2} \frac{\dd}{\dd t}\left[ \partial_1^{n-1} \partial_2^{n-1} H(t,t) \right] = i a_k^{n} e^{i\lambda_k t}.
    \end{equation}
    Using \cref{st-Qn-1} with $n=p$, $a_k^{m}=0$ for $m<p$ and the Cauchy--Schwarz inequality, we get $C>0$ such that, for every $T \in (0,1]$ and $u \in L^2((0,T);\R)$,
    \begin{equation}
        \left| \langle \psi_2(T),\varphi_k e^{-i\lambda_k T}\rangle + ia_k^p \int_0^T u_p(t)^2 e^{i\lambda_k t} \dd t \right|
        \leq C \left(
        T \|u_p\|_{L^2}^2 + \sum_{q=1}^{p} |u_q(T)|^2 \right)
    \end{equation}
    which implies the conclusion.
\end{proof}

\subsection{Quadratic and cubic estimates}
\label{sec:easy-estimates}

\begin{proposition}
    \label{Prop:cub}
    Let $\mu \in H^3_N((0,1);\R)$ and $k \in \N^*$. There exists $C>0$ such that, for every $T\in (0,1]$ and $u \in L^2((0,T);\R)$ with $\|u\|_{L^2} \leq 1$,
    \begin{align}
        \label{quad_i}
        \| (\psi-\varphi_0-\psi_1)(T) \|_{L^2} & \leq C\left( \|u_1\|_{L^2}^2 + |u_1(T)|^2 \right), \\
        \label{cub_i}
        | \langle (\psi-\varphi_0-\psi_1-\psi_2)(T), \varphi_k \rangle | & \leq C\left( \|u_1\|_{L^2}^3 + |u_1(T)|^3 \right).
    \end{align}
\end{proposition}

\begin{remark}
    The assumption $\mu \in H^3_N$ allows stronger quadratic and cubic estimates than in \cref{sec:remainder} (compare with \eqref{eq:quad-rem} and \eqref{eq:cubic-rem}), with simpler proofs.
\end{remark}

\begin{proof}
    We use the notation $\lesssim$ of \cref{sec:notations} for constants independent of time and controls.
    
    We use the auxiliary system introduced in \cref{sec:aux}.
    The key simplification is that the assumption $\mu \in H^3_N$ implies that $\mu'(0) = \mu'(1) = 0$, so that the boundary conditions for $\widetilde{\psi}$, $\widetilde{\psi}_1$, $\widetilde{\psi}_2$, $R_1$, $R_2$, $R_3$ in \eqref{eq:tilde-psi}, \eqref{eq:tilde-psi1}, \eqref{eq:tilde-psi2} and \eqref{eq:RN} are now homogeneous Neumann boundary conditions.
    Thus, we can write Duhamel formulas to compute the solutions from the source terms.

    From \eqref{eq:tilde-psi1}, using that $e^{iAs}$ is continuous on $H^1$, and $\mu'' \in H^1$, we obtain
    \begin{equation}
        \label{eq:magic-psit1}
        \| \widetilde{\psi}_1 \|_{L^\infty H^1} \lesssim \| u_1 \mu'' \|_{L^1 H^1} \lesssim \| u_1 \|_{L^2}.
    \end{equation}
    From \eqref{eq:tilde-psi2} and \eqref{eq:magic-psit1}, using that $e^{iAs}$ is an isometry on $L^2$, we obtain
    \begin{equation}
        \label{eq:magic-psit2}
        \| \widetilde{\psi}_2 \|_{L^\infty L^2} 
        \lesssim \| u_1 \|_{L^2} \| \widetilde{\psi}_1 \|_{L^\infty H^1} + \| u_1 \|_{L^2}^2
        \lesssim \| u_1 \|_{L^2}^2.
    \end{equation}
    From \eqref{eq:RN} for $R_1$, using that $\mu'' \in H^1$ we have
    \begin{equation}
        \label{eq:magic-R1}
        \| R_1 \|_{L^\infty H^1} \lesssim \| u_1 \|_{L^2} \| \widetilde{\psi} \|_{L^\infty H^1} + \| u_1 \|_{L^2}^2 \| \widetilde{\psi} \|_{L^\infty H^1} \lesssim \| u_1 \|_{L^2}.            
    \end{equation}
    From \eqref{eq:RN} for $R_2$ and \eqref{eq:magic-R1}, we obtain
    \begin{equation}
        \label{eq:magic-R2}
        \| R_2 \|_{L^\infty L^2} \lesssim \| u_1 \|_{L^2} \| R_1 \|_{L^\infty H^1} + \| u_1 \|_{L^2}^2 \| \widetilde{\psi} \|_{L^\infty L^2}
        \lesssim \| u_1 \|_{L^2}^2.
    \end{equation}
    Let $L := 2 \mu' \partial_x + \mu''$.
    Since $\mu'(0) = \mu'(1) = 0$, $L^* = - L$.
    From \eqref{eq:RN} for $R_3$, we obtain
    \begin{equation}
        \langle R_3(T) , \varphi_k e^{-i u_1(T)\mu} \rangle =
        \int_0^T \left( u_1(t) \langle R_2(t) , L \phi(t) \rangle
        - i u_1(t)^2 \langle (\mu')^2 R_1(t) , \phi(t) \rangle
        \right) \dd t
    \end{equation}
    where $\phi(t):=e^{iA(T-t)}( \varphi_k e^{-iu_1(T)\mu})$.
    We have $\phi \in C^0([0,T];H^1)$ because $\varphi_k e^{-iu_1(T)\mu} \in H^1$.
    Thus
    \begin{equation}
        \label{eq:magic-R3}
        \left| \langle R_3(T)e^{i u_1(T)\mu} , \varphi_k \rangle \right|
        \lesssim \| u_1 \|_{L^2} \| R_2 \|_{L^\infty L^2} + \| u_1 \|_{L^2}^2 \|R_1 \|_{L^\infty L^2}.
    \end{equation}
    Eventually, we get \eqref{quad_i} from the decomposition \eqref{eq:R2-R2t} and the estimates \eqref{eq:magic-psit1} and \eqref{eq:magic-R2}.
    We get \eqref{cub_i} from the decomposition \eqref{eq:R3-R3t} and the estimates \eqref{eq:magic-psit1}, \eqref{eq:magic-psit2} and \eqref{eq:magic-R3}.
\end{proof}

\subsection{Closed-loop estimate}

\begin{proposition}
    \label{Prop:CL}
    Let $k \in \N$, $p \in \N^*$, $\mu \in H^{2p+1}_N((0,1);\R)$ such that $p$ is the minimal integer for which $a_k^p \neq 0$. Then there exists $C,T^*>0$ such that, for every $T \in (0,T^*)$ and $u \in L^2((0,T);\R)$,
    \begin{equation}
        |u_1(T)| + \dotsb + |u_p(T)| \leq C \left( T^{\frac12} \|u_p\|_{L^2} + \|u_1\|_{L^2}^2 + \|\psi(T)-\varphi_0\|_{L^2} \right).
    \end{equation}
\end{proposition}

\begin{proof}
    We follow exactly \cite[Lemma 5.3 and Proposition 5.4]{Bournissou2023_Quad}.
    The strategy consists in proving first that there exist at least $p$ values of $j \in \N^*$ for which $\langle \mu, \varphi_j \rangle \neq 0$.
    The associated $p$ moments $\int_0^T u(t) e^{-i\lambda_j(T-t) } \dd t$ are present in the expansion of $\psi_1$ (see \cref{linearise_explicit}).
    In each of them, we perform~$p$ integrations by parts, to produce $u_1(T), \dotsc, u_p(T)$ and moments of $u_p$.
    We get the conclusion thanks to a Vandermonde argument and the quadratic estimate \eqref{quad_i}.
\end{proof}

\subsection{Proof of the drift}

\begin{proof}[Proof of \cref{thm:integer_drift}]
    Let $k \in \N$, $p\in\N^*$, $\mu \in H^{2p+1}_N((0,1);\R)$ such that $\langle \mu ,\varphi_k \rangle=0$ and $p$ is the minimal integer for which $a_k^p \neq 0$.
    Gathering \cref{Prop:coer,Prop:cub,Prop:CL}, we obtain $C,T^*>0$ such that, for every $T\in(0,T^*)$ and $u \in L^2((0,T);\R)$ with $\| u \|_{L^2} \leq 1$,
    \begin{equation}
        \left| \langle \psi(T)-\varphi_0 , \varphi_k \rangle + i a_k^p \|u_p\|_{L^2}^2 \right| \leq C \left( T\|u_p\|_{L^2}^2 + \|u_1\|_{L^2}^3 + \|\psi(T)-\varphi_0\|_{L^2}^2 \right).
    \end{equation}
    By the Gagliardo--Nirenberg inequality (see \cite{Gagliardo1959,Nirenberg1959}), there exists $C>0$ such that, for every $T > 0$ and $u \in L^2((0,T);\R)$,
    \begin{equation}
        \begin{split}
            \|u_1\|_{L^2}^3 = \|u_p^{(p-1)}\|_{L^2}^3
            & \leq C \left( \|u_p^{(3(p-1))}\|_{L^2} \|u_p\|_{L^2}^2
            +T^{-3(p-1)} \|u_p\|_{L^2}^3 \right)
            \\ & \leq
            C \left( \|u_1^{(2p-2)}\|_{L^2} + T^{-2p+2} \|u_1\|_{L^2} \right) \|u_p\|_{L^2}^2,
        \end{split}
    \end{equation}
    which implies \eqref{eq:drift-Wp}.
\end{proof}

\section{Some classical lemmas}
\label{sec:appendix}

\subsection{Bessel--Parseval estimates}

We recall the following elementary inequality of Bessel--Parseval type.

\begin{lemma}
    \label{lem:bessel}
    For all $T > 0$ and $g \in L^2((0,T);\C)$,
    \begin{equation}
        \sum_{j=0}^{+\infty} \left| \int_0^T g(t) e^{i \lambda_j t} \dd t \right|^2 \leq (1+T) \| g \|_{L^2}^2.
    \end{equation}
\end{lemma}

\begin{proof}
    Let $N := \lceil \pi T/2 \rceil$ and set $T_0 := 2N / \pi$.
    Then $T \leq T_0 < T+1$.
    Extend $g$ by $0$ to $(0,T_0)$.
    This keeps $\|g\|_{L^2}$ and the integrals unchanged.
    Consider the orthonormal basis $\{\phi_k\}_{k\in\Z}$ of $L^2((0,T_0);\C)$ defined by $\phi_k(t) := T_0^{-\frac 12} e^{(2ik\pi t)/T_0}$.
    Since $\lambda_j = (j\pi)^2$ (see \eqref{eq:lambda_j}),
    \begin{equation}
        \int_0^T g(t) e^{i \lambda_j t} \dd t = T_0^{\frac 12} \langle g, \phi_{j^2 N} \rangle.
    \end{equation}
    Therefore, by Bessel (or Parseval) for the orthonormal basis $\{\phi_k\}$,
    \begin{equation}
        \sum_{j=0}^{\infty}\left|\int_0^{T_0} g(t) e^{i\lambda_j t}\dd t\right|^2
        = T_0 \sum_{j=0}^{\infty} |\langle g,\phi_{j^2N}\rangle|^2
        \le T_0 \sum_{k\in\mathbb Z} |\langle g,\phi_k\rangle|^2
        = T_0 \|g\|_{L^2(0,T_0)}^2.
    \end{equation}
    Since $\|g\|_{L^2(0,T_0)}=\|g\|_{L^2(0,T)}$ and $T_0 \leq T+1$, this concludes the proof.
\end{proof}

\subsection{Smoothing lemma}

The proofs of the well-posedness of \eqref{eq:psi} stated in \cref{p:WP} and of the estimates of \cref{p:estimates-classical} rely on the following lemma, which is an adaptation to our Neumann boundary conditions of the smoothing statement \cite[Lemma 1]{BeauchardLaurent2010} for Dirichlet boundary conditions.

\begin{lemma}
    \label{lem:smoothing}
    There exists $C > 0$ such that, for all $T > 0$ and $f \in L^2((0,T);H^2((0,1);\C))$, the function $F : t \mapsto \int_0^t e^{i A s} f(s) \dd s$ belongs to $C^0([0,T];H^2_N((0,1);\C))$ and moreover
    \begin{equation}
        \| F \|_{L^\infty H^2_N} \leq C (1+T) \| f \|_{L^2 H^2}.
    \end{equation}
\end{lemma}

\begin{proof}
    Although $e^{iAs}$ preserves $H^2_N$, the difficulty is that $f(s)$ is only assumed to be in $H^2$ (not in $H^2_N$).
    First, by standard arguments, recalling \eqref{eq:H2N},
    \begin{equation}
        H^2_N((0,1);\C) = \left\{ \varphi \in L^2((0,1);\C) ; \enskip \sum_{j=1}^\infty | j^2 \langle \varphi, \varphi_j \rangle |^2 < \infty \right\}.
    \end{equation}
    Thus, it suffices to check that $\sum_{j=1}^\infty | j^2 \langle F(t), \varphi_j \rangle |^2$ is finite and tends to $0$ as $t \to 0$.

    For $s \in [0,T]$ and $j \geq 1$, integrating by parts,
    \begin{equation}
        \langle f(s), \varphi_j \rangle =
        \sqrt{2} \frac{(-1)^j \partial_x f(s,1) - \partial_x f(s,0)}{j^2 \pi^2} - \frac{1}{j^2 \pi^2} \langle \partial_x^2 f(s), \varphi_j \rangle.
    \end{equation}
    Thus, noting that $\frac{6}{\pi^4} \leq 1$ and using \cref{lem:bessel} for the first 2 terms,
    \begin{equation}
        \begin{split}
            \sum_{j=1}^\infty & | j^2 \langle F(t), \varphi_j \rangle |^2
            = \sum_{j=1}^\infty \left| j^2 \int_0^t e^{i \lambda_j s} \langle f(s), \varphi_j \rangle \dd s\right|^2 \\
            & \leq \sum_{j=1}^\infty \left| \int_0^t e^{i \lambda_j s} \partial_x f(s,1) \dd s\right|^2 + \left| \int_0^t e^{i \lambda_j s} \partial_x f(s,0) \dd s\right|^2 + \left| \int_0^t e^{i \lambda_j s} \langle \partial_x^2 f(s), \varphi_j \rangle \dd s\right|^2 \\
            & \leq \|\partial_x f(\cdot, 1) \|_{L^2(0,t)}^2
            + \|\partial_x f(\cdot, 0) \|_{L^2(0,t)}^2
            + t \|\partial_x^2 f \|_{L^2((0,t);L^2(0,1))}^2
        \end{split}
    \end{equation}
    which both yields the continuity and the claimed estimate.
\end{proof}

\subsection{Integrating integral operators by parts}

The following result is a minor generalization of the statement \cite[Lemma 3.8]{CoronKoenigNguyen2024} in a sesquilinear setting $(u,v)$ instead of the quadratic one $(u,u)$.
We use the notations of \eqref{eq:I-I2-Delta}.

\begin{lemma}
    \label{lem:IPP-Armand}
    Let $T > 0$ and $K \in C^0([0,T]^2;\C)$.
    Assume that $K \in H^2(\Delta_-)$ and $K \in H^2(\Delta_+)$.
    Let $w(s):=\partial_1 K(s,s+0)-\partial_1 K(s,s-0)$ for $s \in (0,T)$.
    Then, for $u, v \in L^2((0,T);\C)$,
    \begin{equation}
        \label{eq:IPP-Armand}
        \begin{split}
            \int_{I^2} & K(s,t) u(s) \bar{v}(t) \dd s \dd t
            =
            \int_I w(s)u_1(s)\bar{v}_1(s) \dd s
            + \int_{\Delta_+ \cup \Delta_-} \partial_{21} K(s,t)u_1(s)\bar{v}_1(t)\dd s \dd t
            \\ &
            - u_1(T) \int_I \partial_2 K(T,t) \bar{v}_1(t) \dd t
            - \bar{v}_1(T) \int_I \partial_1 K(s,T) u_1(s) \dd s
            + K(T,T)u_1(T)\bar{v}_1(T)
        \end{split}
    \end{equation}
    where $u_1(t) := \int_0^t u$ and $v_1(t) := \int_0^t v$.
\end{lemma}

\begin{proof}
    As for \cite[Lemma 3.8]{CoronKoenigNguyen2024}, the proof consists in a few integration by parts.
    Since $K \in H^2(\Delta_\pm)$, its first-order derivatives have traces on each side of the diagonal.
    In particular, $w$ is well defined.
    First, integrating by parts in $s$, we obtain
    \begin{equation}
        \int_{I^2} K(s,t) u(s) \bar{v}(t) \dd s \dd t = - \int_{I^2} \partial_1 K(s,t) u_1(s) \bar{v}(t) \dd s \dd t + u_1(T) \int_I K(T,t) \bar{v}(t) \dd t.
    \end{equation}
    Integrating by parts in $t$ in the second term yields
    \begin{equation}
        u_1(T) \int_I K(T,t) \bar{v}(t) \dd t = K(T,T) u_1(T) \bar{v}_1(T) - u_1(T) \int_I \partial_2 K(T,t) \bar{v}_1(t) \dd t.
    \end{equation}
    We now focus on the first term.
    In $\Delta_+$, integrating by parts in $t$ leads to
    \begin{equation}
        \begin{split}
            - \int_{s=0}^T u_1(s) \int_{t=s}^T & \partial_1 K(s,t) \bar{v}(t) \dd t \dd s =
            \int_{\Delta_+} \partial_{21} K(s,t) u_1(s) \bar{v}_1(t) \dd s \dd t
            \\
            & - \bar{v}_1(T) \int_I u_1(s) \partial_1 K(s,T) \dd s
            + \int_I \partial_1 K(s,s+0) u_1(s) \bar{v}_1(s) \dd s.
        \end{split}
    \end{equation}
    Proceeding likewise in $\Delta_-$ and gathering all terms leads to the claimed equality.
\end{proof}

\begin{example}
    \label{ex:kernel-ab}
    Let $\alpha, \beta \in \C$.
    For all $u,v \in \mathcal{H}$ (see \eqref{eq:cH}),
    \begin{equation}
        \int_{I^2} \left(\alpha |t-s| + \beta \right) u(s) \overline{v}(t) \dd s \dd t
        = - 2 \alpha \langle u_1, v_1 \rangle.
    \end{equation}
\end{example}

\subsection{An estimate for a principal value integral}

\begin{lemma}
    \label{lem:pv_bound}
    For $\phi \in C^2(\R)$ and $M > 0$, one has
    \begin{equation}
        \label{eq:pv_bound}
        \left| \int_{-M}^{+M} \frac{\phi(\omega)-\phi(0)}{\omega} \dd \omega \right| \leq 2M |\phi'(0)| + M \int_{-M}^{+M} | \phi'' |
        \leq \int_{-M}^M |\phi'| + 2 M \int_{-M}^M |\phi''|.
    \end{equation}
\end{lemma}

\begin{proof}
    Substituting Taylor's expansion $\phi(\omega) - \phi(0) = \omega \phi'(0) + \int_0^\omega (\omega - z) \phi''(z) \dd z$ into the integrand, we obtain
    \begin{equation}
        \int_{-M}^{+M} \frac{\phi(\omega)-\phi(0)}{\omega} \dd \omega = 2 M \phi'(0) + \int_{-M}^{+M} \int_0^\omega \left(1 - \frac{z}{\omega}\right) \phi''(z) \dd z \dd \omega,
    \end{equation}
    which implies the first inequality of \eqref{eq:pv_bound}.
    Moreover,
    \begin{equation}
        \phi'(0) = \frac{1}{2M} \int_{-M}^M \Big(\phi'(\omega) - \int_0^\omega \phi''(z) \dd z\Big) \dd \omega,
    \end{equation}
    which implies the second inequality.
\end{proof}

\subsection{Estimates with fractional Sobolev norms}

\begin{lemma}
    Let $\nu \in [0,\frac12)$.
    There exists $C_\nu > 0$ such that, for all $T > 0$,
    \begin{align}
        \label{eq:norm-1-H-nu}
        \nsob{\mathbbm{1}_{[0,T]}}{-\nu} & \leq C_\nu T^{\frac 12 + \nu}, \\
        \label{eq:norm-1-H+nu}
        \nsob{\mathbbm{1}_{[0,T]}}{+\nu} & \leq C_\nu T^{\frac 12} + C_\nu T^{\frac 12 - \nu}, \\
        \label{eq:norm-eit-H+nu}
        \nsob{e^{i \frac{\lambda_k}{2} t} \mathbbm{1}_{[0,T]}(t)}{+\nu} & \leq C_\nu T^{\frac 12} + C_\nu T^{\frac 12 - \nu}.
    \end{align}
\end{lemma}

\begin{proof}
    Let $f:= \mathbbm{1}_{[0,T]}$.
    A direct computation shows that $\widehat{f}(\omega) = T h(\omega T)$ where
    \begin{equation}
        h(\sigma) := \frac{i}{\sigma}(e^{-i \sigma} - 1)
    \end{equation}
    so that $h \in C^\infty(\R;\R)$ with $|h(\sigma)| \leq 1$ and $|\sigma h(\sigma)| \leq 2$.

    First, recalling \eqref{eq:Hnu}, and using the change of variables $\sigma = T \omega$, we write
    \begin{equation}
        \begin{split}
            \nsob{f}{-\nu}^2
            = \frac{T^2}{2 \pi} \int_\R h^2(\omega T) \langle \omega \rangle^{- 2 \nu} \dd \omega
            & \leq \frac{T^{1+2\nu}}{2 \pi} \int_\R h^2(\omega T) |T\omega|^{-2\nu} T \dd \omega
            \\ & \leq \frac{T^{1+2\nu}}{2 \pi} \int_\R h^2(\sigma) |\sigma|^{-2\nu} \dd \sigma,
        \end{split}
    \end{equation}
    where the integral converges near $0$ using $\nu < \frac 12$.

    Second, recalling \eqref{eq:Hnu}, and using the change of variables $\sigma = T \omega$, we write
    \begin{equation}
        \begin{split}
            \nsob{f}{\nu}^2
            = \frac{T^2}{2 \pi} \int_\R h^2(\omega T) \langle \omega \rangle^{2 \nu} \dd \omega
            & \leq \frac{T^2}{2 \pi} \int_\R h^2(\omega T) (1+|\omega|^{2\nu}) \dd \omega
            \\ & \leq \frac{T}{2 \pi} \int_\R h^2 + \frac{T^{1-2\nu}}{2 \pi} \int_\R h^2(\sigma) \sigma^{2\nu} \dd \sigma,
        \end{split}
    \end{equation}
    where the integral converges near $\pm \infty$ using $\nu < \frac 12$.

    Third, let $g(t) := e^{i \omega_0 t} \mathbbm{1}_{[0,T]}(t)$ with $\omega_0 := \frac{\lambda_k}{2}$.
    Then $\widehat{g}(\omega) = T h((\omega-\omega_0)T)$.
    Writing $\langle \omega \rangle^{2\nu} \leq 1 + |\omega-\omega_0|^{2\nu} + |\omega_0|^{2\nu}$ and proceeding as in the previous case yields \eqref{eq:norm-eit-H+nu}.
\end{proof}

\begin{lemma}
    Let $\nu \in [0,\frac12)$ and $\omega_0 \in \R$.
    There exists $C > 0$ such that, for all $T \in (0,1]$,
    \begin{equation}
        \nsob{e^{i \omega_0 t}}{\nu} \leq C.
    \end{equation}
\end{lemma}

\begin{proof}
    Let $f_T(t) := e^{i \omega_0 t} \mathbbm{1}_{[0,T]}(t)$.
    Its Fourier transform is $\widehat{f_T}(\omega) = \widehat{\mathbbm{1}_{[0,T]}}(\omega - \omega_0)$.
    A direct computation shows that $\widehat{\mathbbm{1}_{[0,T]}}(\sigma) = T h(T \sigma)$, where $h(\sigma) := i(e^{-i\sigma}-1)/\sigma$.
    Thus, $\widehat{f_T}(\omega) = T h(T(\omega - \omega_0))$.
    Using the definition of the Sobolev norm, and the change of variables $\sigma = T(\omega - \omega_0)$, we write
    \begin{equation}
        \begin{split}
            \nsob{f_T}{\nu}^2
            & = \frac{T^2}{2\pi} \int_\R |h(T(\omega - \omega_0))|^2 \langle \omega \rangle^{2\nu} \dd\omega
            \\ & = \frac{T}{2\pi} \int_\R |h(\sigma)|^2 \langle \omega_0 + \sigma/T \rangle^{2\nu} \dd\sigma.
        \end{split}
    \end{equation}
    Using the inequality $\langle x+y \rangle^{2\nu} \leq 2^\nu (\langle x \rangle^{2\nu} + \langle y \rangle^{2\nu})$, we can bound the norm by the sum of two terms:
    \begin{equation}
        \nsob{f_T}{\nu}^2 \leq \frac{T 2^\nu \langle \omega_0 \rangle^{2\nu}}{2\pi} \int_\R |h(\sigma)|^2 \dd\sigma + \frac{T 2^\nu}{2\pi} \int_\R |h(\sigma)|^2 \langle \sigma/T \rangle^{2\nu} \dd\sigma.
    \end{equation}
    The first integral is a constant. Since $T \leq 1$, the first term is bounded.
    For the second term, we write $\langle \sigma/T \rangle^{2\nu} = (1 + \sigma^2/T^2)^\nu = T^{-2\nu} (T^2 + \sigma^2)^\nu$. This gives
    \begin{equation}
        \frac{T 2^\nu}{2\pi} \int_\R |h(\sigma)|^2 \langle \sigma/T \rangle^{2\nu} \dd\sigma
        = \frac{T^{1-2\nu} 2^\nu}{2\pi} \int_\R |h(\sigma)|^2 (T^2 + \sigma^2)^\nu \dd\sigma.
    \end{equation}
    Since $T \in (0,1]$ and $1-2\nu > 0$, we have $T^2 \leq 1$ and $T^{1-2\nu} \leq 1$.
    Therefore, this second term is bounded by
    \begin{equation}
        \frac{2^\nu}{2\pi} \int_\R |h(\sigma)|^2 (1 + \sigma^2)^\nu \dd\sigma
        = \frac{2^\nu}{2\pi} \int_\R |h(\sigma)|^2 \langle \sigma \rangle^{2\nu} \dd\sigma.
    \end{equation}
    The function $h$ is bounded near the origin and decays as $1/|\sigma|$ for large $|\sigma|$. The integrand thus behaves like $|\sigma|^{2\nu-2}$ for large $|\sigma|$. As $\nu < 1/2$, the integral converges.
    Both terms are bounded by a constant independent of $T$, which concludes the proof.
\end{proof}

\begin{lemma}
    \label{lem:H-beta-T-beta}
    Let $\nu \in [0,\frac 12)$.
    There exists $C_\nu > 0$ such that, for all $T > 0$ and $u \in L^2((0,T);\C)$,
    \begin{equation}
        \label{eq:H-beta}
        \nsob{u}{-\nu} \leq C_\nu T^\nu \| u \|_{L^2}.
    \end{equation}
\end{lemma}

\begin{proof}
    First, \eqref{eq:H-beta} holds for all $u \in \mathcal{H}$ (i.e.\ such that $u_1(T) = 0$).
    Indeed, for such $u$, one has $\widehat{u}(\omega) = i \omega \widehat{u_1}(\omega)$, so $\nsob{u}{-1} \leq \| u_1 \|_{L^2}$ using $|\omega \langle \omega \rangle^{-1} | \leq 1$.
    Moreover, $\nsob{u}{0} = \| u \|_{L^2}$ and $\| u_1 \|_{L^2} \leq T \| u \|_{L^2}$, so \eqref{eq:H-beta} holds for $\nu = 0$ and $\nu = 1$.
    By Hölder's inequality, for $\nu \in (0,1)$,
    \begin{equation}
        \nsob{u}{-\nu} \leq \nsob{u}{-1}^{\nu} \nsob{u}{0}^{1-\nu}
    \end{equation}
    so \eqref{eq:H-beta} follows for all $u \in \mathcal{H}$ and $\nu \in [0,1]$ with $C_\nu = 1$.

    Let $\nu \in [0,\frac 12)$.
    For a general $u \in L^2((0,T);\C)$, write $u = v + \frac{u_1(T)}{T} \mathbbm{1}_{[0,T]}$ where $v = \PH u \in \mathcal{H}$.
    Then
    \begin{equation}
        \nsob{u}{-\nu} \leq \nsob{v}{-\nu} + \frac{|u_1(T)|}{T} \nsob{\mathbbm{1}_{[0,T]}}{-\nu}
        \leq T^\nu \| v \|_{L^2} + C_\nu \frac{|u_1(T)|}{T} T^{\frac 12 + \nu}
    \end{equation}
    using the first step and \eqref{eq:norm-1-H-nu} (since $\nu < \frac 12$).

    We conclude using $\|v\|_{L^2} \leq \|u\|_{L^2}$ and $|u_1(T)| \leq T^{\frac 12} \|u\|_{L^2}$.
\end{proof}

\begin{lemma}
    \label{lem:mul-1-Hnu}
    Let $\nu \in (-\frac 12,\frac 12)$.
    There exists $C_\nu > 0$ such that, for any segment $(a,b) \subset \R$ and $u \in H^\nu(\R;\C)$, $\| \mathbbm{1}_{(a,b)} u \|_{H^\nu(\R)} \leq C_\nu \| u \|_{H^\nu(\R)}$.
\end{lemma}

\begin{proof}
    It suffices to know that the restriction operator to a half-line acts continuously on $H^\nu(\R;\C)$ for $\nu \in (-\frac 12,\frac 12)$ (see \cite[Section 4.6.3, Theorem 1]{RunstSickel1996}) and write $\mathbbm{1}_{(a,b)} = \mathbbm{1}_{(a,\infty)} \mathbbm{1}_{(-\infty,b)}$.
\end{proof}

\begin{lemma}
    \label{lem:mul-t}
    Let $\nu \in (-\frac12,\frac12)$.
    There exists $C_\nu > 0$ such that, for every $T > 0$, and every control $u \in L^2 \cap \sobC{\nu}$,
    \begin{equation}
        \nsob{t u(t)}{\nu} \leq C_\nu T \nsob{u}{\nu}.
    \end{equation}
\end{lemma}

\begin{proof}
    For $T > 0$ and $u \in L^2 \cap \sobC{\nu}$, let $\widetilde{u}$ denote its extension by $0$ to $\R$.
    By definition \eqref{eq:Hnu}, $\nsob{u}{\nu} =\|\widetilde{u}\|_{H^\nu(\R)}$ and $\nsob{tu(t)}{\nu}=\|t\widetilde{u}(t)\|_{H^\nu(\R)}$.
    First, for $\nu = 0$, it is clear that $\| t u(t) \|_{L^2} \leq T \| u \|_{L^2}$.
    Second, for a given $\nu \in (0,\frac 12)$, we know that the $H^\nu(\R)$ norm is equivalent (up to a constant depending on $\nu$) to $\| \cdot \|_{L^2} + [\cdot]_\nu$ where the Gagliardo seminorm is defined as
    \begin{equation}
        [f]_\nu^2 
        := \int_\R \int_\R \frac{|f(t)-f(s)|^2}{|t-s|^{1+2\nu}} \dd s \dd t
        = \int_0^T \int_0^T \frac{|f(t)-f(s)|^2}{|t-s|^{1+2\nu}} \dd s \dd t + \frac{1}{\nu} \int_0^T |f(t)|^2 A_T(t) \dd t
    \end{equation}
    when $\supp f \subset (0,T)$, where $A_T(t) := t^{-2\nu} + (T-t)^{-2\nu}$.
    Thus, we have
    \begin{equation}
        \begin{split}
        [t u(t)]_\nu^2 & \leq \int_0^T \int_0^T \frac{|tu(t)-su(s)|^2}{|t-s|^{1+2\nu}} \dd s \dd t + \frac{1}{\nu} \int_0^T |tu(t)|^2 A_T(t) \dd t \\
        & \leq \int_0^T \int_0^T \frac{t^2 |u(t)-u(s)|^2}{|t-s|^{1+2\nu}}+ \frac{|t-s|^2 |u(s)|^2}{|t-s|^{1+2\nu}} \dd s \dd t + \frac{1}{\nu} \int_0^T |t u(t)|^2 A_T(t) \dd t \\
        & \leq T^2 [u]_\nu^2 + 2 T^{2-2\nu} \| u \|_{L^2}^2
        \end{split}
    \end{equation}
    so we conclude using $\| f \|_{L^2}^2 \leq \nu 4^{-\nu} T^{2\nu} [ f ]_\nu^2$ (which holds since $A_T(t) \geq (T/2)^{-2\nu}$).

    Third, we proceed by duality for $u \in L^2 \cap H^{-\nu}(\R;\C)$ with $\nu \in (0,\frac12)$.
    Then
    \begin{equation}
        \nsob{t u(t)}{-\nu}
        = \| t\widetilde{u}(t) \|_{H^{-\nu}(\R)}
        = \underset{v \in H^\nu(\R), \|v\|_{H^\nu(\R)} = 1}{\sup} \int_\R t \widetilde{u}(t) v(t) \dd t.
    \end{equation}
    For such $v \in H^\nu(\R)$, let $w := v \mathbbm{1}_{(0,T)} \in L^2(0,T)$.
    By \cref{lem:mul-1-Hnu}, $\| w \|_{H^\nu(\R)} \leq C_\nu \| v \|_{H^\nu(\R)}$.
    By the previous case $\| t w(t) \|_{H^\nu(\R)} \leq T C_\nu \| w \|_{H^\nu(\R)}$.
    Thus $\nsob{t u(t)}{-\nu} \leq T C_\nu \nsob{u}{-\nu}$.
\end{proof}

\begin{lemma}
    \label{lem:fourier-holder}
    Let $\alpha \in [0,1]$ and $\rho \in C^{0,\alpha}_{\mathrm{loc}}(\R;\C)$.
    There exists $C > 0$ such that, for all $T \in (0,1]$,
    \begin{equation}
        | \mathcal{F}(\rho \mathbbm{1}_{[-T,T]})(\omega) | \leq C \langle \omega \rangle^{-\alpha} .
    \end{equation}
\end{lemma}

\begin{proof}
    For $T \in (0,1]$, let $\rho_T := \rho \mathbbm{1}_{[-T,T]}$.
    For $|\omega| \leq 1$, we just use $|\widehat{\rho_T}(\omega)| \leq 2 T \| \rho \|_{L^\infty(-1,1)}$.
    For $|\omega| \geq 1$, let $h := (\sign \omega) / |\omega|$.
    A change of variables yields
    \begin{equation}
        \int_\R (\rho_T(t+h) - \rho_T(t)) e^{-i \omega t} \dd t = (e^{i\omega h} -1 ) \widehat{\rho_T}(\omega)
        = (e^i - 1) \widehat{\rho_T}(\omega).
    \end{equation}
    Since $|e^i-1| > 0$, it suffices to bound the left-hand side.
    If $-T+|h| < t < T-|h|$ (an interval of length at most $2T$), we have
    \begin{equation}
        |\rho_T(t+h) - \rho_T(t)| = |\rho(t+h)-\rho(t)| \leq |h|^\alpha \| \rho \|_{C^{0,\alpha}([-1,1])} = |\omega|^{-\alpha} \| \rho \|_{C^{0,\alpha}([-1,1])}.
    \end{equation}
    If $t < -T-|h|$ or $t > T+|h|$, the integrand vanishes.
    On the remaining intervals $[-T-|h|,-T+|h|]$ and $[T-|h|,T+|h|]$, we use the crude bound $|\rho_T(t+h) - \rho_T(t)| \leq 2 \|\rho\|_{L^\infty(-1,1)}$, but they are of length $2|h| = 2 |\omega|^{-1} \leq 2 |\omega|^{-\alpha}$.
    Gathering all bounds proves the claimed estimate.
\end{proof}

\begin{lemma}
    \label{lem:convol-holder}
    Let $\alpha \in [0,1]$ and $\rho \in C^{0,\alpha}_{\mathrm{loc}}(\R;\C)$.
    There exists $C > 0$ such that, for $T \in (0,1]$ and $u,v \in L^2((0,T);\C)$,
    \begin{equation}
        \bigg| \int_0^T \int_0^T \rho(t-s) u(s) \bar{v}(t) \dd s \dd t \bigg| \leq C \nsob{u}{-\alpha/2} \nsob{v}{-\alpha/2}.
    \end{equation}
\end{lemma}

\begin{proof}
    For $T \in (0,1]$, let $\rho_T := \rho \mathbbm{1}_{[-T,T]}$.
    For $u,v\in L^2((0,T);\C)$, keeping the same notation for their extension by $0$ to $\R$, we have
    \begin{equation}
        I := \int_0^T \int_0^T \rho(t-s) u(s) \bar{v}(t) \dd s \dd t
        = \int_\R \int_\R \rho_T(t-s) u(s) \bar{v}(t) \dd s \dd t
        = \langle \rho_T \star u, v \rangle_{L^2(\R)}.
    \end{equation}
    Thus, by Plancherel's formula
    \begin{equation}
        I = \frac{1}{2 \pi} \int_\R \widehat{\rho_T}(\omega) \widehat{u}(\omega) \widehat{v}(\omega) \dd \omega
        = \frac{1}{2 \pi} \int_\R \widehat{\rho_T}(\omega) \langle \omega\rangle^\alpha \widehat{u}(\omega) \overline{\widehat{v}}(\omega) \langle \omega\rangle^{-\alpha} \dd \omega.
    \end{equation}
    Thus, by \eqref{eq:Hnu} and Cauchy--Schwarz,
    \begin{equation}
        |I| \leq \left(\sup_{\omega \in \R} | \widehat{\rho_T}(\omega) \langle \omega\rangle^\alpha | \right) \nsob{u}{-\alpha/2} \nsob{v}{-\alpha/2}.
    \end{equation}
    Hence, the conclusion follows from \cref{lem:fourier-holder}.
\end{proof}

\subsection{Spectral Sobolev norm}
\label{sec:spectral}

In this paragraph, we define a fractional Sobolev norm using spectral theory, which is used in the proof of \cref{p:R-chi}, and equivalent to the norm \eqref{eq:Hnu} -- defined by extending a function from $(0,T)$ to $\R$ by $0$, then using the usual fractional Sobolev norm on $\R$ in Fourier space.

Let $T > 0$.
Consider the Dirichlet Laplacian on $(0,T)$, its eigenvalues $\beta_k^2 := (k \pi / T)^2$ for $k \in \N^*$ and its basis of orthonormal eigenfunctions $e_k \in L^2((0,T);\C)$.
For $\nu \in \R$ and $u \in C^\infty_c((0,T);\C)$, set
\begin{equation}
    N_{\mathrm{sp},\nu}(u) := \left(\sum_{k = 1}^{\infty} \langle \beta_k \rangle^{2\nu} |\langle u, e_k\rangle|^2\right)^{\frac 12}
    \quad \text{where} \quad
    e_k(t) := \sqrt{\frac 2 T} \sin (\beta_k t).
\end{equation}

\begin{lemma}
    \label{lem:Hspec<=Hnu}
    Let $\nu \in (-\frac12,\frac12)$.
    There exists $C_\nu > 0$ such that, for $T > 0$ and $u \in C^\infty_c((0,T);\R)$,
    \begin{equation}
        N_{\mathrm{sp},\nu}(u) \leq C_\nu \nsob{u}{\nu}.
    \end{equation}
\end{lemma}

\begin{proof}
    Let $u \in C^\infty_c((0,T);\R)$.
    By definition,
    \begin{equation}
        \begin{split}
            N_{\mathrm{sp},\nu}^2(u) & = \frac{2}{T} \sum_{k=1}^{\infty} \langle\beta_k\rangle^{2\nu} \left| \int_0^T u(t) \sin(\beta_k t) \dd t \right|^2 \\
            & \leq \frac{1}{T} \sum_{k=1}^{\infty} \langle\beta_k\rangle^{2\nu} \left( |\widehat{u}(-\beta_k)|^2 + |\widehat{u}(\beta_k)|^2 \right),
        \end{split}
    \end{equation}
    since the coefficients are related to the Fourier transform $\widehat{u}(\omega) := \int_0^T u(t) e^{-i\omega t} \dd t$ by the identity
    \begin{equation}
        \int_0^T u(t) \sin(\beta_k t) \dd t = \frac{\widehat{u}(-\beta_k) - \widehat{u}(\beta_k)}{2i}.
    \end{equation}
    Let $I_k := [\frac{k\pi}{T}, \frac{(k+1)\pi}{T})$ for $k \in \N^*$.
    Since the length of each interval $I_k$ is $\pi/T$, we can write $\frac{1}{T} = \frac{1}{\pi} \int_{I_k} \dd\omega$.
    Moreover, for $\omega \in I_k$, $\beta_k \leq \omega \leq 2 \beta_k$ and thus there exists $C_\nu > 0$ such that $\langle \beta_k \rangle^{2\nu} \leq C_\nu \langle\omega\rangle^{2\nu}$.
    By the triangular inequality,
    \begin{equation}
        \begin{split}
            N_{\mathrm{sp},\nu}^2(u) & \leq \frac{2 C_\nu}{\pi} \sum_{k = 1}^\infty \int_{I_k} \langle\omega\rangle^{2\nu} (|\widehat{u}(\omega)|^2 + |\widehat{u}(-\omega)|^2) \dd \omega \\
            & \quad + \frac{2}{\pi} \sum_{k=1}^{\infty} \langle\beta_k\rangle^{2\nu} \int_{I_k} \left( |\widehat{u}(-\beta_k)-\widehat{u}(-\omega)|^2 + |\widehat{u}(\beta_k)-\widehat{u}(\omega)|^2 \right) \dd\omega.
        \end{split}
    \end{equation}
    The first term is bounded above by $4 C_\nu \nsob{u}{\nu}^2$.
    Let us focus on bounding the remainder involving positive frequencies.
    We write using Cauchy--Schwarz,
    \begin{equation}
        |\widehat{u}(\beta_k) - \widehat{u}(\omega)|^2 = \left| \int_{\beta_k}^\omega \partial_\omega \widehat{u}(\omega') \dd \omega' \right|^2
        \leq \frac{\pi}{T} \int_{I_k} | \partial_\omega \widehat{u}(\zeta)|^2 \dd \zeta.
    \end{equation}
    Thus
    \begin{equation}
        \langle\beta_k\rangle^{2\nu} \int_{I_k} |\widehat{u}(\beta_k) - \widehat{u}(\omega)|^2 \dd \omega \leq C_\nu \frac{\pi^2}{T^2} \int_{I_k} \langle \zeta\rangle^{2\nu} | \partial_\omega \widehat{u}(\zeta)|^2 \dd \zeta.
    \end{equation}
    Hence
    \begin{equation}
        \frac{2}{\pi} \sum_{k=1}^{\infty} \langle\beta_k\rangle^{2\nu} \int_{I_k} \left( |\widehat{u}(-\beta_k)-\widehat{u}(-\omega)|^2 + |\widehat{u}(\beta_k)-\widehat{u}(\omega)|^2 \right) \dd\omega
        \leq \frac{2 C_\nu \pi}{T^2} \int_\R \langle \zeta\rangle^{2\nu} | \partial_\omega \widehat{u}(\zeta)|^2 \dd \zeta.
    \end{equation}
    Since $\partial_\omega \widehat{u}(\omega) = \mathcal{F}[-itu(t)](\omega)$, we have
    \begin{equation}
        N_{\mathrm{sp},\nu}^2(u) \leq 4 C_\nu \nsob{u}{\nu}^2 + \frac{4 C_\nu \pi^2}{T^2} \nsob{t u(t)}{\nu}^2.
    \end{equation}
    By \cref{lem:mul-t}, for any $T > 0$, we have $\nsob{t u(t)}{\nu} \leq C_\nu T \nsob{u}{\nu}$, which concludes the proof.
\end{proof}

\begin{lemma}
    \label{lem:H-nu<=Hspec}
    Let $\nu \in (-\frac 12,\frac 12)$.
    There exists $C_\nu > 0$ such that, for $T > 0$ and $u \in C^\infty_c((0,T);\R)$
    \begin{equation}
        \nsob{u}{\nu} \leq C_\nu N_{\mathrm{sp},\nu}(u).
    \end{equation}
\end{lemma}

\begin{proof}
    Let $T > 0$ and $u \in C^\infty_c((0,T);\R)$.
    Let $\widetilde{u}$ denote its extension by $0$ to $\R$.
    Then
    \begin{equation}
        \nsob{u}{\nu}
        = \| \widetilde{u} \|_{H^\nu(\R)}
        = \underset{v \in H^{-\nu}(\R), \|v\|_{H^{-\nu}(\R)} = 1}{\sup} \int_\R \widetilde{u} v.
    \end{equation}
    For such $v \in H^{-\nu}(\R)$, let $w := v \mathbbm{1}_{(0,T)} \in L^2(0,T)$.
    By \cref{lem:mul-1-Hnu}, $\| w \|_{H^{-\nu}(\R)} \leq C_\nu \| v \|_{H^{-\nu}(\R)}$.
    Moreover,
    \begin{equation}
        \left| \int_\R \widetilde{u} v \right|
        = \left| \int_0^T u w \right|
        = \left| \sum_{k = 1}^\infty \langle u, e_k \rangle \langle w, e_k \rangle \right|
        \leq N_{\mathrm{sp},\nu}(u) N_{\mathrm{sp},-\nu}(w).
    \end{equation}
    By \cref{lem:Hspec<=Hnu}, $N_{\mathrm{sp},-\nu}(w) \leq C_\nu \| w \|_{H^{-\nu}(\R)} \leq C_\nu^2$, which proves the claimed estimate.
\end{proof}

\begin{remark}
    For $\nu = 0$ and $\nu = 1$, by Plancherel and Parseval, one has $\nsob{\cdot}{\nu} = N_{\mathrm{sp},\nu}$ for any $T > 0$.
    Unfortunately, we have not been able to use this fact prove the equivalence of both norms for $\nu \in (-\frac12,\frac12)$ with a constant independent of $T$ (at least as $T \to 0$).
    It is known that these norms cease to be exactly equal for $\nu \in (0,1)$ (see e.g.\ \cite[Example 4.15]{ChandlerWildeHewettMoiola2015} for $T = 1$).
\end{remark}

\section*{Acknowledgments}

All three authors acknowledge support from the Fondation Simone et Cino Del Duca -- Institut de France, which is the main sponsor of this work.

KB and FM also acknowledge support from grant ANR-20-CE40-0009 (Project TRECOS), while KB and TP also acknowledge support from grant ANR-11-LABX-0020 (Labex Lebesgue).

\bibliographystyle{plain}
\bibliography{biblio}

@article{BatalOzsari2015,
	title = {Nonlinear {Schrödinger} equation on the half-line with nonlinear boundary condition},
	author = {Batal, Ahmet and Özsarı, Türker},
    fjournal = {Electronic Journal of Differential Equations (EJDE)},
    journal = {Electron. J. Differ. Equ.},
    issn = {1072-6691},
    volume = {2016},
    pages = {20},
    note = {Id/No 222},
    year = {2016},
    language = {English},
    zbmath = {6637292},
    zbl = {1351.35179},
    url = {https://ejde.math.txstate.edu/Volumes/2016/222/batal.pdf}
}

@article{BeauchardLeBorgneMarbach2023,
	author = {Karine Beauchard and J\'er\'emy Le Borgne and Fr\'ed\'eric Marbach},
	title = {On expansions for nonlinear systems, error estimates and convergence issues},
	journal = {Comptes Rendus. Math\'ematique},
	pages = {97--189},
	publisher = {Acad\'emie des sciences, Paris},
	volume = {361},
	year = {2023},
	doi = {10.5802/crmath.395},
	language = {en},
    url = {https://comptes-rendus.academie-sciences.fr/mathematique/articles/10.5802/crmath.395/}
}

@article{BeauchardLaurent2010,
	title = {Local controllability of {1D} linear and nonlinear {S}chr{\"o}dinger equations with bilinear control},
	author = {Beauchard, Karine and Laurent, Camille},
	journal = {Journal de math{\'e}matiques pures et appliqu{\'e}es},
	volume = {94},
	number = {5},
	pages = {520--554},
	year = {2010},
	publisher = {Elsevier},
    doi = {10.1016/j.matpur.2010.04.001},
    url = {https://doi.org/10.1016/j.matpur.2010.04.001}
}

@article{BeauchardMarbach2018,
	author = {Beauchard, Karine and Marbach, Fr\'{e}d\'{e}ric},
	title = {Quadratic obstructions to small-time local controllability for scalar-input systems},
	journal = {J. Differential Equations},
	fjournal = {Journal of Differential Equations},
	volume = {264},
	year = {2018},
	number = {5},
	pages = {3704--3774},
	issn = {0022-0396},
	mrclass = {93B05 (93C10 93C15)},
	mrnumber = {3741402},
	mrreviewer = {Matthias Kawski},
	doi = {10.1016/j.jde.2017.11.028},
	url = {https://doi.org/10.1016/j.jde.2017.11.028}
}

@article{BeauchardMarbach2020,
	title = {Unexpected quadratic behaviors for the small-time local null controllability of scalar-input parabolic equations},
	author = {Beauchard, Karine and Marbach, Fr{\'e}d{\'e}ric},
	journal = {Journal de Math{\'e}matiques Pures et Appliqu{\'e}es},
	volume = {136},
	pages = {22--91},
	year = {2020},
	publisher = {Elsevier},
    doi = {10.1016/j.matpur.2020.02.001},
    url = {https://doi.org/10.1016/j.matpur.2020.02.001}
}

@article{BeauchardMarbach2022,
    author = {Beauchard, Karine and Marbach, Fr{\'e}d{\'e}ric},
    title = {A unified approach of obstructions to small-time local controllability for scalar-input systems},
    journal = {Journal of Dynamical and Control Systems},
    issn = {1079-2724},
    volume = {32},
    number = {1},
    pages = {95},
    year = {2026},
    language = {English},
    doi = {10.1007/s10883-025-09752-1},
    zbMATH = {8144822}
}

@article{BeauchardMorancey2014,
	author = {Beauchard, Karine and Morancey, Morgan},
	title = {Local controllability of 1{D} {S}chr\"odinger equations with bilinear control and minimal time},
	journal = {Math. Control Relat. Fields},
	fjournal = {Mathematical Control and Related Fields},
	volume = {4},
	year = {2014},
	number = {2},
	pages = {125--160},
	issn = {2156-8472},
	mrclass = {93B05 (35Q41 93C20)},
	mrnumber = {3167929},
	mrreviewer = {Andrzej \L ada},
	url = {https://doi.org/10.3934/mcrf.2014.4.125},
	doi = {10.3934/mcrf.2014.4.125},
	zbmath = {6267638},
	zbl = {1281.93016}
}

@inproceedings{BianchiniStefani1986,
	title = {Sufficient conditions of local controllability},
	author = {Bianchini, Rosa Maria and Stefani, Gianna},
	booktitle = {1986 25th IEEE Conference on Decision and Control},
	pages = {967--970},
	year = {1986},
	organization = {IEEE}
}

@article{BonaSunZhang2016,
	author = {Bona, Jerry and Sun, Shu-Ming and Zhang, Bing-Yu},
	title = {Nonhomogeneous boundary-value problems for one-dimensional nonlinear {Schr{\"o}dinger} equations},
	fjournal = {Journal de Math{\'e}matiques Pures et Appliqu{\'e}es. Neuvi{\`e}me S{\'e}rie},
	journal = {J. Math. Pures Appl. (9)},
	issn = {0021-7824},
	volume = {109},
	pages = {1--66},
	year = {2018},
	language = {English},
	doi = {10.1016/j.matpur.2017.11.001},
	zbmath = {6817967},
	zbl = {1379.35287},
	url = {https://doi.org/10.1016/j.matpur.2017.11.001}
}

@article{Bournissou2023_Lin,
    author = {Bournissou, M{\'e}gane},
    title = {Local controllability of the bilinear 1D {Schr{\"o}dinger} equation with simultaneous estimates},
    fjournal = {Mathematical Control and Related Fields},
    journal = {Math. Control Relat. Fields},
    issn = {2156-8472},
    volume = {13},
    number = {3},
    pages = {1047--1080},
    year = {2023},
    language = {English},
    doi = {10.3934/mcrf.2022027},
    zbmath = {7697856},
    zbl = {1518.93016},
    url = {https://doi.org/10.3934/mcrf.2022027},
}

@article{Bournissou2023_Quad,
	title = {{Quadratic behaviors of the 1D linear Schr{\"o}dinger equation with bilinear control}},
	author = {Bournissou, M{\'e}gane},
	journal = {Journal of Differential Equations},
	volume = {351},
	pages = {324--360},
	year = {2023},
	publisher = {Elsevier},
    doi = {10.1016/j.jde.2023.01.007},
    url = {https://doi.org/10.1016/j.jde.2023.01.007}
}

@article{Bournissou2024,
    author = {Bournissou, M{\'e}gane},
    title = {Small-time local controllability of the bilinear {Schr{\"o}dinger} equation with a nonlinear competition},
    fjournal = {European Series in Applied and Industrial Mathematics (ESAIM): Control, Optimization and Calculus of Variations},
    journal = {ESAIM, Control Optim. Calc. Var.},
    issn = {1292-8119},
    volume = {30},
    pages = {38},
    note = {Id/No 2},
    year = {2024},
    language = {English},
    doi = {10.1051/cocv/2023077},
    zbmath = {7798860},
    zbl = {1530.93024},
    url = {https://doi.org/10.1051/cocv/2023077},
}

@article{ChandlerWildeHewettMoiola2015,
    author = {Chandler-Wilde, Simon and Hewett, David and Moiola, Andrea},
    title = {Interpolation of {Hilbert} and {Sobolev} spaces: quantitative estimates and counterexamples},
    fjournal = {Mathematika},
    journal = {Mathematika},
    issn = {0025-5793},
    volume = {61},
    number = {2},
    pages = {414--443},
    year = {2015},
    language = {English},
    doi = {10.1112/S0025579314000278},
    zbmath = {6447875},
    zbl = {1341.46014},
    url = {https://doi.org/10.1112/S0025579314000278},
}

@article{ChowdhuryErvedoza2019,
    author = {Chowdhury, Shirshendu and Ervedoza, Sylvain},
    title = {Open loop stabilization of incompressible {Navier}-{Stokes} equations in a 2d channel using power series expansion},
    fjournal = {Journal de Math{\'e}matiques Pures et Appliqu{\'e}es. Neuvi{\`e}me S{\'e}rie},
    journal = {J. Math. Pures Appl. (9)},
    issn = {0021-7824},
    volume = {130},
    pages = {301--346},
    year = {2019},
    language = {English},
    doi = {10.1016/j.matpur.2019.01.006},
    zbmath = {7107308},
    zbl = {1428.35280},
    url = {https://doi.org/10.1016/j.matpur.2019.01.006}
}

@article{Coron2006,
	author = {Coron, Jean-Michel},
	title = {On the small-time local controllability of a quantum particle in a moving one-dimensional infinite square potential well},
	journal = {C. R. Math. Acad. Sci. Paris},
	fjournal = {Comptes Rendus Math\'ematique. Acad\'emie des Sciences. Paris},
	volume = {342},
	year = {2006},
	number = {2},
	pages = {103--108},
	issn = {1631-073X},
	mrclass = {93B05 (35Q40 81Q99)},
	mrnumber = {2193655},
    doi = {10.1016/j.crma.2005.11.004},
	url = {https://doi.org/10.1016/j.crma.2005.11.004},
	zbmath = {2248350},
	zbl = {1082.93002}
}

@article{CoronKoenigNguyen2022,
	author = {Coron, Jean-Michel and Koenig, Armand and Nguyen, Hoai-Minh},
	title = {On the small-time local controllability of a {KdV} system for critical lengths},
	fjournal = {Journal of the European Mathematical Society (JEMS)},
	journal = {J. Eur. Math. Soc. (JEMS)},
	issn = {1435-9855},
	volume = {26},
	number = {4},
	pages = {1193--1253},
	year = {2024},
	language = {English},
	doi = {10.4171/JEMS/1307},
	zbmath = {7835491},
	zbl = {1535.93006},
    url = {https://doi.org/10.4171/JEMS/1307}
}

@article{CoronKoenigNguyen2024,
	author = {Coron, Jean-Michel and Koenig, Armand and Nguyen, Hoai-Minh},
	title = {Lack of local controllability for a water-tank system when the time is not large enough},
	fjournal = {Annales de l'Institut Henri Poincar{\'e} C. Analyse Non Lin{\'e}aire},
	journal = {Ann. Inst. Henri Poincar{\'e} C, Anal. Non Lin{\'e}aire},
	issn = {0294-1449},
	volume = {41},
	number = {6},
	pages = {1327--1365},
	year = {2024},
	language = {English},
	doi = {10.4171/AIHPC/123},
	zbmath = {7930628},
	zbl = {1553.93018},
    url = {https://doi.org/10.4171/AIHPC/123}
}

@article{DionKellerAtabekBandrauk1999,
	title = {Laser-induced alignment dynamics of {HCN}: roles of the permanent dipole moment and the polarizability},
	author = {Dion, Claude and Keller, Arne and Atabek, Osman and Bandrauk, Andr{\'e}},
	journal = {Physical Review A},
	volume = {59},
	number = {2},
	pages = {1382},
	year = {1999},
	publisher = {APS},
    doi = {10.1103/PhysRevA.59.1382},
    url = {https://doi.org/10.1103/PhysRevA.59.1382}
}

@article{Gagliardo1959,
	author = {Gagliardo, Emilio},
	title = {Ulteriori proprieta di alcune classi di funzioni in piu variabili},
	fjournal = {Ricerche di Matematica},
	journal = {Ric. Mat.},
	issn = {0035-5038},
	volume = {8},
	pages = {24--51},
	year = {1959},
	language = {Italian},
	zbmath = {3318089},
	zbl = {0199.44701}
}

@article{Gerver1970,
    author = {Gerver, Joseph},
    title = {The differentiability of the {Riemann} function at certain rational multiples of {{\(\pi\)}}},
    fjournal = {American Journal of Mathematics},
    journal = {Am. J. Math.},
    issn = {0002-9327},
    volume = {92},
    pages = {33--55},
    year = {1970},
    language = {English},
    doi = {10.2307/2373496},
    zbmath = {3322192},
    zbl = {0203.05904}
}

@article{Gerver1971,
    author = {Gerver, Joseph},
    title = {More on the differentiability of the {Riemann} function},
    fjournal = {American Journal of Mathematics},
    journal = {Am. J. Math.},
    issn = {0002-9327},
    volume = {93},
    pages = {33--41},
    year = {1971},
    language = {English},
    doi = {10.2307/2373445},
    url = {semanticscholar.org/paper/629e04efd2d62fd5a62981d51b656eab79313b9a},
    zbmath = {3362159},
    zbl = {0228.26008}
}

@article{Gherdaoui2025,
    author = {Gherdaoui, Th{\'e}o},
    title = {Small-time local controllability of the multi-input bilinear {Schr{\"o}dinger} equation thanks to a quadratic term},
    fjournal = {European Series in Applied and Industrial Mathematics (ESAIM): Control, Optimization and Calculus of Variations},
    journal = {ESAIM, Control Optim. Calc. Var.},
    issn = {1292-8119},
    volume = {31},
    pages = {49},
    note = {Id/No 44},
    year = {2025},
    language = {English},
    doi = {10.1051/cocv/2024091},
    zbmath = {8056732},
    url = {https://doi.org/10.1051/cocv/2024091}
}

@book{Grisvard1985,
    author = {Grisvard, Pierre},
    title = {Elliptic problems in nonsmooth domains},
    fseries = {Monographs and Studies in Mathematics},
    series = {Monogr. Stud. Math.},
    volume = {24},
    isbn = {0-273-08647-2},
    year = {1985},
    publisher = {Pitman, Boston, MA},
    language = {English},
    zbmath = {4138299},
    zbl = {0695.35060}
}

@article{Hardy1916,
    author = {Hardy, Godfrey},
    title = {\emph{Weierstrass}'s non-differentiable function.},
    fjournal = {Transactions of the American Mathematical Society},
    journal = {Trans. Am. Math. Soc.},
    issn = {0002-9947},
    volume = {17},
    pages = {301--325},
    year = {1916},
    language = {English},
    doi = {10.2307/1989005},
    zbmath = {2610887},
    jfm = {46.0401.03}
}

@article{HermesKawski1987,
	title = {Local controllability of a single input, affine system},
	author = {Hermes, Henry and Kawski, Matthias},
	journal = {Nonlinear analysis and applications},
	volume = {109},
	pages = {235--248},
	year = {1987},
	publisher = {Lecture Notes in Pure and Appl. Math}
}

@inproceedings{Marbach2018,
	title = {An obstruction to small-time local null controllability for a viscous {Burgers}' equation},
	author = {Marbach, Fr{\'e}d{\'e}ric},
	booktitle = {Annales scientifiques de l'{\'E}cole Normale Sup{\'e}rieure},
	volume = {51(5)},
	pages = {1129--1177},
	year = {2018},
    doi = {10.24033/asens.2373},
    url = {https://doi.org/10.24033/asens.2373}
}

@book{McLean2000,
    author = {McLean, William},
    title = {Strongly elliptic systems and boundary integral equations},
    isbn = {0-521-66332-6; 0-521-66375-X},
    year = {2000},
    publisher = {Cambridge: Cambridge University Press},
    language = {English},
    zbmath = {1446717},
    zbl = {0948.35001}
}

@misc{NiuXiang2025,
    title={{Small-time local controllability of a KdV system for all critical lengths}}, 
    author={Jingrui Niu and Shengquan Xiang},
    year={2025},
    eprint={2501.13640},
    archivePrefix={arXiv},
    primaryClass={math.AP},
    url={https://arxiv.org/abs/2501.13640}, 
}

@article{Nguyen2025,
    author = {Hoai-Minh Nguyen},
    title = {Local controllability of the {Korteweg}-de {Vries} equation with the right {Dirichlet} control},
    fjournal = {Journal of Differential Equations},
    journal = {J. Differ. Equations},
    issn = {0022-0396},
    volume = {435},
    pages = {85},
    note = {Id/No 113235},
    year = {2025},
    language = {English},
    doi = {10.1016/j.jde.2025.113235},
    zbmath = {8035404},
    url = {https://doi.org/10.1016/j.jde.2025.113235}
}

@misc{NguyenTran2025,
    title={On the local null controllability of a viscous Burgers' system in finite time}, 
    author={Hoai-Minh Nguyen and Minh-Nguyen Tran},
    year={2025},
    eprint={2507.07442},
    archivePrefix={arXiv},
    primaryClass={math.OC},
    url={https://arxiv.org/abs/2507.07442}, 
}

@article{Nirenberg1959,
	author = {Nirenberg, Louis},
	title = {On elliptic partial differential equations},
	journal = {Ann. Scuola Norm. Sup. Pisa (3)},
	volume = {13},
	year = {1959},
	pages = {115--162},
	mrclass = {35.00},
	mrnumber = {0109940},
	mrreviewer = {L. Garding}
}

@article{Rouchon2003,
	title = {Control of a quantum particle in a moving potential well},
	author = {Rouchon, Pierre},
	journal = {IFAC Proceedings Volumes},
	volume = {36},
	number = {2},
	pages = {287--290},
	year = {2003},
	publisher = {Elsevier},
    doi = {10.1016/S1474-6670(17)38906-1},
    url = {https://doi.org/10.1016/S1474-6670(17)38906-1}
}

@book{RunstSickel1996,
	author = {Runst, Thomas and Sickel, Winfried},
	title = {Sobolev spaces of fractional order, {Nemytskij} operators and nonlinear partial differential equations},
	fseries = {De Gruyter Series in Nonlinear Analysis and Applications},
	series = {De Gruyter Ser. Nonlinear Anal. Appl.},
	issn = {0941-813X},
	volume = {3},
	isbn = {3-11-015113-8},
	year = {1996},
	publisher = {Berlin: de Gruyter},
	language = {English},
	zbmath = {929821},
	zbl = {0873.35001}
}

@book{Schwartz1966,
	author = {Schwartz, Laurent},
	title = {Th{\'e}orie des distributions. {Nouvelle} {\'e}dition, enti{\`e}rement corrig{\'e}e, refondue et augment{\'e}e},
	fseries = {Publications de l'Institut de Math{\'e}matique de l'Universit{\'e} de Strasbourg},
	series = {Publ. Inst. Math. Univ. Strasbourg},
	volume = {9-10},
	year = {1966},
	publisher = {Hermann, Paris},
	language = {French},
	zbmath = {3240665},
	zbl = {0149.09501}
}

@article{Sussmann1983,
	author = {Sussmann, H\'ector},
	title = {Lie brackets and local controllability: a sufficient condition for scalar-input systems},
	journal = {SIAM J. Control Optim.},
	fjournal = {SIAM Journal on Control and Optimization},
	volume = {21},
	year = {1983},
	number = {5},
	pages = {686--713},
	issn = {0363-0129},
	mrclass = {49E15 (58F40 93B15)},
	mrnumber = {710995},
	mrreviewer = {Henry Hermes},
	doi = {10.1137/0321042},
	url = {http://dx.doi.org/10.1137/0321042}
}

\end{document}